\documentclass{amsart}

\newtheorem{theorem}{Theorem}[section]
\newtheorem{lemma}[theorem]{Lemma}

\theoremstyle{definition}
\newtheorem{definition}[theorem]{Definition}

\theoremstyle{remark}
\newtheorem{remark}[theorem]{Remark}

\numberwithin{equation}{section}

\newcommand{\abs}[1]{\lvert#1\rvert}

\begin{document}

\title{Riemann Hypothesis holds}

\author{ZhaoKun Ma}
\address{}
\curraddr{YanZhou College, ShanDong Radio and TV University, YanZhou, ShanDong 272100 China}
\email{dzy200408@126.com}
\thanks{}

\author{LanDe Ma}
\curraddr{School of Mathematical Sciences, Tongji University, Shanghai, 200092, China}
\email{malandeyouxiang@foxmail.com}
\thanks{}


\date{November 23th, 2026.}

\dedicatory{}


\begin{abstract}
In this paper the root locus theory is completed and extended. The concept of two classes of demarcation points is proposed. Two necessary and sufficient conditions for the critical points of the meromorphic functions studied in this paper are given. The results for the Dirichlet eta series are given. None of these new results had been discovered or proved before this paper. Using these new results, we prove that the non-trivial zeros of the Riemann zeta function all lie on the critical line.
\end{abstract}

\maketitle

\section*{Introduction}
For $\zeta(s)=\sum^{+\infty}_{n=1}\frac{1}{n^{s}}$, $(n\in N,\ s\in\mathbb{C}\cup\{\infty\})$, in 1859 Riemann proposed that the non-trivial zeros of $\zeta(s)$ all lie on the critical line $Re(s)=\frac{1}{2}$. This is the famous Riemann Hypothesis. There are very many specialized books on the Riemann Hypothesis. Analytic number theory, which constitutes a major branch of mathematics, consists mainly of the Riemann Hypothesis and the Goldbach conjecture. Therefore, in textbooks of analytic number theory there are also many classical results on the Riemann Hypothesis. Of course, as the most prominent open problem in mathematics, there have been many new advances and works in recent years \cite{Lee,Bergweiler1,Bergweiler2,Ki,Konstantin}, which need not be written out here. One only needs to be aware of the following fact: up to the present, the results that have been published are still very far from a complete solution of the Riemann Hypothesis.

Because historically there have been far too many papers claiming to solve the Riemann Hypothesis that were subsequently refuted, we consider that in the introduction there is no need to introduce the history and the basic knowledge of the Riemann Hypothesis and similar matters. Such history and basic knowledge can be found very easily. Therefore, in the introduction we only need to write out the reasons for people to read this paper. We need to write out the reliability of this paper's complete solution of the Riemann Hypothesis, and the main ideas of that solution.

In the textbooks of \emph{automatic control theory} \cite{Richard}, the root locus method of W.~R.~Evans has been presented for over 70 years. The root locus method in automatic control consists of six or seven rules for drawing, in the complex plane, the loci of the roots of a rational function. We first prove these rules rigorously according to purely mathematical standards. Then we discover more valuable mathematical results. Once the root locus theory has been established, we find that it can be used to solve the Riemann Hypothesis completely. The mathematical tool for a complete solution of the Riemann Hypothesis must also possess several important characteristics. That is, the function of this tool must be extremely powerful; it must be a new mathematical idea and a new theory that changes the foundations of mathematics on a large scale. This new mathematical idea and new theory must bring to mathematics unprecedented new results that mathematics very much needs. This new mathematical idea and new theory cannot solve only one historical problem; it must also be usable for solving other difficult problems. A complete solution of the Riemann Hypothesis is very important, but even more important is the discovery of new mathematical ideas and new theoretical tools for the foundations of mathematics. Even more important still is the solution of the fundamental important problems in mathematics that have not yet been solved.

A major contribution of the root locus theory is that it gives the law of the distribution of function values. The distribution of function values is a very fundamental and very much-needed core problem in function theory. Once the problem of the distribution of function values is solved, many mathematical problems can be solved. The root locus theory also gives the reason for the appearance of the critical points of a function. Since the founding of calculus, this problem has always been an important problem that needs to be solved. Therefore, with such a powerful tool as the root locus, we have completely solved the B.~Shapiro 12 conjecture. We have also completely solved the Gauss--Fourier conjecture, which has a history of two hundred years, and several other important problems.

In what follows we briefly introduce the root locus theory and the main ideas for solving the Riemann Hypothesis.

In the paper, we study the meromorphic function $W(s)=\frac{G(s)\prod^{m}_{l=1}(1-\frac{s}{z_{l}})^{\gamma_{l}}G_{lz}(s)}
{\prod^{n}_{j=1}(1-\frac{s}{p_{j}})^{\beta_{j}}G_{jp}(s)}$, in which the exponents of the zeros and of the poles are all non-negative real numbers. Using $W(s)$ we form an equation: $$KW(s)=a+ib.$$

In which, $a+ib$ is a constant in unit circle. The equation is called as the root locus equation.

In this paper, we present findings using the root locus theory for $W(s)$. In the extended complex plane $\mathbb{C} \cup \{\infty\}$, the set of points $s$ where the argument $\varphi(\sigma,t)$ of $W(s)$ is equal to $\alpha$ degrees forms a curve known as the root loci of $\alpha$ degree. The argument of $W(s)$ at each point on a specific curve has the same phase angle. The argument of $W(s)$ at a point on a certain root locus is $\alpha$ degree. The degrees of distinct curves are typically distinct. (The phase angle is the argument). $\alpha$ is the degree of the root locus. Notably, adjacent root loci generally have distinct degrees.

The root locus begins at poles of the root locus equation, and it ends at zeros of the root locus equation. As a rule, root loci with distinct unit complex numbers do not intersect. However, an exception arises when the unit complex numbers are identical; under this condition, the root loci may intersect. In root locus theory, such intersection points are termed "breakaway points." A breakaway point of the root loci is a critical point of $W(s)$. Thus, we deduce the following: a critical point of $W(s)$ must coincide with the intersection point of multiple root loci that share the same unit complex number value. If multiple root loci with the same unit complex number intersect at a single point, this point is identified as the critical point of $W(s)$. For $W(s)$, this condition is both necessary and sufficient.

The following is also an important result; it is another new result concerning the necessary and sufficient condition for the argument $\varphi(\sigma,t)$ of $W(s)$ to produce a critical point of $W(s)$.

$s_0$ is the finite zero of $W^{'}(s)$
if and only if $\frac{\partial{\varphi}}{\partial{\sigma}}(\sigma,t)\mid{_{s=s_0}}=0$
and $\frac{\partial{\varphi}}{\partial{t}}(\sigma,t)\mid{_{s=s_0}}=0$ hold simultaneously.

These two results on the necessary and sufficient conditions are used repeatedly in solving the Riemann Hypothesis.

We prove: at finite positions of the complex plane, if the root loci originating from several distinct poles of $W(s)$ come into contact, then the root loci of the same degree starting from several poles of $W(s)$ can only intersect, and they produce the breakaway points of these root loci. We prove that the finite breakaway points of $W(s)$ are precisely the finite critical points of $W(s)$. Using the Dirichlet eta series $\eta(s)$ we form a root locus equation: $$K\eta(s)=a+ib.$$

Using the necessary and sufficient condition for the argument $\varphi(\sigma,t)$ of $W(s)$ to produce critical points of $W(s)$, for the infinite series $FF(s)=\sum^{+\infty}_{n=1}\frac{a_{n}(n)}{(b_{n}(n))^{s}}$, we prove: at finite positions of the complex plane, the critical point of $FF(s)$ does not exist. The Dirichlet eta series $\eta(s)=\sum^{+\infty}_{n=1}\frac{(-1)^{n+1}}{n^{s}}$ is a special series of $FF(s)$. Therefore, at finite positions of the right half complex plane, the critical point of the Dirichlet eta series does not exist. Therefore, at finite positions of the right half complex plane, the breakaway point of the root locus of the Dirichlet eta series $\eta(s)$ does not exist. Hence, if non-trivial zeros exist outside the critical line, then the root loci of the non-trivial zeros outside the critical line and the root loci of the other zeros of the Dirichlet eta series cannot come into contact at finite positions. That is, these root loci cannot be adjacent at finite positions, nor can they produce breakaway points at finite positions.

We have also proved the result: if non-trivial zeros exist outside the critical line, the root loci of the non-trivial zeros in the left half of the critical strip cannot come into contact with the root loci of the other zeros of the Dirichlet eta series at all positions at infinity. That is, these root loci cannot be adjacent at any position at infinity, nor can they produce breakaway points at any position at infinity.

At finite positions, the root loci of the non-trivial zeros outside the critical line cannot come into contact with the root loci of the other zeros of the Dirichlet eta series; that is, these root loci cannot be adjacent at finite positions, and they cannot produce finite breakaway points. At all points at infinity, the root loci of the non-trivial zeros in the left half of the critical strip cannot come into contact with the root loci of the other zeros of the Dirichlet eta series; that is, these root loci cannot be adjacent at any point at infinity, and they cannot produce breakaway points at infinity. But if non-trivial zeros exist outside the critical line, these non-trivial zeros in the left half of the critical strip must emit root loci. The root loci of these non-trivial zeros in the left half of the critical strip must come into contact with the root loci of the other zeros of the Dirichlet eta series. The result of the root loci coming into contact is that the root loci are adjacent or intersect, and intersecting root loci produce breakaway points. Either the root loci of these non-trivial zeros in the left half of the critical strip must be adjacent to the root loci of the other zeros of the Dirichlet eta series, or the root loci of these non-trivial zeros in the left half of the critical strip must intersect the root loci of the other zeros of the Dirichlet eta series and produce breakaway points. One of the two cases must occur. Thus we obtain two classes of mutually contradictory results.

Therefore, this proves that the hypothesis that non-trivial zeros exist outside the critical line cannot hold. Hence non-trivial zeros cannot exist outside the critical line.

\section{Root Locus Of the meromorphic function $W(s)$ and Their Properties}
In  textbooks of \emph{automatic control theory}, the factor at the left side of the root locus equation is the rational fraction function of the constant coefficient. The root locus equation is only concerning two degree of 0 and 180 degree numbers which obtains the real number values\cite{Richard}. So, the root locus equations and the results of the root locus in \emph{automatic control theory} are all very special and limitted. The proofs of results in \emph{automatic control theory} are not comprehensive and accurate.

We need to study the more general root locus equations of the the meromorphic function $W(s)$. Let
$K=|\frac{\prod^{n}_{j=1}(1-\frac{s}{p_{j}})^{\beta_{j}}G_{jp}(s)}
{G(s)\prod^{m}_{l=1}(1-\frac{s}{z_{l}})^{\gamma_{l}}G_{lz}(s)}|$. The $K$ is the reciprocal of the modulus of  meromorphic function $W(s)=u(\sigma,t)+iv(\sigma,t)$. After $K$ and meromorphic function
$W(s)$ are multiplied, the product result is the unit complex value of meromorphic function $W(s)$.

\begin{definition}
In the extended complex plane $\mathbb{C}\cup\{\infty\}$, the equation (1.1) is called as \emph{the root locus equation}.
\begin{equation}
KG(s)\frac{\prod^{m}_{l=1}(1-\frac{s}{z_{l}})^{\gamma_{l}}G_{lz}(s)}
{\prod^{n}_{j=1}(1-\frac{s}{p_{j}})^{\beta_{j}}G_{jp}(s)}=a+ib
\end{equation}
\end{definition}
In which, $a+ib$ is the unit complex number value of meromorphic function $W(s)$. $\alpha=2q\pi+\arg(\frac{b}{a})$.
The zeros $z_{l}$ and poles $p_{j}$ are points in $\mathbb{C}\cup\{\infty\}$, and may be no conjugate.

\begin{lemma}
After the coincident and finite zeros and poles of Eq(1.1) are cancelled, for any finite zero of Eq(1.1), its $K$ value is $K=+\infty$; for any finite pole of Eq(1.1), its $K$ value is $K=0$.
\end{lemma}
\begin{proof}
The Eq(1.1) can be transformed to the following characteristic equation.

$KG(s)\prod^{m}_{l=1}(1-\frac{s}{z_{l}})^{\gamma_{l}}G_{lz}(s)-
(a+ib)\prod^{n}_{j=1}(1-\frac{s}{p_{j}})^{\beta_{j}}G_{jp}(s)=0$. Obviously, if $K=0$, all roots of the pole factor $\prod^{n}_{j=1}(1-\frac{s}{p_{j}})^{\beta_{j}}G_{jp}(s)$ are the roots of characteristic equation. Conversely, all the roots of characteristic equation if $K=0$ are all of the roots of the pole factor. So, at poles of Eq(1.1), $K=0$.

The Eq(1.1) can be transformed to its another characteristic equation.
$G(s)\prod^{m}_{l=1}$$(1-\frac{s}{z_{l}})^{\gamma_{l}}G_{lz}(s)-
\frac{a+ib}{K}\prod^{n}_{j=1}(1-\frac{s}{p_{j}})^{\beta_{j}}G_{jp}(s)=0$. If $K=+\infty$, all roots of the zero factor $\prod^{m}_{l=1}(1-\frac{s}{z_{l}})^{\gamma_{l}}G_{lz}(s)$ are the roots of the characteristic equation. Conversely, all roots of the characteristic equation when $K=+\infty$ are the roots of the zero factor. So, at zeros of Eq(1.1), $K=+\infty$.
\end{proof}

According to the expression of $K$ of the root locus equation (1.1) and according to Lemma 1.2, it is obvious that the $K$ values are continuous concerning the complex variable $s$ in $\mathbb{C}\cup\{\infty\}$.
So, we can prove Theorem 1.3 simply.

\begin{theorem}
Let $s\in\mathbb{C}\cup\{\infty\}$, the $K$ value of $s\in \mathbb{C}\cup\{\infty\}$ takes all of the non-negative real number value from 0 to the $+\infty$.
\end{theorem}

In complex analysis, the phase angle is namely argument\cite{Stein}. In \emph{automatic control theory}, it is called as the phase angle.
Here, we need to use name of the phase angle. In this paper, $K$ is called as the gain.

Let $\Delta=\{s=\sigma+it\in\mathbb{C}\cup\{\infty\}$, $s$ is not zero or pole of $W(s)$ and $Eq(1.1)\}$.

\begin{lemma}
For any point $s\in\Delta$, the phase angle of meromorphic function $W(s)$ can be written as the $\varphi(\sigma,t)$ function:
$\varphi(\sigma,t)=\arg(\frac{G_{y}(\sigma,t)}{G_{x}(\sigma,t)})+
\sum^{m}_{l=1}(\gamma_{l}\arg(\frac{t-t_{l}}{\sigma-\sigma_{l}})-\gamma_{l}\arg(\frac{t_{l}}{\sigma_{l}})+
\arg(\frac{G_{zy}(\sigma,t)}{G_{zx}(\sigma,t)}))
-\sum^{n}_{j=1}(\beta_{j}\arg(\frac{t-t_{j}}{\sigma-\sigma_{j}})-\beta_{j}\arg(\frac{t_{j}}{\sigma_{j}})+
\arg(\frac{G_{py}(\sigma,t)}{G_{px}(\sigma,t)}))$.
\end{lemma}
Let $s=(\sigma,t)\in\Delta$, then the phase angle condition equation of Eq(1.1) on $s$ is:
\begin{equation}
\varphi(\sigma,t)=2q\pi+arg{(b/a)}
\end{equation}
In which, $q$ is an integer number.
\begin{lemma}
Let Eq(1.1) be the root locus equation of $2q\pi+\alpha$ degree, for any point $s=(\sigma,t)\in\Delta$, if the point $s$ satisfies the phase angle condition equation (1.2) of $2q\pi+\alpha$ degree, then it must be on the root locus of $2q\pi+\alpha$ degree of Eq(1.1).
\end{lemma}
\begin{proof}
Assuming that a point $s_1=(\sigma_1,t_1)\in\Delta$ satisfies the phase angle condition
equation (1.2), $\varphi(\sigma_1,t_1)=2q\pi+arg{(b_1/a_1)}$. The phase angle of $W(s)$ is $\varphi(\sigma,t)$, this phase
angle expression is same as the expression of the left side of
Eq(1.2).

Because the point $s_{1}$ satisfies the phase angle condition
equation (1.2), when the point $s_{1}$ is substituted into the
function $W(s)$, $W(s)$ obtains a complex value, according to the
phase angle expression of Eq(1.2) of the point $s_{1}$ and the
condition which a point $s_{1}$ satisfies the phase angle condition
equation (1.2), this complex value can be written
as: $K_1^{*}(a_{1}+ib_{1})$, and  $\alpha_1=\arg{(b_1/a_1)}$,
$K_1^{*}$ is the modulus of function $W(s)$ of point $s_{1}$.
$K_1^{*}$ is a non-zero positive real value. $\alpha_1$ is the phase
angle of the function $W(s)$ of the point $s_{1}$, and it is the value
of the right side of Eq(1.2).

If we bring the point $s_1$ into the gain expression $|\frac{\prod^{n}_{j=1}(1-\frac{s}{p_{j}})^{\beta_{j}}G_{jp}(s)}
{G(s)\prod^{m}_{l=1}(1-\frac{s}{z_{l}})^{\gamma_{l}}G_{lz}(s)}|$, according to the definition of the gain, we can obtain
an unique gain $K_1$, $K_1=|\frac{\prod^{n}_{j=1}(1-\frac{s_1}{p_{j}})^{\beta_{j}}G_{jp}(s_1)}
{G(s_1)\prod^{m}_{l=1}(1-\frac{s_1}{z_{l}})^{\gamma_{l}}G_{lz}(s_1)}|$, the gain $K_1$ is a
reciprocal of the modulus $K_1^{*}$ of $W(s)$ of the point $s_1$.
Further obtain $K_1^{*}$. $K_1^{*}=1/K_1$. For that the gain $K_1$
multiplied the expression $G(s_1)\prod^{m}_{l=1}(1-\frac{s_1}{z_{l}})^{\gamma_{l}}G_{lz}(s_1)/\prod^{n}_{j=1}(1-\frac{s_1}{p_{j}})^{\beta_{j}}G_{jp}(s_1)$, we have,
$K_1G(s_1)\prod^{m}_{l=1}(1-\frac{s_1}{z_{l}})^{\gamma_{l}}G_{lz}(s_1)/\prod^{n}_{j=1}(1-\frac{s_1}{p_{j}})^{\beta_{j}}G_{jp}(s_1)=a_1+ib_1$. This equation is a concrete situation which a point $s_1$ satisfies Eq(1.1). The previous results prove
that the point $s_1$ is a root of Eq(1.1).

So, the point $s_1$ satisfies Eq(1.1). Therefore, it is proved
that the point $s_1$  is on the root locus of Eq(1.1), and which its gain is $K_1$ and its degree is
$2q\pi+\alpha_1$. Hence, if the point $s_1$ satisfies Eq(1.2), then
the point $s_1$ is on the root locus of Eq(1.1), and which its
gain is $K_1$ and its degree is $2q\pi+\alpha_1$.
\end{proof}

\begin{lemma}
Let Eq(1.1) be the root locus equation of $2q\pi+\alpha$ degree, for any point $s=(\sigma,t)\in\Delta$, if the point $s$ is on the root locus of $2q\pi+\alpha$  degree of Eq(1.1), then it must satisfy the phase angle condition equation
of $2q\pi+\alpha$ degree of Eq(1.2).
\end{lemma}
\begin{proof}
Assume that the point $s_2=(\sigma_2,t_2)\in\Delta$ is an arbitrary point on the root locus of Eq(1.1). The points which satisfy Eq(1.1) are
all on the root locus of Eq(1.1). So, the point $s_2$
surely satisfies Eq(1.1). When $s_2$ is substituted into
Eq(1.1), $K_2G(s_2)\prod^{m}_{l=1}(1-\frac{s_2}{z_{l}})^{\gamma_{l}}G_{lz}(s_2)/\prod^{n}_{j=1}(1-\frac{s_2}{p_{j}})^{\beta_{j}}G_{jp}(s_2)=a_2+ib_2$.
According to Theorem 1.3, here, the gain $K_2$ is a positive real
number, its phase angle of $K_2$ is $2q_1\pi$, the phase angle of
the right side of the equation in this paragraph on the point $s_2$
is $2q_2\pi+\arg{(b_2/a_2)}=2q_2\pi+\alpha_2$.

The factor of the left side of the equation in last paragraph on the
point $s_2$ can be looked as two factors. One is $K_2$, another is
$G(s_2)\prod^{m}_{l=1}(1-\frac{s_2}{z_{l}})^{\gamma_{l}}G_{lz}(s_2)/\prod^{n}_{j=1}(1-\frac{s_2}{p_{j}})^{\beta_{j}}G_{jp}(s_2)$.
The phase angle of the factor of the left side of the equation in
last paragraph on the point $s_2$ is equal to the summation of two
phase angles of two factors of $K_2$  and $G(s_2)\prod^{m}_{l=1}(1-\frac{s_2}{z_{l}})^{\gamma_{l}}G_{lz}(s_2)/\prod^{n}_{j=1}(1-\frac{s_2}{p_{j}})^{\beta_{j}}G_{jp}(s_2)$. The phase angle of the
factor of the left side of the equation in last paragraph on the
point $s_2$ is equal to the phase angle of the right side of the
equation in the last paragraph on the point $s_2$. The phase angle
$2q_2\pi+\alpha_2$ subtracts the phase angle $2q_1\pi$  is equal to
$2q_2\pi+\alpha_2-2q_1\pi=2q\pi+\alpha_2$.

The difference of the phase angle of the factor $a_2+ib_2 $  and the
factor $K_2$ is: $2q\pi+\alpha_2$. So , the phase angle of
expression $G(s_2)\prod^{m}_{l=1}(1-\frac{s_2}{z_{l}})^{\gamma_{l}}G_{lz}(s_2)/\prod^{n}_{j=1}(1-\frac{s_2}{p_{j}})^{\beta_{j}}G_{jp}(s_2)$ is $2q\pi+\alpha_2=2q\pi+\arg{(b_2/a_2)}$. We can obtain:
the phase angle of $W(s)$ is $\varphi(\sigma,t)=\arg(\frac{G_{y}(\sigma,t)}{G_{x}(\sigma,t)})+
\sum^{m}_{l=1}(\gamma_{l}\arg(\frac{t-t_{l}}{\sigma-\sigma_{l}})-\gamma_{l}\arg(\frac{t_{l}}{\sigma_{l}})+
\arg(\frac{G_{zy}(\sigma,t)}{G_{zx}(\sigma,t)}))
-\sum^{n}_{j=1}(\beta_{j}\arg(\frac{t-t_{j}}{\sigma-\sigma_{j}})-\beta_{j}\arg(\frac{t_{j}}{\sigma_{j}})+
\arg(\frac{G_{py}(\sigma,t)}{G_{px}(\sigma,t)}))$. So,
according to the previous proof, we can obtain: $ \varphi(\sigma_2,t_2)=2q\pi+arg{(b_2/a_2)}$,  the point $s_2$ satisfies the
phase angle condition equation (1.2) . Hence, if the point $s_2$ is
on the root locus of Eq(1.1), which its gain is  $K_2$ and its
degree is $2q\pi+\alpha_2$, and satisfies Eq(1.1), then point
$s_2$ satisfies Eq(1.2).
\end{proof}

Lemma 1.5 gives the sufficient condition result of Theorem 1.7. Lemma
1.6 gives the necessary condition result of Theorem 1.7. So, sum up Lemma
1.5 and Lemma 1.6, we can obtain the following theorem.

\begin{theorem}
Let Eq(1.1) be $2q\pi+\alpha$ degree and $s=(\sigma,t)\in\Delta$ be an arbitrary
point. A necessary and sufficient condition
for that the point $s$  is on the root locus of $2q\pi+\alpha$
degree of Eq(1.1) is that point $s$ whether or not satisfies the
phase angle condition equation (1.2).
\end{theorem}

\begin{definition}
Let $\Xi$ be a point set in $\mathbb{C}\cup\{\infty\}$, that are consist of all of points on the path of roots of Eq(1.1) traced out in $\mathbb{C}\cup\{\infty\}$
as $2q\pi+\alpha=2q\pi+\arg(\frac{b}{a})$ is a constant. That path of Eq(1.1) in $\mathbb{C}\cup\{\infty\}$ is called as \emph{the root locus of Eq(1.1)}.
Namely, the set $\Xi$ is \emph{the root locus of the $2q\pi+\alpha$ degree}.
\end{definition}

The subset of $\Xi$ is part or entire root locus except  zeros and poles.
And the empty set $\emptyset$ is the case when zero and pole are coincident, the case has no root locus.
Let $\tau$ be the collection of subsets of $\Xi$, the ordered pair $(\Xi,\tau)$ satisfying the following properties:
$\emptyset$ and $\Xi$ itself are both open in $\tau$, the intersection of any two open sets is
open in $\tau$, and the union of every collection of open sets is open in $\tau$.
So, the collection $\tau$ is a topology on $\Xi$, and the ordered pair $(\Xi,\tau)$ is a topological space.

The meromorphic function $W(s)$ is the factor at Eq(1.1)'s left side. So, the $\varphi(\sigma,t)$ function which is in Lemma 1.4 can be used to compute the degree of the root locus of Eq(1.1). The factor at the right side of Eq(1.1) is the computing value.

\begin{definition}
When the phase angle of meromorphic function $W(s)$ at the left side of the root locus equation (1.1) is $2q\pi+\alpha$ degree, we call the root locus equation (1.1) as \emph{the $2q\pi+\alpha$ degree root locus equation}.
\end{definition}

In $\mathbb{C}\cup\{\infty\}$, the degree number of the phase angle of meromorphic function $W(s)$ at the left side of the root locus equation (1.1) is the degree number of the root locus of the root locus equation (1.1). So, we can give a definition of the degree of the root locus of Eq(1.1).

\begin{definition}
When the phase angle of meromorphic function $W(s)$ at the left side of the root locus equation (1.1) is $2q\pi+\alpha$ degree, we call the root locus of the root locus equation (1.1) as \emph{$2q\pi+\alpha$ degree root locus}.
\end{definition}

\begin{lemma}
For all finite zeros of Eq(1.1), they are on the root locus of all degree numbers of Eq(1.1) from $2q\pi$ degree to $2q\pi+2\gamma_{l}\pi$ degree.
\end{lemma}
\begin{proof}
Assuming that the point $zp$ is an arbitrary finite zeros of Eq(1.1). When the point $zp$ is substituted into meromorphic function $W(s)$ at the left side of Eq(1.1),
we have: $G(zp)\frac{\prod^{m}_{l=1}(1-\frac{zp}{z_{l}})^{\gamma_{l}}G_{lz}(zp)}
{\prod^{n}_{j=1}(1-\frac{zp}{p_{j}})^{\beta_{j}}G_{jp}(zp)}=0$.
This equation can also be expressed as: $G(zp)\frac{\prod^{m}_{l=1}(1-\frac{zp}{z_{l}})^{\gamma_{l}}G_{lz}(zp)}
{\prod^{n}_{j=1}(1-\frac{zp}{p_{j}})^{\beta_{j}}G_{jp}(zp)}=K_{zp}e^{i\theta_{zp}}=0$.
In this equation, $K_{zp}$ is the modulus of the function $W(s)$ and $\theta_{zp}$ is the phase angle of the function $W(s)$.
So, $K_{zp}=0$, and $\theta_{zp}=2q\pi+\alpha$, $K_{zp}e^{i(2q\pi+\alpha)}=0*(cos(\theta_{zp})+isin(\theta_{zp}))=0$ is true,
and $\theta_{zp}$ is an arbitrary degree number from $2q\pi$ degree to $2q\pi+2\gamma_{l}\pi$ degree.

For the point $zp$ that lets the left side of Eq(1.1) obtain 0, no matter what phase angle it is, since its modulus is 0. $cos(\theta_{zp})+isin(\theta_{zp})$ is the non-zero and non-infinity, the modulus of the factor $cos(\theta_{zp})+isin(\theta_{zp})$ is 1.
 $\theta_{zp}$ represents an arbitrary degree number, which shows: no matter what value $\theta_{zp}$ is,
there is $K_{zp}e^{i(2q\pi+\alpha)}=0*(\cos(\theta_{zp})+i\sin(\theta_{zp}))=0$.

This proves: For the finite zero $zp$ of Eq(1.1), when it is substituted into the meromorphic function $W(s)$, for the phase angle of the arbitrary $2q\pi+\alpha$ degree number of the meromorphic function $W(s)$, its values are all equal to 0. According to Definition 1.10, the phase angle of the meromorphic function $W(s)$ is namely the phase angle of the root locus of Eq(1.1). Because $2q\pi+\alpha$ is an arbitrary degree number, when it obtains all of degree numbers, this shows the finite zeros of Eq(1.1) are simultaneously on the root locus of all of degree numbers of Eq(1.1) from $2q\pi$ degree to $2q\pi+2\gamma_{l}\pi$ degree.
\end{proof}

\begin{lemma}
For all finite poles of Eq(1.1), they are on the root locus of all degree numbers of Eq(1.1) from $2q\pi$ degree to $2q\pi+2\beta_{j}\pi$ degree.
\end{lemma}
\begin{proof}
Assuming that the point $pz$ is an arbitrary finite pole of Eq(1.1). When the point $pz$ is substituted into the meromorphic function $W(s)$, we can obtain a value $Gpz=G(pz)\frac{\prod^{m}_{l=1}(1-\frac{pz}{z_{l}})^{\gamma_{l}}G_{lz}(pz)}
{\prod^{n}_{j=1}(1-\frac{pz}{p_{j}})^{\beta_{j}}G_{jp}(pz)}=\infty$.

In complex analysis, 0 can be written as: $0*e^{i\theta}=0*(\cos\theta+i\sin\theta)$.
The points of non-zero and non-infinity finite values can also be written as: $k*e^{i\theta}=k*(\cos\theta+i\sin\theta)$.
The infinity can also be written as: $(+\infty)*e^{i\theta}=(+\infty)*(\cos\theta+i\sin\theta)$. In which, $(\cos\theta+i\sin\theta)$ is a non-zero and non-infinity, and its modulus is 1 of the unit complex number value.
For $(+\infty)*e^{i\theta}=(+\infty)*(\cos\theta+i\sin\theta)$, no matter what value $\theta=2q\pi+\alpha$ obtains on the unit circle in $\mathbb{C}\cup\{\infty\}$, the values of this expression all obtain the infinity.

The above equation can also be expressed as: $Gpz=G(pz)\frac{\prod^{m}_{l=1}(1-\frac{pz}{z_{l}})^{\gamma_{l}}G_{lz}(pz)}
{\prod^{n}_{j=1}(1-\frac{pz}{p_{j}})^{\beta_{j}}G_{jp}(pz)}=K_{pz}e^{i\theta_{pz}}=\infty$.
In this equation, $K_{pz}$ is the modulus of the meromorphic function $W(s)$ and $\theta_{pz}$ is the phase angle of the meromorphic function $W(s)$.
So, $K_{pz}=+\infty$, and $\theta_{pz}=2q\pi+\alpha$, $K_{pz}*e^{i\theta}=(+\infty)*(\cos\theta+i\sin\theta)=\infty$ is true,
and $\theta_{pz}=2q\pi+\alpha$ is an arbitrary degree number from $2q\pi$ degree to $2q\pi+2\beta_{j}\pi$ degree.

For the finite pole $pz$ of Eq(1.1), when it is substituted into the meromorphic function $W(s)$, for the phase angle of the arbitrary $2q\pi+\alpha$ degree number of the meromorphic function $W(s)$, its values are all equal to infinty. According to Definition 1.10, the phase angle of the meromorphic function $W(s)$ is namely the phase angle of the root locus of Eq(1.1).
Because $2q\pi+\alpha$ is an arbitrary degree number from $2q\pi$ degree to $2q\pi+2\beta_{j}\pi$ degree, when it obtains all of degree numbers, this shows the finite poles of Eq(1.1) are simultaneously on the root locus of all of degree numbers of Eq(1.1) from $2q\pi$ degree to $2q\pi+2\beta_{j}\pi$ degree.
\end{proof}

\begin{lemma}
All of the root locus of arbitrary $2q\pi+\alpha$ degree number of Eq(1.1) are originated from poles of Eq(1.1),
and are finally received at zeros of Eq(1.1).
\end{lemma}
\begin{proof}
The arbitrary $2q\pi+\alpha$ degree number root locus of Eq(1.1) are the curves in $\mathbb{C}\cup\{\infty\}$. So, they all need to have their own origination points and receiving points. According to the relationship between the points in $\mathbb{C}\cup\{\infty\}$ and the root locus of Eq(1.1), the points in $\mathbb{C}\cup\{\infty\}$ can be divided into four types, one is the poles of Eq(1.1), one is the zeros of Eq(1.1), the other is the general finite points on the $2q\pi+\alpha$ degree root locus of Eq(1.1), and the last type is the infinity points in $\mathbb{C}\cup\{\infty\}$.

Except the finite and infinite zeros of Eq(1.1), and except the finite and infinite poles of Eq(1.1), a general finite point in $\mathbb{C}\cup\{\infty\}$ is on a root loci of a certain degree number of Eq(1.1). And a general finite point in $\mathbb{C}\cup\{\infty\}$ only has the phase angles of only one degree number and non-zero finite gain value. It can not be the originating point and receiving point of the root locus of Eq(1.1).

When the infinite point in $\mathbb{C}\cup\{\infty\}$ let the meromorphic function $W(s)$  obtain a constant $A$.
The gain values of the infinite point in $\mathbb{C}\cup\{\infty\}$ are the constant $\abs{A}$.
The phase angles of the meromorphic function $W(s)$  are all equal to $2q\pi+arg{(A)}$ degree.
So, these infinite points in $\mathbb{C}\cup\{\infty\}$ are also the ordinary points of the finite values with the gain $\abs{A}$ on the $2q\pi+\arg(a)$ degree root locus. They can not originate or receive the root locus.

When the infinite points in $\mathbb{C}\cup\{\infty\}$ are the infinite zeros or poles of Eq(1.1). On the infinite zero or pole, the gain value is $K=+\infty$ or $K=0$ respectively.

The finite poles and zeros of Eq(1.1) are on the root locus of all of degree numbers of Eq(1.1), and the infinite number of the different degree numbers root locus are at the same one point, so, they satisfy the condition that the root locus can be originated or received. On the finite and infinite poles of Eq(1.1), there is $K=0$. On the finite and infinite zeros of Eq(1.1), there is $K=+\infty$. Thus, we can let the finite and infinite poles of Eq(1.1) as the origination points of the root locus. The finite and infinite zero points of Eq(1.1) are the receiving points of the root locus.
\end{proof}

\begin{definition}
\emph{The origination points} of the root locus of the arbitrary $2q\pi+\alpha$ degree of Eq(1.1) are the points on the $2q\pi+\alpha$ degree root locus that their corresponding gain values equal to zero in $\mathbb{C}\cup\{\infty\}$, $K=0$.
\end{definition}

\begin{definition}
\emph{The receiving points} of the root locus of the arbitrary $2q\pi+\alpha$ degree of Eq(1.1) are the points on the $2q\pi+\alpha$ degree root locus that their corresponding gain values equal to infinity in $\mathbb{C}\cup\{\infty\}$, $K=+\infty$.
\end{definition}

\begin{lemma}
(1). For any infinite point of $\mathbb{C}\cup\{\infty\}$, if it satisfies

$K=\lim\limits_{\abs{s}\to +\infty}\frac{\prod^{n}_{j=1}|1-\frac{s}{p_{j}}|^{\beta_{j}}|G_{jp}(s)|}{{|G(s)|\prod^{m}_{l=1}|1-\frac{s}{z_{l}}|^{\gamma_{l}}|G_{lz}(s)|}}=0$, then it is the infinite pole of the root locus of Eq(1.1). These infinite points must be the origination points of  $2q\pi+\alpha$ degrees root locus of Eq(1.1) in $\mathbb{C}\cup\{\infty\}$.

(2). For any infinite point of $\mathbb{C}\cup\{\infty\}$, if it satisfies

$K=\lim\limits_{\abs{s}\to +\infty}\frac{\prod^{n}_{j=1}|1-\frac{s}{p_{j}}|^{\beta_{j}}|G_{jp}(s)|}{{|G(s)|\prod^{m}_{l=1}|1-\frac{s}{z_{l}}|^{\gamma_{l}}|G_{lz}(s)|}}=A$, in which $A\neq 0$ is a finite value, then it is a general point that its gain is a finite value. These infinite points are the general points on the root locus of Eq(1.1) in $\mathbb{C}\cup\{\infty\}$.

(3). For any infinite point of $\mathbb{C}\cup\{\infty\}$, if it satisfies

$K=\lim\limits_{\abs{s}\to +\infty}\frac{\prod^{n}_{j=1}|1-\frac{s}{p_{j}}|^{\beta_{j}}|G_{jp}(s)|}{{|G(s)|\prod^{m}_{l=1}|1-\frac{s}{z_{l}}|^{\gamma_{l}}|G_{lz}(s)|}}=+\infty$, then the infinite point is the infinite zero of the root locus of Eq(1.1). These infinite points must be the receiving points of  $2q\pi+\alpha$ degrees root locus of Eq(1.1) in $\mathbb{C}\cup\{\infty\}$.

\end{lemma}
\begin{proof}
(i). For any infinite points of $\mathbb{C}\cup\{\infty\}$, if they satisfy

$K=\lim\limits_{\abs{s}\to +\infty}\frac{\prod^{n}_{j=1}|1-\frac{s}{p_{j}}|^{\beta_{j}}|G_{jp}(s)|}{{|G(s)|\prod^{m}_{l=1}|1-\frac{s}{z_{l}}|^{\gamma_{l}}|G_{lz}(s)|}}=0$, according to the definition of the origination points of the arbitrary degree root locus of Eq(1.1), the infinity points on the root locus of the arbitrary $2q\pi+\alpha$ degree of Eq(1.1) are also origination points of the arbitrary $2q\pi+\alpha$ degree root locus of Eq(1.1).

(ii). For any infinite points of $\mathbb{C}\cup\{\infty\}$, if they satisfy

$K=\lim\limits_{\abs{s}\to +\infty}\frac{\prod^{n}_{j=1}|1-\frac{s}{p_{j}}|^{\beta_{j}}|G_{jp}(s)|}{{|G(s)|\prod^{m}_{l=1}|1-\frac{s}{z_{l}}|^{\gamma_{l}}|G_{lz}(s)|}}=A$, in which $A\neq 0$ is a finite value. So, all of the infinity points in $\mathbb{C}\cup\{\infty\}$ that their gains are the finite values are all general points. The infinity points in $\mathbb{C}\cup\{\infty\}$ are the general points that their gains are the finite values $A$.

(iii). For any infinite points of $\mathbb{C}\cup\{\infty\}$, if they satisfy

$K=\lim\limits_{\abs{s}\to +\infty}\frac{\prod^{n}_{j=1}|1-\frac{s}{p_{j}}|^{\beta_{j}}|G_{jp}(s)|}{{|G(s)|\prod^{m}_{l=1}|1-\frac{s}{z_{l}}|^{\gamma_{l}}|G_{lz}(s)|}}=+\infty$. According to the definition of the receiving points of the arbitrary degree root locus of Eq(1.1),  the infinity points on the root locus of the arbitrary $2q\pi+\alpha$ degree of Eq(1.1) are also receiving points of the arbitrary $2q\pi+\alpha$ degree root locus of Eq(1.1).
\end{proof}
Sum up Lemma 1.13 and Lemma 1.16, we can obtain Theorem 1.17.
\begin{theorem}
In $\mathbb{C}\cup\{\infty\}$, all root locus of the arbitrary $2q\pi+\alpha$ degree number of Eq(1.1) are originated from their finite poles or the infinite poles. And these root locus are received by their finite zeros or the infinite zeros.
\end{theorem}

\begin{definition}
In $\mathbb{C}\cup\{\infty\}$. If an angle $\phi$ is between the positive
direction of the real axis and the tangent line of the root loci of
arbitrary $2q\pi+\alpha$ degree that originates from a finite pole.
Then, the angle $\phi$ is called as \emph{angle of origination} of the
root loci of $2q\pi+\alpha$ degree of Eq(1.1).
\end{definition}

\begin{definition}
In $\mathbb{C}\cup\{\infty\}$. If an angle $\varphi$ is between the
positive direction of the real axis and the tangent line of the root
loci of arbitrary $2q\pi+\alpha$ degree that receives at a finite
zero. Then the angle $\varphi$ is called as \emph{angle of receiving} of
the root loci of $2q\pi+\alpha$ degree of Eq(1.1).
\end{definition}

In which, the \emph{angle of origination} and \emph{angle of receiving} are
both in positive counterclockwise orientation.

\begin{theorem}
In $\mathbb{C}\cup\{\infty\}$, let $P_k$  be a $\beta_k$ repeated finite poles of Eq(1.1). The angle of origination of the root locus of arbitrary $2q\pi+\alpha$ degree of Eq(1.1) which is originated from the finite pole $P_k$ is:

$\theta_{P_k}=\frac{1}{\beta_k}(2q\pi-\alpha+arg(G(P_k))+\beta_{k}arg(-p_k)-arg(G_{kp}(P_k))+\sum\limits_{l=1}^{m}
(\gamma_{l}arg{(P_k-z_{l})}-\gamma_{l}arg(-z_{l})+arg(G_{lz}(P_k)))-\sum\limits_{j=1,j \neq k}^{n}
(\beta_{j}arg{(P_k-p_{j})}-\beta_{j}arg(-p_j)+arg(G_{jp}(P_k)))$.

In which, $\beta_k$ is a real number.
\end{theorem}
\begin{proof}
Selecting a finite point $s_{1}$ on the root locus of $2q\pi+\alpha$
degree. The finite point $s_{1}$ should be infinitely approaching to $\beta_k$
repeated finite pole $P_k$ that needs to compute its angle of
origination. Because the finite point  $s_{1}$ is infinitely approaching to the
finite pole $P_k$. The phase angle $\beta_k arg{(s_1-p_j)}$ of vectors between the finite point $s_{1}$ and the finite poles $p_j$, $j\neq k$. (Namely, except the finite pole $P_k$.) all can be substituted by angles  $\beta_{k}arg{(P_k-p_{j})}$, $j\neq k$. The phase angle  $\gamma_{k}arg{(s_1-z_{l})}$ of vectors between the finite point $s_{1}$ and all of the finite zeros $z_l$ of Eq(1.1) all can be substituted by angles $\gamma_{k}arg{(P_k-z_{l})}$.

The finite point $s_{1}$ is on the root locus of $2q\pi+\alpha$ degree.
Because this finite point is infinitely approaching to the finite pole
$P_k$.  So, the vector between the finite point $s_{1}$ and the finite pole $P_k$ is infinitely approaching to the tangent line of the root loci of arbitrary $2q\pi+\alpha$ degree that is originated from the finite pole $P_k$. According to the Definition 1.18, we can let the phase angle of vectors between the finite point $s_{1}$ and the finite pole $P_k$ be infinitely approaching to angle of origination $\theta_{P_k}$.

The finite point $s_{1}$ must satisfy the phase angle condition equation,
$arg(G(s_1))+\sum\limits_{l=1}^{m}
(\gamma_{l}arg{(s_1-z_{l})}-\gamma_{l}arg(-z_{l})+arg(G_{lz}(s_1)))-\sum\limits_{j=1}^{n}
(\beta_{j}arg{(s_1-p_{j})}-\beta_{j}arg(-p_j)+arg(G_{jp}(s_1)))=2q\pi+\alpha$.
 So, when the finite pole $P_k$
substitutes the finite point  $s_{1}$, we have $arg(G(P_k))+\sum\limits_{l=1}^{m}
(\gamma_{l}arg{(P_k-z_{l})}-\gamma_{l}arg(-z_{l})+arg(G_{lz}(P_k)))-\sum\limits_{j=1,j \neq k}^{n}
(\beta_{j}arg{(P_k-p_{j})}-\beta_{j}arg(-p_j)+arg(G_{jp}(P_k)))-\beta_k\theta_{P_k}+\beta_karg(-p_k)-arg(G_{kp}(P_k))=2q\pi+\alpha$.

Then, we have: $\beta_k\theta_{P_k}=arg(G(P_k))+\sum\limits_{l=1}^{m}
(\gamma_{l}arg{(P_k-z_{l})}-\gamma_{l}arg(-z_{l})+arg(G_{lz}(P_k)))-\sum\limits_{j=1,j \neq k}^{n}
(\beta_{j}arg{(P_k-p_{j})}-\beta_{j}arg(-p_j)+arg(G_{jp}(P_k)))+\beta_karg(-p_k)-arg(G_{kp}(P_k))-2q\pi-\alpha$.

The $2q\pi$ and $-2q\pi$ are equivalent, so, $2q\pi$ is substituted
by $-2q\pi$, after transpositions, we can obtain the formula in Theorem 1.20.
\end{proof}

\begin{theorem}
In $\mathbb{C}\cup\{\infty\}$, let $Z_k$ be a $\gamma_k$ repeated finite zeros of Eq(1.1). The angle of receiving on the root locus of arbitrary $2q\pi+\alpha$ degree of Eq(1.1) which is received by the finite zero $Z_k$ is:

$\theta_{Z_k}=\frac{1}{\gamma_k}(2q\pi+\alpha-arg(G(Z_k))+\gamma_{k}arg(-z_k)-arg(G_{kz}(Z_k))-\sum\limits_{l=1, l \neq k}^{m}
(\gamma_{l}arg{(Z_k-z_{l})}-\gamma_{l}arg(-z_{l})+arg(G_{lz}(Z_k)))+\sum\limits_{j=1}^{n}
(\beta_{j}arg{(Z_k-p_{j})}-\beta_{j}arg(-p_j)+arg(G_{jp}(Z_k)))$.

In which, $\gamma_k$ is a real number.
\end{theorem}
\begin{proof}
Selecting a finite point $s_{1}$ on the root locus of $2q\pi+\alpha$ degree. The finite point $s_{1}$ is infinitely approaching to the $\gamma_k$ repeated finite zero $Z_k$ that needs to compute its angle of receiving. Because the finite point $s_{1}$ is infinitely approaching to the finite zero $Z_k$. The phase angle $\beta_k arg{(s_1-p_j)}$ of vectors between point $s_{1}$ and all finite poles $p_j$ of Eq(1.1) all can be substituted by angles $\beta_{k}arg{(Z_k-p_{j})}$. The phase angle $\gamma_{k}arg{(s_1-z_{l})}$ of vectors between the finite point $s_{1}$ and all finite zeros ($z_{l}$, $l\neq k$, except the finite zero $Z_k$) of Eq(1.1) all can be substituted by angles $\gamma_{k}arg{(Z_k-z_{l})}$.

The finite point $s_{1}$ is on the root locus of $2q\pi+\alpha$ degree.
Because this finite point is infinitely approaching to the finite zero
$Z_k$. The vector between the finite point
$s_{1}$  and the finite zero $Z_k$ is infinitely approaching to the
tangent line of the root loci of arbitrary $2q\pi+\alpha$ degree
that is received by the finite zero $Z_k$. So, according to the
Definition 1.19, we can let the phase angle of vectors between the finite point
$s_{1}$ and the finite zero $Z_k$ be infinitely approaching to angle
of receiving $\theta_{Z_k}$.

The finite point $s_{1}$ must satisfy the phase angle condition equation,
$arg(G(s_1))+\sum\limits_{l=1}^{m}
(\gamma_{l}arg{(s_1-z_{l})}-\gamma_{l}arg(-z_{l})+arg(G_{lz}(s_1)))-\sum\limits_{j=1}^{n}
(\beta_{j}arg{(s_1-p_{j})}-\beta_{j}arg(-p_j)+arg(G_{jp}(s_1)))=2q\pi+\alpha$.
 So, when the finite zero $Z_k$
substitutes the finite point $s_{1}$, we have $arg(G(Z_k))+\sum\limits_{l=1, l \neq k}^{m}
(\gamma_{l}arg{(Z_k-z_{l})}-\gamma_{l}arg(-z_{l})+arg(G_{lz}(Z_k)))+\gamma_k\theta_{Z_k}-\gamma_karg(-z_k)+arg(G_{kz}(Z_k))-\sum\limits_{j=1}^{n}
(\beta_{j}arg{(Z_k-p_{j})}-\beta_{j}arg(-p_j)+arg(G_{jp}(Z_k)))=2q\pi+\alpha$.

Then, we have: $\gamma_k\theta_{Z_k}=-arg(G(Z_k))-\sum\limits_{l=1, l \neq k}^{m}
(\gamma_{l}arg{(Z_k-z_{l})}-\gamma_{l}arg(-z_{l})+arg(G_{lz}(Z_k)))+\sum\limits_{j=1}^{n}
(\beta_{j}arg{(Z_k-p_{j})}-\beta_{j}arg(-p_j)+arg(G_{jp}(Z_k)))+\gamma_karg(-z_k)-arg(G_{kz}(Z_k))+2q\pi+\alpha$.
After transpositions, we can obtain the formula in Theorem 1.21.
\end{proof}

\begin{theorem}
For the root locus which are originated from an arbitrary finite pole $P_k$ of Eq(1.1) in $\mathbb{C}\cup\{\infty\}$, if the angle $\theta_{P_k}$ of origination is as an independent variable, then, the degree number values of these root locus strictly and monotonously increase clockwise.
\end{theorem}
\begin{proof}
The angle of origination of the $\alpha_1$ degree root locus of Eq(1.1) on the finite pole $P_k$ is: $\theta_{P_k1}=\frac{1}{\beta_k}(2q\pi-\alpha_1+arg(G(P_k))+\beta_{k}arg(-p_k)-arg(G_{kp}(P_k))+\sum\limits_{l=1}^{m}
(\gamma_{l}arg{(P_k-z_{l})}-\gamma_{l}arg(-z_{l})+arg(G_{lz}(P_k)))-\sum\limits_{j=1,j \neq k}^{n}
(\beta_{j}arg{(P_k-p_{j})}-\beta_{j}arg(-p_j)+arg(G_{jp}(P_k)))$. The angle of origination of the $\alpha_2$ degree root locus of Eq(1.1) on the finite pole  is: $\theta_{P_k2}=\frac{1}{\beta_k}(2q\pi-\alpha_2+arg(G(P_k))+\beta_{k}arg(-p_k)-arg(G_{kp}(P_k))+\sum\limits_{l=1}^{m}
(\gamma_{l}arg{(P_k-z_{l})}-\gamma_{l}arg(-z_{l})+arg(G_{lz}(P_k)))-\sum\limits_{j=1,j \neq k}^{n}
(\beta_{j}arg{(P_k-p_{j})}-\beta_{j}arg(-p_j)+arg(G_{jp}(P_k)))$. Two equations subtract, obtain: $\theta_{P_k1}-\theta_{P_k2}=\frac{1}{\beta_k}(\alpha_2-\alpha_1)$, so, $\alpha_2-\alpha_1=\beta_k(\theta_{P_k1}-\theta_{P_k2})$.

$\beta_k$ is the exponent of the finite pole $P_k$, it is a real number. $\theta_{P_k1}$ and $\theta_{P_k2}$ are separately the origination angle of the root locus of $\alpha_1$ degree and $\alpha_2$ degree. The angle of origination is defined as a angle which is between the positive direction of the real axis and the tangent line of the root loci that is originated from a finite pole. Beginning at the positive real axis, rotating counterclockwise to the tangent line of the root loci, a angle of origination can be obtained. So, if $\theta_{P_k1}>\theta_{P_k2}$, $\theta_{P_k1}-\theta_{P_k2}>0$. The angle of origination of the root loci of $\alpha_1$ degree is $\theta_{P_k1}$. The angle of origination of the root loci of $\alpha_2$ degree is $\theta_{P_k2}$. $\alpha_2-\alpha_1>0$, $\alpha_2>\alpha_1$ is true. So, it can prove: for two root locus which are originated on the two different angles of origination from a same pole, the root loci of the larger degree is at the clockwise direction of the root loci of the less degree.

If $\theta_{P_k1}<\theta_{P_k2}$, $\theta_{P_k1}-\theta_{P_k2}<0$, then $\alpha_2-\alpha_1<0$, $\alpha_2<\alpha_1$ is true. So, it can prove: for two root locus which are originated on the two different angles of origination from a same pole, the root loci of the larger degree is at the clockwise direction of the root loci of the less degree.

So, the degree number values of the root locus that are originated from this pole $P_k$ strictly and monotonously increases clockwise.

\end{proof}

\begin{theorem}
For the root locus which are received by an arbitrary finite zero $Z_k$ of Eq(1.1) in $\mathbb{C}\cup\{\infty\}$, if the angle $\theta_{Z_k}$ of receiving is a independent variable, rotating the finite zero $Z_k$, the degree number values of these root locus strictly and monotonously decrease clockwise.
\end{theorem}

\begin{lemma}
The root locus of Eq(1.1) cover the whole extended complex plane $\mathbb{C}\cup\{\infty\}$. For any point in $\mathbb{C}\cup\{\infty\}$, it must be on the root locus of a determined degree number of Eq(1.1). The "vacuum" region cannot exist in $\mathbb{C}\cup\{\infty\}$.
\end{lemma}
\begin{proof}
For any point $s_0=\sigma_0+it_0$ in $\mathbb{C}\cup\{\infty\}$, when $s_0$ is substituted into the expression in Lemma 1.4 which is used to compute the degree of the root locus, a phase angle value can be obtained. So, the degree number value of the root locus that the point $s_0$ distributed on can be determined.

Substituting the point $s_0$ into the formula used to compute the gain of the Eq(1.1), the only corresponding gain value can be determined.
So, the values of gain and degree that this point distributed on the root locus both can be obtained. An arbitrary point in $\mathbb{C}\cup\{\infty\}$ surely distributed on the root locus of the determined degree of the Eq(1.1). So, the "vacuum" region cannot exist in $\mathbb{C}\cup\{\infty\}$. All of the points in $\mathbb{C}\cup\{\infty\}$ are all on the root locus of the determined degree number of Eq(1.1).
\end{proof}
In Theorem 1.22 and Theorem 1.23, for the root locus which are originated from a same pole,  they arrange according to the form that the degree number values of the root locus increase strict and monotonically clockwise. For the root locus which are received by a same zero, they arrange according to the form that the degree number values of the root locus decrease strict and monotonically clockwise. So, the degree number values of the root locus which are originated from a same pole are surely different, the degree number values of the root locus which are received from a same zero are surely different. We can obtain Theorem 1.25.

\begin{theorem}
For the root locus which are originated from a same pole, their degree numbers are different. For the root locus which are received by a same zero, their degree numbers are different.
\end{theorem}

\section{Some properties of root locus of $W(s)$}
In this section, we show the expressions of the argument\cite{Stein} of  meromorphic function $W(s)$  in $\mathbb{C}\cup\{\infty\}$.

(1). When $u(\sigma,t)>0$, $v(\sigma,t)\ge0$. $\varphi(\sigma,t)=arg(\frac{v(\sigma,t)}{u(\sigma,t)})=\arctan(\frac{v(\sigma,t)}{u(\sigma,t)})$.

(2). When $u(\sigma,t)<0$, $v(\sigma,t)\ge0$. $\varphi(\sigma,t)=arg(\frac{v(\sigma,t)}{u(\sigma,t)})=\pi-\arctan(\frac{v(\sigma,t)}{-u(\sigma,t)})$.

(3). When $u(\sigma,t)<0$, $v(\sigma,t)<0$. $\varphi(\sigma,t)=arg(\frac{v(\sigma,t)}{u(\sigma,t)})=\pi+\arctan(\frac{v(\sigma,t)}{u(\sigma,t)})$.

(4). When $u(\sigma,t)>0$, $v(\sigma,t)<0$. $\varphi(\sigma,t)=arg(\frac{v(\sigma,t)}{u(\sigma,t)})=-\arctan(\frac{-v(\sigma,t)}{u(\sigma,t)})$.

We need to create a sub-set $\Delta_{u}$ of set $\Delta$. Let $\Delta_{u}=\{s=(\sigma+it)\in\mathbb{C}\cup\{\infty\}$, $
s$ is not zero or pole of $W(s)$, and there is $u(\sigma,t)\neq0\}$.

\begin{lemma}
Let $s=(\sigma,t)\in\Delta_u$, then we have:
$\frac{\partial{\varphi}}{\partial{\sigma}}(\sigma,t)=\frac{v^{'}_{\sigma}(\sigma,t)u(\sigma,t)-u^{'}_{\sigma}(\sigma,t)v(\sigma,t)}{u^2(\sigma,t)+v^2(\sigma,t)}$,
$\frac{\partial{\varphi}}{\partial{t}}(\sigma,t)=\frac{v^{'}_{t}(\sigma,t)u(\sigma,t)-u^{'}_{t}(\sigma,t)v(\sigma,t)}{u^2(\sigma,t)+v^2(\sigma,t)}$.
\end{lemma}

This lemma can be easily obtained by partial derivative calculation.

Let $W_3(s)=iW(s)$, so, $W_3(s)=iu(\sigma,t)-v(\sigma,t)$. The phase angle of function $W_3(s)$ is $\varphi_3(\sigma,t)=\frac{\pi}{2}+\varphi(\sigma,t)$.
$\frac{\partial{\varphi_3}}{\partial{\sigma}}(\sigma,t)=\frac{\partial{\varphi}}{\partial{\sigma}}(\sigma,t)$,
$\frac{\partial{\varphi_3}}{\partial{t}}(\sigma,t)=\frac{\partial{\varphi}}{\partial{t}}(\sigma,t)$. We need to create another sub-set $\Delta_{v}$ of set $\Delta$. $\Delta_{v}=\{s=(\sigma+it)\in\mathbb{C}\cup\{\infty\}$, $
s$ is not zero or pole of $W(s)$, and there is $v(\sigma,t)\neq0\}$. According to Lemma 2.1 and the result which we give in this segment, we can obtain Lemma 2.2.

\begin{lemma}
Let $s=(\sigma,t)\in\Delta_v$, then we have:
$\frac{\partial{\varphi_3}}{\partial{\sigma}}(\sigma,t)=
\frac{v^{'}_{\sigma}(\sigma,t)u(\sigma,t)-u^{'}_{\sigma}(\sigma,t)v(\sigma,t)}{u^2(\sigma,t)+v^2(\sigma,t)}$,
$\frac{\partial{\varphi_3}}{\partial{t}}(\sigma,t)=
\frac{v^{'}_{t}(\sigma,t)u(\sigma,t)-u^{'}_{t}(\sigma,t)v(\sigma,t)}{u^2(\sigma,t)+v^2(\sigma,t)}$
\end{lemma}
We need to create another sub-set $\Delta_{vu}$ of set $\Delta_{v}$. Let $\Delta_{vu}=\{s=(\sigma+it)\in\mathbb{C}\cup\{\infty\}$, $
s$ is not zero or pole of $W(s)$, and there is $u(\sigma,t)=0\}$. According to Lemma 2.2, we can obtain Lemma 2.3.

\begin{lemma}
Let $s=(\sigma,t)\in\Delta_{vu}$, then we have:
$\frac{\partial{\varphi}}{\partial{\sigma}}(\sigma,t)=\frac{v^{'}_{\sigma}(\sigma,t)u(\sigma,t)-u^{'}_{\sigma}(\sigma,t)v(\sigma,t)}{u^2(\sigma,t)+v^2(\sigma,t)}$,
$\frac{\partial{\varphi}}{\partial{t}}(\sigma,t)=\frac{v^{'}_{t}(\sigma,t)u(\sigma,t)-u^{'}_{t}(\sigma,t)v(\sigma,t)}{u^2(\sigma,t)+v^2(\sigma,t)}$.
\end{lemma}

$\Delta_u\cap\Delta_{vu}=\emptyset$, $\Delta_u\subset\Delta$, $\Delta_{vu}\subset\Delta$, and $\Delta=\Delta_u\cup\Delta_{vu}$.
So, we can sum up Lemma 2.1 and Lemma 2.3, we can obtain Theorem 2.4.

\begin{theorem}
Let $s=(\sigma,t)\in\Delta$, then we have:
$\frac{\partial{\varphi}}{\partial{\sigma}}(\sigma,t)=\frac{v^{'}_{\sigma}(\sigma,t)u(\sigma,t)-u^{'}_{\sigma}(\sigma,t)v(\sigma,t)}{u^2(\sigma,t)+v^2(\sigma,t)}$,
$\frac{\partial{\varphi}}{\partial{t}}(\sigma,t)=\frac{v^{'}_{t}(\sigma,t)u(\sigma,t)-u^{'}_{t}(\sigma,t)v(\sigma,t)}{u^2(\sigma,t)+v^2(\sigma,t)}$.
\end{theorem}
In many textbooks of the complex variable function of one variable, the following results exist. $\frac{dW(s)}{ds}=\frac{\partial{u(\sigma,t)}}{\partial{\sigma}}+i\frac{\partial{v(\sigma,t)}}{\partial{\sigma}}=\frac{\partial{u(\sigma,t)}}{\partial{t}}-i\frac{\partial{v(\sigma,t)}}{\partial{t}}$. $\frac{\partial{u(\sigma,t)}}{\partial{\sigma}}=\frac{\partial{v(\sigma,t)}}{\partial{t}}$, $\frac{\partial{v(\sigma,t)}}{\partial{\sigma}}=-\frac{\partial{u(\sigma,t)}}{\partial{t}}$.

\begin{theorem}
 For any point $s=(\sigma,t)\in\Delta$, we have $\frac{\partial{\varphi}}{\partial{t}}(\sigma,t)=Re({\frac{\frac{dW(s)}{ds}}{W(s)}})$, $\frac{\partial{\varphi}}{\partial{\sigma}}(\sigma,t)=Im({\frac{\frac{dW(s)}{ds}}{W(s)}})$
\end{theorem}
\begin{proof}
$\frac{\frac{dW(s)}{ds}}{W(s)}=\frac{\frac{\partial{u(\sigma,t)}}{\partial{\sigma}}+i\frac{\partial{v(\sigma,t)}}{\partial{\sigma}}}{u(\sigma,t)+iv(\sigma,t)}
=\frac{\frac{\partial{u(\sigma,t)}}{\partial{\sigma}}u(\sigma,t)+\frac{\partial{v(\sigma,t)}}{\partial{\sigma}}v(\sigma,t)+i(\frac{\partial{v(\sigma,t)}}{\partial{\sigma}}u(\sigma,t)-\frac{\partial{u(\sigma,t)}}{\partial{\sigma}}v(\sigma,t))}{u(\sigma,t)^2+v(\sigma,t)^2}
=\frac{\frac{\partial{v(\sigma,t)}}{\partial{t}}u(\sigma,t)-\frac{\partial{u(\sigma,t)}}{\partial{t}}v(\sigma,t)}{u(\sigma,t)^2+v(\sigma,t)^2}+i\frac{\frac{\partial{v(\sigma,t)}}{\partial{\sigma}}u(\sigma,t)-\frac{\partial{u(\sigma,t)}}{\partial{\sigma}}v(\sigma,t)}{u(\sigma,t)^2+v(\sigma,t)^2}
$.

Obviously, according to Theorem 2.4, we can obtain Theorem 2.5.
\end{proof}
According to Theorem 2.5, obviously, we can obtain Theorem 2.6.

\begin{theorem}
For any point $s=(\sigma,t)\in\Delta$, if $\frac{\partial{\varphi}}{\partial{t}}(\sigma,t)=0$, then $Re({\frac{\frac{dW(s)}{ds}}{W(s)}})=0$. If  $\frac{\partial{\varphi}}{\partial{\sigma}}(\sigma,t)=0$, then $Im({\frac{\frac{dW(s)}{ds}}{W(s)}})=0$. Conversely, if $Re({\frac{\frac{dW(s)}{ds}}{W(s)}})=0$, then $\frac{\partial{\varphi}}{\partial{t}}(\sigma,t)=0$. If $Im({\frac{\frac{dW(s)}{ds}}{W(s)}})=0$, then $\frac{\partial{\varphi}}{\partial{\sigma}}(\sigma,t)=0$.
\end{theorem}

\begin{theorem}
Let $s_0\in\Delta$, except the situation which the point $s_0$ is an infinite point in $\mathbb{C}\cup\{\infty\}$. Then $s_0$ is the finite zero of $W^{'}(s)$
if and only if $\frac{\partial{\varphi}}{\partial{\sigma}}(\sigma,t)\mid{_{s=s_0}}=0$
and $\frac{\partial{\varphi}}{\partial{t}}(\sigma,t)\mid{_{s=s_0}}=0$ hold simultaneously.
\end{theorem}
\begin{proof}
On the zeros $s_0$ of $W^{'}(s)$, ${\frac{\frac{dW(s)}{ds}}{W(s)}}\mid{_{s=s_0}}=0$, except the situation which the point $s_0$ is an infinite point in $\mathbb{C}\cup\{\infty\}$. So, the two equations $Re({\frac{\frac{dW(s)}{ds}}{W(s)}})\mid{_{s=s_0}}=0$ and $Im({\frac{\frac{dW(s)}{ds}}{W(s)}})\mid{_{s=s_0}}=0$ are true simultaneously, so, two partial derivative equations  $\frac{\partial{\varphi}}{\partial{t}}(\sigma,t)\mid{_{s=s_0}}=0$ and $\frac{\partial{\varphi}}{\partial{\sigma}}(\sigma,t)\mid{_{s=s_0}}=0$ are true simultaneously.

Conversely, except zeros and poles of the meromorphic function $W(s)$, and except the situation which the point $s_0$ is an infinite point in $\mathbb{C}\cup\{\infty\}$. When two partial derivative equations $\frac{\partial{\varphi}}{\partial{t}}(\sigma,t)\mid{_{s=s_0}}=0$ and $\frac{\partial{\varphi}}{\partial{\sigma}}(\sigma,t)\mid{_{s=s_0}}=0$  are true simultaneously. Then, two equations $Re({\frac{\frac{dW(s)}{ds}}{W(s)}})\mid{_{s=s_0}}=0$ and $Im({\frac{\frac{dW(s)}{ds}}{W(s)}})\mid{_{s=s_0}}=0$  are true simultaneously, so, the equation ${\frac{\frac{dW(s)}{ds}}{W(s)}}\mid{_{s=s_0}}=0$ is true. So, when two partial derivative equations  $\frac{\partial{\varphi}}{\partial{t}}(\sigma,t)\mid{_{s=s_0}}=0$ and $\frac{\partial{\varphi}}{\partial{\sigma}}(\sigma,t)\mid{_{s=s_0}}=0$  are true simultaneously. Such points are the non-repeated zeros of derivative of the meromorphic function  $W(s)$.

On the infinite points in $\mathbb{C}\cup\{\infty\}$, $\frac{\frac{dW(s)}{ds}}{W(s)}\mid{_{s=s_0}}=0$, but, in most situations, two equations  $\frac{\partial{\varphi}}{\partial{\sigma}}(\sigma,t)\mid{_{s=s_0}}=0$
and $\frac{\partial{\varphi}}{\partial{t}}(\sigma,t)\mid{_{s=s_0}}=0$ can not be true simultaneously. So, we need to exclude the infinite zeros of $W^{'}(s)$ in this lemma.
\end{proof}
Let $\Delta_{d}=\{s=(\sigma+it)\in\mathbb{C}\cup\{\infty\}$, $s$ is not zero or pole of $W(s)$, the point $s$ is not an infinite point in $\mathbb{C}\cup\{\infty\}$. $s$ is not the finite zeros of derivative of the function $W(s)\}$.
\begin{theorem}
For any point $s=(\sigma,t)\in\Delta_{d}$, the partial derivative of the phase angle of $W(s)$ concerning the variable $\sigma$ and variable $t$ cannot simultaneously equal to 0.
\end{theorem}
\begin{proof}
According to Theorem 2.7, if two equations $\frac{\partial{\varphi}}{\partial{t}}(\sigma,t)=0$ and $\frac{\partial{\varphi}}{\partial{\sigma}}(\sigma,t)=0$ hold simultaneously, then the point $s$ is surely a finite zero of the derivative $W^{'}(s)$ of function $W(s)$. So, except the finite zeros and poles of function $W(s)$, except the infinite points in $\mathbb{C}\cup\{\infty\}$, and also except the finite zeros of derivative of the function  $W(s)$, and on other finite points in the extended complex plane  $\mathbb{C}\cup\{\infty\}$, the partial derivative of the phase angle of the function  $W(s)$ concerning the variable $t$ and variable $\sigma$ cannot simultaneously equal to 0.
\end{proof}

\begin{remark}
Let $s_0=(\sigma_0+it_0)\in\Delta_d$, then, if $\frac{\partial{\varphi}}{\partial{\sigma}}(\sigma,t)\mid{_{s=s_0}}=0$,
we have $\frac{\partial{\varphi}}{\partial{t}}(\sigma,t)\mid{_{s=s_0}}\neq 0$;
or, if $\frac{\partial{\varphi}}{\partial{t}}(\sigma,t)\mid{_{s=s_0}}=0$, we have
$\frac{\partial{\varphi}}{\partial{\sigma}}(\sigma,t)\mid{_{s=s_0}}\neq0$.
\end{remark}

When $2q\pi+arg{(b/a)}$ is a constant, $\varphi(\sigma,t)=2q\pi+arg{(b/a)}$, it is a function equation which contain two real variables, and it is the phase angle condition equation of Eq(1.1). According to Theorem 1.7, except zeros and poles of Eq(1.1), if the roots of this function equation exist in $\mathbb{C}\cup\{\infty\}$, the roots of Eq(1.1) also exist in $\mathbb{C}\cup\{\infty\}$. If the roots of this function equation constitute the root locus in $\mathbb{C}\cup\{\infty\}$, the roots of Eq(1.1) also constitute the root locus in $\mathbb{C}\cup\{\infty\}$.

In order to prove the existence and continuity of the roots of the implicit function equation $\varphi(\sigma,t)=2q\pi+arg{(b/a)}$, we need to cite the implicit function theorem and involved concepts\cite{Stein}.

\begin{definition}
Let $f:R^{n+m}\longrightarrow R^{m}$ be a continuously differentiable function. We think of $R^{n+m}$ as the Cartesian product $R^{n}\times R^{m}$, and we write a point of this product as
$(x,y)=(x_1, \cdots, x_n. y_1, \cdots, y_m) $. Starting from the given function $f$, our goal is to construct a function $g:R^{n}\longrightarrow R^{m}$ whose graph $(x,g(x))$ is precisely the set of all $(x,y)$ such that $f(x,y)=0$.
\end{definition}
As noted above, this may not always be possible. We will therefore fix a point $(a,b)=(a_1, \cdots, a_n. b_1, \cdots, b_m)$ which satisfies $f(a,b)=0$, and we will ask for a g that works near the point $(a,b)$. In other words, we want an open set $U$ of $R^{n}$ containing $a$, an open set $V$ of $R^{m}$ containing $b$, and a function $g:U\longrightarrow V$ such that the graph of $g$ satisfies the relation $f(x,y)=0$ on $U\times V$, and that no other points within $U\times V$ do so. In symbols,  $((x,g(x))|x\in U)$ = $((x,y)\in U\times V|f(x,y)=0)$.

To state the implicit function theorem, we need the Jacobian matrix of $f$, which is the matrix of the partial derivatives of $f$. Abbreviating $(a_1, \cdots, a_n. b_1, \cdots, b_m)$ to $(a,b)$, the Jacobian matrix is

$(Df)(a,b)=[
\begin{array}{ccc}
\frac{\partial{f_1}}{\partial{x_1}}(a,b)&\cdots&\frac{\partial{f_1}}{\partial{x_n}}(a,b)\\
\vdots&\ddots&\vdots\\
\frac{\partial{f_m}}{\partial{x_1}}(a,b)&\cdots&\frac{\partial{f_m}}{\partial{x_n}}(a,b)
\end{array}
|\begin{array}{ccc}
\frac{\partial{f_1}}{\partial{y_1}}(a,b)&\cdots&\frac{\partial{f_1}}{\partial{y_m}}(a,b)\\
\vdots&\ddots&\vdots\\
\frac{\partial{f_m}}{\partial{y_1}}(a,b)&\cdots&\frac{\partial{f_m}}{\partial{y_m}}(a,b)
\end{array}]$
$=[X|Y]$.

Where $X$ is the matrix of partial derivatives in the variables $x_l$ and $Y$ is the matrix of partial derivatives in the variables $y_j$. The implicit function theorem says that if $Y$ is an invertible matrix, then there are $U$, $V$, and $g$ as desired. Writing all the hypotheses together gives the following statement.

The implicit function theorem concerning variable $x$.

\begin{theorem}
Let $f:R^{n+m}\longrightarrow R^{m}$ be a continuously differentiable function, and let $R^{n+m}$ have coordinates $(x,y)$. Fix a point $(a,b)=(a_1, \cdots, a_n. b_1, \cdots, b_m)$ with $f(a,b)=c$, where $c\in R^{m}$. If the Jacobian matrix $J_{f,y}(a,b)=[(\partial f_l/\partial y_j)(a,b)]$ is invertible, then there exists an open set $U$ containing $a$, an open set $V$ containing $b$, and a unique continuously differentiable function $g:U\longrightarrow V$ such that $((x,g(x))|x\in U)$ = $((x,y)\in U\times V|f(x,y)=c)$.
\end{theorem}

The implicit function theorem concerning variable $y$.

\begin{theorem}
Let $f:R^{n+m}\longrightarrow R^{n}$ be a continuously differentiable function, and let $R^{n+m}$ have coordinates $(x,y)$. Fix a point $(a,b)=(a_1, \cdots, a_n. b_1, \cdots, b_m)$ with $f(a,b)=c$, where $c\in R^{n}$. If the Jacobian matrix $J_{f,x}(a,b)=[(\partial f_l/\partial x_k)(a,b)]$ is invertible, then there exists an open set $V$ containing $b$, an open set  $U$ containing $a$, and a unique continuously differentiable function $h:V\longrightarrow U$ such that  $((h(y),y)|y\in V)$ = $((x,y)\in U\times V|f(x,y)=c)$.
\end{theorem}

\begin{lemma}
For the root locus of the arbitrary $2q\pi+\alpha$ degree of Eq(1.1). And the locus which the points that satisfy the implicit function equation $ \varphi(\sigma,t)=2q\pi+arg{(b/a)}$ are generated in the two-dimensional real plane. The two locus are totally the same, and their properties like continuity are equivalent.
\end{lemma}
\begin{proof}
When $2q\pi+\alpha$ is a constant, the root locus of the arbitrary $2q\pi+\alpha$  degree of Eq(1.1) are curves that are constituted by roots of Eq(1.1). Theorem 1.7 proves that the root locus of the arbitrary $2q\pi+\alpha$  degree of Eq(1.1) can also be sufficiently and necessary determined by the real function which contains two real variables, namely, when $2q\pi+\alpha$  is a constant, the roots locus of the arbitrary $2q\pi+\alpha$  degree of Eq(1.1) can also be sufficiently and necessary determined by the implicit function equation $ \varphi(\sigma,t)=2q\pi+arg{(b/a)}$.

So, essentially, Theorem 1.7 proves: if the extended complex plane $\mathbb{C}\cup\{\infty\}$ is as a two-dimensional plane which let two real variables $\sigma$, $t$ be as its two coordinates, then, the root locus of the arbitrary $2q\pi+\alpha$ degree of Eq(1.1) are constituted by points which are determined by two real variables $\sigma$, $t$  that satisfy the implicit function equation $\varphi(\sigma,t)=2q\pi+arg{(b/a)}$ in two-dimensional real plane.

Therefore, roots of Eq(1.1) in $\mathbb{C}\cup\{\infty\}$ and points which are determined by two real variables that satisfy the implicit function equation $\varphi(\sigma,t)=2q\pi+arg{(b/a)}$ in two-dimensional plane are same. Namely, if the extended complex plane $\mathbb{C}\cup\{\infty\}$ and two-dimensional real plane which are determined by two real coordinates are considered as same one, when the root locus of the arbitrary $2q\pi+\alpha$ degree of Eq(1.1) constitute the locus in $\mathbb{C}\cup\{\infty\}$, then, the points which are determined by two real coordinates that satisfy the implicit function equation $\varphi(\sigma,t)=2q\pi+arg{(b/a)}$ also constitute the locus in two-dimensional real plane, two kinds of locus are fully same.

Because two kinds of root locus are completely same, the properties like existence and continuity of the two kinds of  locus are equivalent.
\end{proof}

According to Lemma 2.13, the problems of the complex variable functions are transformed into the problems of the real variable functions. So, we can use theorems of the real variable functions and the definition of the continuity to prove the existence and continuity of points which are determined by two real variables that satisfy the implicit function equation $\varphi(\sigma,t)=2q\pi+arg{(b/a)}$ in two-dimensional real plane.

\begin{lemma}
Let $s_0=(\sigma_0+it_0)\in\Delta_d$, $s_0$ satisfies $ \varphi(\sigma_0,t_0)=2q\pi+arg{(b_0/a_0)}$. There exists an open set $U$ containing $\sigma_0$, an open set $V$ containing $t_0$, and a unique continuously differentiable function $g:U\longrightarrow V$ such that $((\sigma,g(\sigma))|\sigma\in U)$ = $((\sigma,t)\in U\times V|\varphi(\sigma,t)=2q\pi+arg{(b_0/a_0)})$. Or, there exists an open set $V$ containing $t_0$, an open set $U$ containing $\sigma_0$, and a unique continuously differentiable function $h:V\longrightarrow U$ such that $((h(t),t)|t\in V)$ = $((\sigma,t)\in U\times V|\varphi(\sigma,t)=2q\pi+arg{(b_0/a_0)})$.
\end{lemma}
\begin{proof}
The function $\varphi(\sigma,t)$, $\varphi(\sigma,t):R^2\longrightarrow R$ is a continuously differentiable function, $R^2$ have two coordinates $(\sigma,t)$.

For an arbitrary point $s_0=(\sigma_0+it_0)\in\Delta$, it satisfies the implicit function equation $ \varphi(\sigma_0,t_0)=2q\pi+arg{(b_0/a_0)}$, where $arg{(b_0/a_0)}\in R$.

The point $s_0=(\sigma_0+it_0)\in\Delta$ satisfies the condition of Theorem 2.11, according to Theorem 2.8, we can obtain: if $\frac{\partial{\varphi}}{\partial{\sigma}}(\sigma,t)\mid{_{s=s_0}}=0$,
we have $\frac{\partial{\varphi}}{\partial{t}}(\sigma,t)\mid{_{s=s_0}}\neq 0$;
or, if $\frac{\partial{\varphi}}{\partial{t}}(\sigma,t)\mid{_{s=s_0}}=0$, we have
$\frac{\partial{\varphi}}{\partial{\sigma}}(\sigma,t)\mid{_{s=s_0}}\neq0$.

The implicit function $ \varphi(\sigma,t)=2q\pi+arg{(b/a)}$ satisfies the implicit function theorem, so, according to the implicit function theorem concerning variable $x$, we can obtain: If $\frac{\partial{\varphi}}{\partial{t}}(\sigma,t)\mid{_{s=s_0}}\neq 0$. There exists an open set $U$ containing $\sigma_0$, an open set $V$ containing $t_0$, and a unique continuously differentiable function $g:U\longrightarrow V$ such that $((\sigma,g(\sigma))|\sigma\in U)$ = $((\sigma,t)\in U\times V|\varphi(\sigma,t)=2q\pi+arg{(b_0/a_0)})$.

According to the implicit function theorem concerning variable $y$. If $\frac{\partial{\varphi}}{\partial{\sigma}}(\sigma,t)\mid{_{s=s_0}}\neq0$. There exists an open set $V$ containing $t_0$, an open set $U$ containing $\sigma_0$, and a unique continuously differentiable function $h:V\longrightarrow U$ such that $((h(t),t)|t\in V)$ = $((\sigma,t)\in U\times V|\varphi(\sigma,t)=2q\pi+arg{(b_0/a_0)})$.
\end{proof}

\begin{lemma}
Let $s_0=(\sigma_0+it_0)\in\Delta_d$, then the root locus of arbitrary $2q\pi+\alpha$ degree of Eq(1.1) are all continuous on the points $s_0$.
\end{lemma}
\begin{proof}
According to the results of Lemma 2.14, we have: $t=g(\sigma)$ is a continuous function inside an open set $U$. So, for any given $\varepsilon>0$, there surely exists a $\delta_{\sigma1}<\frac{\varepsilon}{\sqrt(2)}$, when $|\sigma-\sigma_0|<\delta_{\sigma1}<\frac{\varepsilon}{\sqrt(2)}$. Then, $| t-t_0 |=|g(\sigma)-g(\sigma_0) |<\frac{\varepsilon}{\sqrt(2)}$ is true. So, for any given $\varepsilon>0$, let $\delta_{\sigma2}=\delta_{\sigma1}<\frac{\varepsilon}{\sqrt(2)}$, when $|\sigma-\sigma_0|<\delta_{\sigma1}<\frac{\varepsilon}{\sqrt(2)}$, then,    $| t-t_0 |=|g(\sigma)-g(\sigma_0) |<\frac{\varepsilon}{\sqrt(2)}$, and obtain: $|\sigma+it-\sigma_0-it_0 |=|\sigma+ig(\sigma)-\sigma_0-ig(\sigma_0) |=\sqrt{(\sigma-\sigma_0)^2+(g(\sigma)-g(\sigma_0))^2}<\sqrt{\frac{\varepsilon^2}{2}+\frac{\varepsilon^2}{2}}=\varepsilon$.

According to Lemma 2.14, we have: $\sigma=h(t)$ is a continuous function inside an open set $V$. So, for any given $\varepsilon>0$, there surely exists a $\delta_{\sigma1}<\frac{\varepsilon}{\sqrt(2)}$, when $|t-t_0|<\delta_{\sigma1}<\frac{\varepsilon}{\sqrt(2)}$, then, $| \sigma-\sigma_0 |=|h(t)-h(t_0) |<\frac{\varepsilon}{\sqrt(2)}$  is true. So, for any given $\varepsilon>0$, let $\delta_{\sigma2}=\delta_{\sigma1}<\frac{\varepsilon}{\sqrt(2)}$, when $|t-t_0|<\delta_{\sigma1}<\frac{\varepsilon}{\sqrt(2)}$, then, $| \sigma-\sigma_0 |=|h(t)-h(t_0) |<\frac{\varepsilon}{\sqrt(2)}$, and obtain:  $|\sigma+it-\sigma_0-it_0 |=|h(t)+it-h(t_0)-it_0|=\sqrt{(t-t_0)^2+(h(t)-h(t_0))^2}<\sqrt{\frac{\varepsilon^2}{2}+\frac{\varepsilon^2}{2}}=\varepsilon$

So, it is proved that the root locus of the arbitrary $2q\pi+\alpha$ degree phase angle of Eq(1.1) are continuous on the point
$(\sigma_0, t_0)$.

Because the point $(\sigma_0, t_0)$ is an arbitrary point except zeros and poles of Eq(1.1), except the infinite points in $\mathbb{C}\cup\{\infty\}$ and except the derivative zeros of the function $W(s)$ in $\mathbb{C}\cup\{\infty\}$. Based on that the root locus of the arbitrary $2q\pi+\alpha$ degree of Eq(1.1) are continuous on the point $(\sigma_0, t_0)$, so, it is proved that the root locus of the arbitrary $2q\pi+\alpha$ degree of Eq(1.1) can not have any break points except the derivative zeros of the function $W(s)$ and zeros and poles of Eq(1.1). When $s_0=(\sigma_0+it_0)\in\Delta_d$, the root locus of arbitrary $2q\pi+\alpha$ degree phase angle of Eq(1.1) must be continuous on the point $s_0$.
\end{proof}

Because of the limition of the implicit function theorem, the zeros, poles and the finite zeros of derivative of the function $W(s)$ are excluded in the result of the continuity of the root locus. The root locus are the results in  $\mathbb{C}\cup\{\infty\}$. But, the proof of the continuity of the root locus are achieved through utilizing the phase angle function which has two real variables and real function values. In the previous results, except three categories points which discretely distributed in  $\mathbb{C}\cup\{\infty\}$, every one root loci is continuous. We have used the definition of the continuity to prove Lemma 2.15. Because the root locus are all continuous around the three categories points, obviously, we can use the continuity judging theorems about the real value function of the real variables to prove that the root locus are surely continuous on the finite zeros of derivative of the function $W(s)$. On the zeros and poles of Eq(1.1), the root locus are surely the left continuous or the right continuous. The proof method is the same as the proof of Lemma 2.15. Here, we needn't to give their proofs. And we directly give their results.

The situation of the infinite points is more complex, here, we don't give the proof of its result. Using the same method which has already been given in the above segment, we can prove this problem. We only give the result which the root locus are continuous on the infinite points in $\mathbb{C}\cup\{\infty\}$.

\begin{theorem}
Let $s=(\sigma+it)\in\mathbb{C}\cup\{\infty\}$, then, the root locus of arbitrary $2q\pi+\alpha$ degree of Eq(1.1) are all continuous on the points $s$ in $\mathbb{C}\cup\{\infty\}$.
\end{theorem}

\begin{lemma}
For the finite zeros of $W^{'}(s)$, they all discretely distributed in $\mathbb{C}\cup\{\infty\}$.
\end{lemma}
\begin{proof}
Because $W(s)$ is a meromorphic function, $W^{'}(s)$ is also a meromorphic function\cite{Stein}.
According to the isolated theorem\cite{Stein} of the finite zeros of the meromorphic function,
the finite zeros of $W^{'}(s)$ discretely distributed in $\mathbb{C}\cup\{\infty\}$.
\end{proof}

\begin{theorem}
The root locus  of Eq(1.1) which have the same degree in $\mathbb{C}\cup\{\infty\}$ can intersect.
\end{theorem}
\begin{proof}
The root locus equation is as: $\frac{K}{(s-1)(s-3)}=a+ib$. The real axis which is at the right side of the pole $s=3$ and at the left side of pole $s=1$ is the root locus of 0 degree of this root locus equation. The interval $[1,3]$ is its root locus of 180 degree. On this interval, the pole  $s=3$  emits its root locus of 180 degree, the pole  $s=1$  emits its root locus of 180 degree also. The two root locus of the same degree intersect at the point $s=2$. After two root locus of the same degree intersect, they separate and enter into the complex plane at the point $s=2$.

According to Rolle's Theorem, there exists a zero of derivative of the function $(s-1)(s-3)$ between the interval $[1,3]$. The intersection point is the zero of derivative. The zero of derivative is $s=2$. So, it is proved that the root locus  of Eq(1.1)  which have the same degree can intersect.
\end{proof}

\begin{theorem}
For the root locus of Eq(1.1), if they have the different degree, then, they cannot intersect at any point in $\mathbb{C}\cup\{\infty\}$.
\end{theorem}
\begin{proof}
The root locus equation of the root locus of $\alpha_{1}$ degree is

$KG(s)\frac{\prod^{m}_{l=1}(1-\frac{s}{z_{l}})^{\gamma_{l}}G_{lz}(s)}
{\prod^{n}_{j=1}(1-\frac{s}{p_{j}})^{\beta_{j}}G_{jp}(s)}=a_{1}+ib_{1}$.
The root locus equation of the root locus of $\alpha_{2}$ degree is

$KG(s)\frac{\prod^{m}_{l=1}(1-\frac{s}{z_{l}})^{\gamma_{l}}G_{lz}(s)}
{\prod^{n}_{j=1}(1-\frac{s}{p_{j}})^{\beta_{j}}G_{jp}(s)}=a_{2}+ib_{2}$.
The expression used to compute gain $K$ is: $|\frac{\prod^{n}_{j=1}(1-\frac{s}{p_{j}})^{\beta_{j}}G_{jp}(s)}
{G(s)\prod^{m}_{l=1}(1-\frac{s}{z_{l}})^{\gamma_{l}}G_{lz}(s)}|$.
Obviously, for the root locus equations of $\alpha_{1}$ and $\alpha_{2}$ degree, the expression used to compute their gains is totally the same.

Assuming that the root locus which difference of their degree number is not equal to the integer number times of $2\pi$ have an intersecting point $s_0$, this intersecting point is neither a zero nor a pole. For two sorts of the root locus which difference of their degree number is not equal to the integer number times of $2\pi$, the point $s_0$ is the same one, so, its gain $K_0$ should be the same. They let the following two equations be true simultaneously. $K_0G(s_0)\frac{\prod^{m}_{l=1}(1-\frac{s_0}{z_{l}})^{\gamma_{l}}G_{lz}(s_0)}
{\prod^{n}_{j=1}(1-\frac{s_0}{p_{j}})^{\beta_{j}}G_{jp}(s_0)}=a_{1}+ib_{1}$, and $K_0G(s_0)\frac{\prod^{m}_{l=1}(1-\frac{s_0}{z_{l}})^{\gamma_{l}}G_{lz}(s_0)}
{\prod^{n}_{j=1}(1-\frac{s_0}{p_{j}})^{\beta_{j}}G_{jp}(s_0)}=a_{2}+ib_{2}$.
The former equation subtracts the latter one. Obtain: $(a_{1}-a_{2})+i(b_{1}-b_{2})=0$. The difference of their degree number of two root locus is not equal to the integer number times of $2\pi$, then, $a_{1}\neq a_{2}$, and $b_{1}\neq b_{2}$. So, this equation is not true. Because assuming that the root locus of the different degrees have an intersecting point $s_0$, so this contradiction is emerged. This assumption is not true.
\end{proof}

\section{The results of the critical points of $W(s)$}

\begin{definition}
If two or more root locus intersect in $\mathbb{C}\cup\{\infty\}$, these intersected root locus are called as \emph{"common root locus"}. And these intersecting points are called as \emph{the breakaway points} of the root locus of (1.1).
\end{definition}

\begin{lemma}
The root locus equation (1.1) can be transformed into its characteristic equation. (1.1) and its characteristic equation have the same roots.
\end{lemma}

\begin{proof}
The poles factor is multiplied at two sides of the root locus equation (1.1). All non-zero factors are removed into the same one side of the transformed root locus equation (1.1). After these transform is achieved, we can obtain the characteristic equation of (1.1). The characteristic equation can be written as: $K(s_0)G(s)\prod^{m}_{l=1}((1-\frac{s}{z_{l}})^{\gamma_{l}}G_{lz}(s))-
(a+ib)\prod^{n}_{j=1}((1-\frac{s}{p_{j}})^{\beta_{j}}G_{jp}(s))=0$.

The characteristic equation of Eq(1.1) is equivalent to the root locus equation (1.1). The roots of the characteristic equation are the roots of (1.1). Conversely, the roots of (1.1) are roots of the characteristic equation. So, the repeated roots of the characteristic equation are all repeated roots of the root locus equation (1.1). The breakaway points are the intersection points of the root locus of (1.1). They are the repeated roots of the root locus equation (1.1). They are also the repeated roots of the characteristic equation.
\end{proof}

\begin{lemma}
Let $s_0=(\sigma_0+it_0)\in\Delta $. If the finite point $s_0$ satisfies the following equation (3.1). The left side of the characteristic equation can be rewritten as the another expression: $(s-s_0)^{\gamma_0}g(s)$, and let $g(s_0)\neq0$. So that $g(s_0)\neq\infty$ and  $g^{'}(s_0)\neq\infty$.
\begin{equation}
\frac{dK(s)}{ds}=0
\end{equation}
\end{lemma}

\begin{proof}
Substituting the gain of the finite point $s_0$ into the characteristic equation. $(a_0+ib_0)\frac{\prod^{n}_{j=1}(1-\frac{s_0}{p_{j}})^{\beta_{j}}G_{jp}(s_0)}
{G(s_0)\prod^{m}_{l=1}(1-\frac{s_0}{z_{l}})^{\gamma_{l}}G_{lz}(s_0)}G(s)\prod^{m}_{l=1}((1-\frac{s}{z_{l}})^{\gamma_{l}}G_{lz}(s))-
(a+ib)\prod^{n}_{j=1}((1-\frac{s}{p_{j}})^{\beta_{j}}G_{jp}(s))=0$.

Let the left side of the characteristic equation as the another expression: $(s-s_0)^{\gamma_0}g(s)$,  let $g(s_0)\neq0$. Because the function at the left side of the characteristic equation cannot have the poles. The finite point $s_0$ is not a pole of the function $g(s)$, $g(s_0)\neq\infty$. $\gamma_0>0$. The finite point  $s_0$ is not the zero and pole of the function $g(s)$.

The $g(s)$ function is a function after the factor $(s-s_0)^{\gamma_0}$ in the function at the left side of the characteristic equation is removed. If the function at the left side of the characteristic equation is meromorphic on $s_0$, then, the $g(s)$ function is also meromorphic on $s_0$. $g(s_0)$ is a non-zero finite value. If $g^{'}(s_0)=\infty$, this result is contradiction to that the $g(s)$ function is also meromorphic on $s_0$. So,  $g(s_0)\neq\infty$ and  $g^{'}(s_0)\neq\infty$.
\end{proof}
\begin{lemma}
\begin{equation}
TT1(s)=TT2(s)
\end{equation}

In which, $TT1(s)=(a_0+ib_0)\prod^{n}_{j=1}(1-\frac{s_0}{p_{j}})^{\beta_{j}}G_{jp}(s_0)(G(s)\prod^{m}_{l=1}(1-\frac{s}{z_{l}})^{\gamma_{l}}G_{lz}(s))^{'}-
(a+ib)(\prod^{n}_{j=1}(1-\frac{s}{p_{j}})^{\beta_{j}}G_{jp}(s))^{'}
G(s_0)\prod^{m}_{l=1}(1-\frac{s_0}{z_{l}})^{\gamma_{l}}G_{lz}(s_0)$.

$TT2(s)=(\gamma_0(s-s_0)^{\gamma_0-1}g(s)+(s-s_0)^{\gamma_0}g^{'}(s))G(s_0)
\prod^{m}_{l=1}(1-\frac{s_0}{z_{l}})^{\gamma_{l}}G_{lz}(s_0)$.
\end{lemma}

\begin{proof}
$(a_0+ib_0)\frac{\prod^{n}_{j=1}(1-\frac{s_0}{p_{j}})^{\beta_{j}}G_{jp}(s_0)}
{G(s_0)\prod^{m}_{l=1}(1-\frac{s_0}{z_{l}})^{\gamma_{l}}G_{lz}(s_0)}G(s)\prod^{m}_{l=1}((1-\frac{s}{z_{l}})^{\gamma_{l}}G_{lz}(s))-
(a+ib)\prod^{n}_{j=1}((1-\frac{s}{p_{j}})^{\beta_{j}}G_{jp}(s))=(s-s_0)^{\gamma_0}g(s)$.

Differentiate this equation, obtain:

$(a_0+ib_0)\frac{\prod^{n}_{j=1}(1-\frac{s_0}{p_{j}})^{\beta_{j}}G_{jp}(s_0)}
{G(s_0)\prod^{m}_{l=1}(1-\frac{s_0}{z_{l}})^{\gamma_{l}}G_{lz}(s_0)}
(G(s)\prod^{m}_{l=1}((1-\frac{s}{z_{l}})^{\gamma_{l}}G_{lz}(s)))^{'}-(a+ib)
(\prod^{n}_{j=1}((1-\frac{s}{p_{j}})^{\beta_{j}}G_{jp}(s)))^{'}
=\gamma_0(s-s_0)^{\gamma_0-1}g(s)+(s-s_0)^{\gamma_0}g^{'}(s)$.

Here, we don't consider the repeated zeros and poles of the root locus equation (1.1). So, on the finite point $s_0$, the zeros factor ${G(s_0)\prod^{m}_{l=1}(1-\frac{s_0}{z_{l}})^{\gamma_{l}}G_{lz}(s_0)}$ of the root locus equation (1.1) surely obtains the non-zero value. At two sides of the above equation, we multiply the zeros factor ${G(s_0)\prod^{m}_{l=1}(1-\frac{s_0}{z_{l}})^{\gamma_{l}}G_{lz}(s_0)}$ meanwhile. Obtain: $(a_0+ib_0)\prod^{n}_{j=1}(1-\frac{s_0}{p_{j}})^{\beta_{j}}G_{jp}(s_0)(G(s)\prod^{m}_{l=1}(1-\frac{s}{z_{l}})^{\gamma_{l}}G_{lz}(s))^{'}-
(a+ib)(\prod^{n}_{j=1}(1-\frac{s}{p_{j}})^{\beta_{j}}G_{jp}(s))^{'}
G(s_0)\prod^{m}_{l=1}(1-\frac{s_0}{z_{l}})^{\gamma_{l}}G_{lz}(s_0)=(\gamma_0(s-s_0)^{\gamma_0-1}g(s)+(s-s_0)^{\gamma_0}g^{'}(s))G(s_0)
\prod^{m}_{l=1}(1-\frac{s_0}{z_{l}})^{\gamma_{l}}G_{lz}(s_0)$.

Let $TT1(s)=(a_0+ib_0)\prod^{n}_{j=1}(1-\frac{s_0}{p_{j}})^{\beta_{j}}G_{jp}(s_0)(G(s)\prod^{m}_{l=1}(1-\frac{s}{z_{l}})^{\gamma_{l}}G_{lz}(s))^{'}-
(a+ib)(\prod^{n}_{j=1}(1-\frac{s}{p_{j}})^{\beta_{j}}G_{jp}(s))^{'}
G(s_0)\prod^{m}_{l=1}(1-\frac{s_0}{z_{l}})^{\gamma_{l}}G_{lz}(s_0)$.

$TT2(s)=(\gamma_0(s-s_0)^{\gamma_0-1}g(s)+(s-s_0)^{\gamma_0}g^{'}(s))G(s_0)
\prod^{m}_{l=1}(1-\frac{s_0}{z_{l}})^{\gamma_{l}}G_{lz}(s_0)$. Utilizing two new expression, we  obtain Eq(3.2).
\end{proof}

\begin{lemma}
Let $s_0=(\sigma_0+it_0)\in\Delta $. All infinite products of the root locus equation (1.1) (if exist) are all absolutely convergent. If the finite point $s_0$ satisfies the equation (3.1). Then, $TT1(s_0)=0$ and $\gamma_0(s_0-s_0)^{\gamma_0-1}g(s_0)+(s_0-s_0)^{\gamma_0}g^{'}(s_0)=0$.

\end{lemma}

\begin{proof}
When we bring the finite point $s_0$ into (3.1), according to the known condition, we can obtain the following result.

$\frac{dK(s_0)}{ds}
=(a_0+ib_0)\frac{(\prod^{n}_{j=1}(1-\frac{s}{p_{j}})^{\beta_{j}}G_{jp}(s))^{'}\mid{_{s=s_0}}
G(s_0)\prod^{m}_{l=1}((1-\frac{s_0}{z_{l}})^{\gamma_{l}}G_{lz}(s_0))-
\prod^{n}_{j=1}(1-\frac{s_0}{p_{j}})^{\beta_{j}}G_{jp}(s_0)
}{(G(s_0)\prod^{m}_{l=1}(1-\frac{s_0}{z_{l}})^{\gamma_{l}}G_{lz}(s_0))^{2}}$

$\frac{(G(s)\prod^{m}_{l=1}((1-\frac{s}{z_{l}})^{\gamma_{l}}G_{lz}(s)))^{'}\mid{_{s=s_0}}}{}
=0$.
According to the result of $\frac{dK(s_0)}{ds}$, when we substitute the finite point $s_0$  into the left side of Eq(3.2), the result of the left side of (3.2) is:
$TT1(s_0)=(a_0+ib_0)(\prod^{n}_{j=1}(1-\frac{s_0}{p_{j}})^{\beta_{j}}G_{jp}(s_0)(G(s)\prod^{m}_{l=1}(1-\frac{s}{z_{l}})^{\gamma_{l}}G_{lz}(s))^{'}\mid{_{s=s_0}}-
(\prod^{n}_{j=1}(1-\frac{s}{p_{j}})^{\beta_{j}}G_{jp}(s))^{'}\mid{_{s=s_0}}
G(s_0)\prod^{m}_{l=1}(1-\frac{s_0}{z_{l}})^{\gamma_{l}}G_{lz}(s_0))=0$

So, we obtain the left side of (3.2) is equal to 0.

For the expression of the right side of Eq(3.2), $TT2(s)=(\gamma_0(s-s_0)^{\gamma_0-1}g(s)+(s-s_0)^{\gamma_0}g^{'}(s))G(s_0)
\prod^{m}_{l=1}(1-\frac{s_0}{z_{l}})^{\gamma_{l}}G_{lz}(s_0)$, the factor $G(s)
\prod^{m}_{l=1}(1-\frac{s}{z_{l}})^{\gamma_{l}}G_{lz}(s)$  has no any finite pole, and the finite point  $s_0$  is not zero and pole of the function $G(s)
\prod^{m}_{l=1}(1-\frac{s}{z_{l}})^{\gamma_{l}}G_{lz}(s)$. So, if the right side of (3.2) is equal to 0, this is not because of this factor. $TT2(s_0)=(\gamma_0(s_0-s_0)^{\gamma_0-1}g(s_0)+(s_0-s_0)^{\gamma_0}g^{'}(s_0))G(s_0)
\prod^{m}_{l=1}(1-\frac{s_0}{z_{l}})^{\gamma_{l}}G_{lz}(s_0)=0$, we can remove the factor $G(s_0)
\prod^{m}_{l=1}(1-\frac{s_0}{z_{l}})^{\gamma_{l}}G_{lz}(s_0)$. And we have: $\gamma_0(s_0-s_0)^{\gamma_0-1}g(s_0)+(s_0-s_0)^{\gamma_0}g^{'}(s_0)=0$.
\end{proof}

\begin{lemma}
Let $s_0=(\sigma_0+it_0)\in\Delta $. All infinite products of the root locus equation (1.1) (if exist) are all absolutely convergent. If the finite point $s_0$ satisfies the equation (3.1). Then, it is the finite breakaway point of the root locus of (1.1).

\end{lemma}

\begin{proof}
Lemma 3.5 proves:
$\gamma_0(s_0-s_0)^{\gamma_0-1}g(s_0)+(s_0-s_0)^{\gamma_0}g^{'}(s_0)=0$. $\gamma_0>0$. If $\gamma_0>1$, because $g(s_0)\neq\infty$ and $g^{'}(s_0)\neq\infty$. $(\gamma_0(s_0-s_0)^{\gamma_0-1}g(s_0)+(s_0-s_0)^{\gamma_0}g^{'}(s_0))=0$.

$(\gamma_0(s_0-s_0)^{\gamma_0-1}g(s_0)+(s_0-s_0)^{\gamma_0}g^{'}(s_0))G(s_0)
\prod^{m}_{l=1}(1-\frac{s_0}{z_{l}})^{\gamma_{l}}G_{lz}(s_0)=0$.

The right side of (3.2) is equal to the left side of (3.2). $\gamma_0>1$, (3.2) holds. This proves that the point $s_0$ is the repeated roots of (1.1), and the finite point $s_0$ is the finite breakaway point of (1.1).

If $1>\gamma_0>0$,  $g^{'}(s_0)\neq\infty$, $(s_0-s_0)^{\gamma_0}g^{'}(s_0)
=0$. $\gamma_0(s_0-s_0)^{\gamma_0-1}=\infty$, $g(s_0)\neq 0$. $(\gamma_0(s_0-s_0)^{\gamma_0-1}g(s_0)+(s_0-s_0)^{\gamma_0}g^{'}(s_0))=\infty$.

$(\gamma_0(s_0-s_0)^{\gamma_0-1}g(s_0)+(s_0-s_0)^{\gamma_0}g^{'}(s_0))G(s_0)
\prod^{m}_{l=1}(1-\frac{s_0}{z_{l}})^{\gamma_{l}}G_{lz}(s_0)=\infty$.

The left side of (3.2) is equal to 0. The right side of (3.2) is not equal to 0. (3.2) is not true. So this contradiction is generated. So, the situation $1>\gamma_0>0$ cannot hold.

If $\gamma_0=1$,  $g^{'}(s_0)\neq\infty$, $(s_0-s_0)^{\gamma_0}g^{'}(s_0)
=0$. $g(s_0)\neq 0$. $(g(s_0)+(s_0-s_0)^{\gamma_0}g^{'}(s_0))=g(s_0)\neq 0$.

$(g(s_0)+(s_0-s_0)^{\gamma_0}g^{'}(s_0))G(s_0)
\prod^{m}_{l=1}(1-\frac{s_0}{z_{l}})^{\gamma_{l}}G_{lz}(s_0)=g(s_0)G(s_0)
\prod^{m}_{l=1}(1-\frac{s_0}{z_{l}})^{\gamma_{l}}G_{lz}(s_0)\neq 0$. The left side of (3.2) is equal to 0. The right side of (3.2) is not equal to 0. (3.2) doesn't hold. Hence, this contradiction is generated. Therefore, the situation $\gamma_0=1$ cannot hold.

Only one situation $\gamma_0>1$ can be reasonable. The finite point $s_0$ is a multiple root of the characteristic equation of (1.1). The finite point $s_0$ is a finite breakaway point of the root locus of (1.1).
\end{proof}

\begin{lemma}
Let $s_0=(\sigma_0+it_0)\in\Delta $. All infinite products of the root locus equation (1.1) (if exist) are all absolutely convergent. Then the finite breakaway point $s_0$ of the root locus of (1.1) must satisfy the derivative equation (3.1).
\end{lemma}
\begin{proof}
According to Lemma 3.2 and Lemma 3.3, and according to the results in the previous in the proof of Lemma 3.4,  we can obtain:

$(a_0+ib_0)\frac{\prod^{n}_{j=1}(1-\frac{s_0}{p_{j}})^{\beta_{j}}G_{jp}(s_0)}
{G(s_0)\prod^{m}_{l=1}(1-\frac{s_0}{z_{l}})^{\gamma_{l}}G_{lz}(s_0)}
(G^{'}(s)\prod^{m}_{l=1}((1-\frac{s}{z_{l}})^{\gamma_{l}}G_{lz}(s))+
G(s)(\prod^{m}_{l=1}((1-\frac{s}{z_{l}})^{\gamma_{l}}G_{lz}(s)))^{'})-
(a+ib)(\prod^{n}_{j=1}((1-\frac{s}{p_{j}})^{\beta_{j}}G_{jp}(s)))^{'}
=\gamma_0(s-s_0)^{\gamma_0-1}g(s)+(s-s_0)^{\gamma_0}g^{'}(s)$. $g(s_0)\neq\infty$ and  $g^{'}(s_0)\neq\infty$.

Substituting $s_0$ into the right side of the above equation, because the finite point $s_0$ is the finite breakaway point of (1.1), it is the repeated root of the characteristic equation. So, $\gamma_0>1$. $\gamma_0(s_0-s_0)^{\gamma_0-1}g(s_0)+(s_0-s_0)^{\gamma_0}g^{'}(s_0)=0$.
So, $(a_0+ib_0)\frac{\prod^{n}_{j=1}(1-\frac{s_0}{p_{j}})^{\beta_{j}}G_{jp}(s_0)}
{G(s_0)\prod^{m}_{l=1}(1-\frac{s_0}{z_{l}})^{\gamma_{l}}G_{lz}(s_0)}
(G^{'}(s)\mid{_{s=s_0}}\prod^{m}_{l=1}((1-\frac{s_0}{z_{l}})^{\gamma_{l}}G_{lz}(s_0))+
G(s_0)(\prod^{m}_{l=1}((1-\frac{s}{z_{l}})^{\gamma_{l}}G^{'}_{lz}(s))))^{'}\mid{_{s=s_0}}-
(a_0+ib_0)(\prod^{n}_{j=1}((1-\frac{s}{p_{j}})^{\beta_{j}}G_{jp}(s)))^{'}\mid{_{s=s_0}}=0$.
So, rearrange it, we can obtain the following result,
$(a_0+ib_0)(G^{'}(s)\mid{_{s=s_0}}\prod^{m}_{l=1}(1-\frac{s_0}{z_{l}})^{\gamma_{l}}G_{lz}(s_0)
\prod^{n}_{j=1}(1-\frac{s_0}{p_{j}})^{\beta_{j}}G_{jp}(s_0)+
G(s_0)\prod^{n}_{j=1}(1-\frac{s_0}{p_{j}})^{\beta_{j}}G_{jp}(s_0)
(\prod^{m}_{l=1}(1-\frac{s}{z_{l}})^{\gamma_{l}}G_{lz}(s))^{'}\mid{_{s=s_0}}-
G(s_0)\prod^{m}_{l=1}(1-\frac{s_0}{z_{l}})^{\gamma_{l}}G_{lz}(s_0)
(\prod^{n}_{j=1}(1-\frac{s}{p_{j}})^{\beta_{j}}G_{jp}(s))^{'}\mid{_{s=s_0}})
=0$.
Derivative the gain function, and substituting the point $s_0$ into the differential expression. Obtain:

$\frac{dK(s_0)}{ds}
=(a_0+ib_0)\frac{(\prod^{n}_{j=1}(1-\frac{s}{p_{j}})^{\beta_{j}}G_{jp}(s))^{'}\mid{_{s=s_0}}
G(s_0)\prod^{m}_{l=1}((1-\frac{s_0}{z_{l}})^{\gamma_{l}}G_{lz}(s_0))-
\prod^{n}_{j=1}(1-\frac{s_0}{p_{j}})^{\beta_{j}}G_{jp}(s_0)
}{(G(s_0)\prod^{m}_{l=1}(1-\frac{s_0}{z_{l}})^{\gamma_{l}}G_{lz}(s_0))^{2}}$

$\frac{(G(s)\prod^{m}_{l=1}((1-\frac{s}{z_{l}})^{\gamma_{l}}G_{lz}(s)))^{'}\mid{_{s=s_0}}}{}
=0$.
\end{proof}

Lemma 3.6 proves the sufficient result of the breakaway points. Lemma 3.7 proves the necessary result. Sum up two lemmas, we can obtain the sufficient and necessary result of the breakaway points. Theorem 3.8 is the coalesce of two results. Here, we only give Theorem 3.6, and needn't to give its proof.

\begin{theorem}
Let $s_0=(\sigma_0+it_0)\in\Delta $. All infinite products of the root locus equation (1.1) (if exist) are all absolutely convergent. Then the finite point $s_0$ is a finite breakaway point of the root locus of (1.1) if and only if the finite point $s_0$ satisfies (3.1).
\end{theorem}

On the finite breakaway points, the several same degree number root locus intersect. Then, they separate from the finite breakaway points again, the different degree numbers root locus cannot intersect.  The zeros and poles of the root locus equation all cannot have this property, so, the finite breakaway points of the root locus are different from the repeated zeros of meromorphic function $W(s)$. The repeated zeros of meromorphic function $W(s)$ cannot become the true finite breakaway points of the root locus.

\begin{theorem}
Let $s=(\sigma+it)\in\Delta$. For the gain function $K(s)$ in (1.1) and the meromorphic function $W(s)$, the zero factors of their derivatives are complete same. So, their derivatives have the same zeros. All finite zeros of $K^{'}(s)$ and $W^{'}(s)$ are complete the same as the finite breakaway points of the root locus of (1.1).
\end{theorem}
\begin{proof}
In the proof of Lemma 3.5 and Lemma 3.7, we give the expression of $\frac{dK(s)}{ds}$. So, here, we no longer give it again. For meromorphic function $W(s)$. Its derivative function is:

$\frac{dW(s)}{ds}=$

$\frac{(G(s)\prod^{m}_{l=1}(1-\frac{s}{z_{l}})^{\gamma_{l}}G_{lz}(s))^{'}
\prod^{n}_{j=1}((1-\frac{s}{p_{j}})^{\beta_{j}}G_{jp}(s))
-G(s)\prod^{m}_{l=1}((1-\frac{s}{z_{l}})^{\gamma_{l}}G_{lz}(s))
(\prod^{n}_{j=1}(1-\frac{s}{p_{j}})^{\beta_{j}}G_{jp}(s))^{'}}
{(\prod^{n}_{j=1}(1-\frac{s}{p_{j}})^{\beta_{j}}G_{jp}(s))^{2}}$.

Obviously, except the constant factor $a+ib$, for two expressions of $\frac{dK(s)}{ds}$ and $\frac{dW(s)}{ds}$, their zeros factors can be considered as the complete same.
So, we prove the first result of this theorem. $-(a+ib)$ is the value on the unit circle. This value must not be 0.
Hence, according to the two results, we prove that the zeros of derivatives of the gain function $K(s)$ and the meromorphic function $W(s)$ in (1.1) are the complete same. According to  Theorem 3.8, the finite zeros of derivative of the gain function $K(s)$ are all finite breakaway points of the root locus of the root locus equation (1.1). Therefore, we prove that these same finite zeros of derivatives are the finite breakaway points of the root locus of the root locus equation (1.1).
\end{proof}
In this paper, the repeated zeros and poles cannot be as the breakaway points of (1.1), so, we exclude them in this lemma.
\begin{theorem}
For any segment of all root locus of arbitrary degree number of (1.1), they cannot be overlapping in the finite position of $\mathbb{C}\cup\{\infty\}$.
\end{theorem}

\begin{proof}
The finite breakaway point is intersecting point of the root locus of the root locus equation (1.1). The overlapping segment is continuous intersecting points of the root locus of (1.1). If the overlapping segments of the root locus of (1.1) exist, they should be special intersecting points, and they are the infinite number of continuous intersection points. To compute the finite breakaway point and to prove their existence, we need to utilize the formula that is used to compute the finite breakaway point. (3.1) is a necessary and sufficient condition of computing the finite breakaway point.

$\frac{dK(s)}{ds}
=(a+ib)\frac{(\prod^{n}_{j=1}(1-\frac{s}{p_{j}})^{\beta_{j}}G_{jp}(s))^{'}
G(s)\prod^{m}_{l=1}((1-\frac{s}{z_{l}})^{\gamma_{l}}G_{lz}(s))-
\prod^{n}_{j=1}(1-\frac{s}{p_{j}})^{\beta_{j}}G_{jp}(s)
}{(G(s)\prod^{m}_{l=1}(1-\frac{s}{z_{l}})^{\gamma_{l}}G_{lz}(s))^{2}}$

$\frac{(G(s)\prod^{m}_{l=1}((1-\frac{s}{z_{l}})^{\gamma_{l}}G_{lz}(s)))^{'}}{}$.

All finite breakaway points are the finite zeros of derivative of the function $W(s)$. The derivative function of the function $W(s)$ is meromorphic function, it maybe has the infinite number of the finite zeros. The finite zeros of the meromorphic function must be isolated to distributed in $\mathbb{C}\cup\{\infty\}$, namely these finite zeros cannot continuously distribute. These finite zeros cannot constitute a continuous curve. So, although there maybe exist the infinite number of the finite breakaway points of Eq(3.1), but, these finite breakaway points are all isolated, and cannot constitute a continuous curve. According to the previous proof, so, any segment of all root locus of arbitrary degree number cannot be overlapping.
\end{proof}
\begin{theorem}
Let $s=(\sigma+it)\in\Delta_d $. For all of the given arbitrary degree root locus of (1.1), they must be originated from poles of (1.1), continuously extend and cannot be interrupted to end at zeros of (1.1). The root locus constitute the continuous curves which have the heads and also have the tails in $\mathbb{C}\cup\{\infty\}$. The poles and zeros of (1.1) can be finite or infinite.
\end{theorem}
\begin{proof}
Theorem 2.16 proves the result which the root locus of (1.1) are continuous. Lemma 1.13 proves: The root locus are originated from poles, they are received by zeros. The zeros and poles can be looked as the end-points of the root locus. Because the root locus are surely continuous, the root locus surely constitute the continuous curves which have the heads and also have the tails in $\mathbb{C}\cup\{\infty\}$.
\end{proof}

Some meromorphic functions $W(s)$ cannot have the signification in the whole $\mathbb{C}\cup\{\infty\}$. Such as Riemann zeta series and Dirichlet ate series which we give their results in the ending of this paper. They cannot have the significance in the whole $\mathbb{C}\cup\{\infty\}$. They both have their definition domains which let themselves have the signification. For a meromorphic functions $W(s)$, in the definition domain which lets itself have the signification, these situations are possible. In the definition domain which can let itself have the signification, there don't exist its poles or zeros. Even its zeros and poles all cannot exist.

But, the result of Theorem 3.11 is suitable to use in the definition domain which lets itself have the signification. Although the root locus of meromorphic functions $W(s)$ probably have no heads or have no endings in the definition domain which lets itself have the signification. Even heads and endings of the root locus of meromorphic functions $W(s)$  all don't exist. But, the continuity of the root locus, the root locus need to be originated from poles. The root locus need to end at zeros. These properties of the root locus of meromorphic functions $W(s)$ are all true in the definition domain which lets itself have the signification.

The turning points of the monotonically increase and monotonically decrease of the degree numbers of the root locus of (1.1) have the special significance, they are the points which have the special values in the root locus of the root locus equation (1.1). Here, we need to give these points a specialize name, we call them as the demarcation points of (1.1). Of course, they are also called as the demarcation points of the meromorphic function $W(s)$.

\begin{definition}
The zeros of one of the  partial differential equations $\frac{\partial{\varphi(\sigma,t)}}{\partial{t}}=0$ are called as the demarcation points of the imaginary part variable of the root locus equation (1.1).

The zeros of one of the  partial differential equations $\frac{\partial{\varphi(\sigma,t)}}{\partial{\sigma}}=0$ are called as the demarcation points of the real part variable of the root locus equation (1.1).
\end{definition}
"The common root locus" is an important concept and tool in the root locus. In the beginning of this section, we have already given it a specific definition. But, "the common root locus" before and hind the breakaway point have the other difference and action, so, we should still give the more detail distinguishing and description. Here, we divide "the common root locus" into two segments. The two segments are separately as: "the beginning segment common root locus" and "the ending segment common root locus". "The beginning segment common root locus" means "the common root locus" are from the origination points of the root locus to the finite breakaway points. "The ending segment common root locus" means "the common root locus" are from the ending points of the root locus to the finite breakaway points.

Namely, "the beginning segment common root locus" means "the common root locus" are from the poles of the root locus equation to the finite breakaway points. "The ending segment common root locus" means "the common root locus" are from the zeros of the root locus equation to the finite breakaway points. The poles and zeros all exist two categories which they distributed on the finite position and the infinite points in $\mathbb{C}\cup\{\infty\}$. So, "the beginning segment common root locus" and "the ending segment common root locus" both have the situation or possibility which they connect to the infinite points in $\mathbb{C}\cup\{\infty\}$.

\begin{definition}
The root locus segments from the poles of (1.1) to the breakaway point are called "beginning segment common root loci". The root locus segments from the zeros of (1.1) to the breakaway point are called "ending segment common root loci".
\end{definition}

\begin{lemma}
At finite positions of the complex plane, root loci starting from two distinct poles come into contact. The two root loci starting from the two distinct poles intersect at a finite breakaway point. Then, the two intersected root locus have same one degree number. The two "ending segment common root loci" leave at the finite breakaway point $P$ from two sides of the two "beginning segment common root loci".
\end{lemma}
\begin{proof}
The two  finite poles of (1.1) are as: $p_1$, $p_2$, two intersected root locus which are emitted from two  finite poles are as: $l_{1}$, $l_{2}$, and they intersect at the finite point $P$. The part which is from pole $p_1$ to the point $P$ is as: $l_{11}$. The part which is from pole $p_2$ to the point $P$ is as: $l_{12}$. The root locus $l_{11}$ and $l_{12}$ are "beginning segment common root locus". The part which is from  the point $P$ to the ending point of an intersected root loci is as: $l_{21}$. The part which is from  the point $P$ to the ending point of another intersected root loci is as: $l_{22}$. The root locus $l_{21}$ and $l_{22}$ are "ending segment common root locus".

Theorem 2.18 and Theorem 2.19 prove: for two intersected root locus, their degrees number can have the difference of the integer number times of $2\pi$ or their degrees number can be same. If the root loci $l_1$ and the root loci $l_2$ have same one degree number. We have already proved the result which the intersected root locus have the same one degree in this lemma.

The phase angle of $W(s)$ is as $\varphi{(\sigma,t)}$.  $\varphi{(\sigma,t)}$ is the degree number of the root locus, and it is a continue function of a real value which have two real variables $\sigma$, $t$. In the whole complex plane, except the infinite points, except zeros and poles of Eq(1.1), the degree number $\varphi{(\sigma,t)}$ of the root locus are all continue on all points of the root locus. The root locus of the root locus equation fully fill the whole extended complex plane, the step phenomenon cannot appear in the degree number of the root locus.

If the difference of degree number of the root loci $l_1$ and the root loci $l_2$  has the integer number times of $2\pi$. Namely, if the degree number of the root loci $l_1$ is $\alpha$, then, the degree number of the root loci $l_2$ is $\alpha+2\pi n_1$. $n_1$ is not equal 0 and $n_1$ is a integer number. So, the continuity of the function $\varphi{(\sigma,t)}$ cannot hold. On the root locus $l_1$ and $l_2$, there cannot exist the poles and zeros of (1.1), the intersected point $P$ and the root locus ($l_1$, $l_2$) are all at the finite position of the extended complex plane. So,  the root locus $l_1$ and $l_2$ have same one degree number. The root locus $l_1$ and $l_2$ are the root locus of same one degree number $\alpha$.

If the two "ending segment common root locus" ($l_{21}$, $l_{22}$) leave at the finite breakaway point $P$ from same one side of two "beginning segment common root locus" ($l_{11}$, $l_{12}$).

With a line segment, join the poles $p_1$ and $p_2$. Take the midpoint of the joining segment as the origin of the coordinate system, and construct a two-dimensional coordinate system. If the poles $p_1$ and $p_2$ are not both on the horizontal axis, then in such a coordinate system one of the two poles lies in the upper half-plane and the other must lie in the lower half-plane. Here let the one in the upper half-plane be the pole $p_1$, and the one in the lower half-plane be the pole $p_2$. If they are both on the horizontal axis, let the one on the left be the pole $p_1$ and the one on the right be the pole $p_2$. We first give the proof for the case in which the two "ending segment common root loci" ($l_{21}$, $l_{22}$) extend to the right side of the two "beginning segment common root loci" $l_{1}$, $l_{2}$, or the case in which the two "ending segment common root loci" ($l_{21}$, $l_{22}$) extend into the upper half-plane.

Here, let the root loci be as $l_{13}$. It is originated from the pole $p_1$. And it is adjacent to the "common root loci" $l_{11}$. It is on the clockwise position of the "common root loci" $l_{11}$  of the pole $p_1$. Let its degree number as $\alpha_{1}$. Let the root loci be as $l_{14}$. It is originated from the pole $p_2$. It is adjacent to the "common root loci" $l_{12}$. It is on the counterclockwise position of the "common root loci" $l_{12}$ of the pole $p_2$. Let its degree number as $\alpha_{2}$.

A extremity point of "the beginning segment common root locus" of $l_{11}$ and $l_{12}$  is the pole $p_1$. At the side of "the beginning segment common root locus" of $l_{11}$ and $l_{12}$, this side is on the clockwise direction of "the beginning segment common root locus"  of $l_{11}$ and $l_{12}$ concerning the pole $p_1$. Because there don't exist the "common root locus" at this side. There don't exist any "vacuum" between the root locus ($l_{11}$, $l_{12}$) and the root locus  ($l_{13}$, $l_{14}$).  The root locus of $l_{11}$ and $l_{13}$  are originated from the poles $p_1$. The root locus of $l_{12}$ and $l_{14}$  are originated from the  $p_2$. The root loci  $l_{13}$ along the root loci $l_{11}$ extends to the direction of the root loci  $l_{14}$. The root loci  $l_{14}$ along the root loci $l_{12}$ extends to the direction of the root loci  $l_{13}$. Namely, the two root locus  $l_{13}$ and $l_{14}$ extend towards the opponent's direction. So, the root locus  $l_{13}$ and $l_{14}$ at this side surely meet, and  are surely adjacent. Namely, the root loci $l_{13}$ surely meets to the root loci $l_{14}$ at this side. Because they have the different degree numbers. They cannot intersect to each other. They only meet and are adjacent.

According to the relative position of the root locus $l_{11}$ and $l_{13}$, and because of Theorem 1.22, we can obtain: $\alpha_{1}$$>$$\alpha$. According to the relative position of the root locus  $l_{12}$ and $l_{14}$, we can obtain: $\alpha_{2}$$<$$\alpha$.

So, $\alpha_{1}$$>$$\alpha$$>$$\alpha_{2}$. $\alpha_{1}$ and $\alpha$ are adjacent degree number. $\alpha$ and  $\alpha_{2}$ are adjacent degree number.  $\alpha_{1}$ and  $\alpha_{2}$ are surely not adjacent degree number. Obviously, there exists the degree number $\alpha$ between two degree numbers $\alpha_{1}$ and  $\alpha_{2}$. In the whole complex plane, except the infinite points, except zeros and poles of Eq(1.1), the degree number $\varphi{(\sigma,t)}$ of the root locus are all continue on all points of the root locus. So, if the root locus $l_{13}$ and $l_{14}$ meet and are adjacent, the step phenomenon that the degree number of the root locus cannot be continue appears. So, this possible situation cannot appear.

In the proof of this possible situation, we don't propose the requirement of the position of finite points which the root locus $l_{13}$ and $l_{14}$ meet and are adjacent. Namely, we needn't to require that two finite points which the root locus $l_{13}$ and $l_{14}$ meet and are adjacent are surely adjacent to the finite breakaway point $P$. So, the root locus $l_{13}$ and $l_{14}$ meet and are adjacent at an arbitrary finite point in the extended complex plane, this is not true.

When the two "ending segment common root loci" ($l_{21}$, $l_{22}$) extend to the left side of the two "beginning segment common root loci" $l_{1}$, $l_{2}$, or when the two "ending segment common root loci" ($l_{21}$, $l_{22}$) extend into the lower half-plane, repeating the preceding proof proves this lemma.

\end{proof}

At finite positions of the complex plane, root loci starting from several distinct poles come into contact, and several root loci of the same degree among them intersect at a finite breakaway point. Then these intersected root loci have the same degree number. Every "ending segment common root locus" leaves from the finite breakaway point $P$ between every two adjacent "beginning segment common root loci". If this result did not hold, the following would occur: there would exist two adjacent "beginning segment common root loci" from which, at the finite breakaway point, no "ending segment common root locus" extends out. Such a situation is the same as the situation of Lemma 3.14.

By repeating the proof of Lemma 3.14, we can prove Theorem 3.15.

\begin{theorem}
At finite positions of the complex plane, the root loci originating from several distinct poles come into contact, and several of the root loci of the same degree among them intersect at a finite breakaway point. Then, these intersected root locus have same one degree number. The every one "ending segment common root locus" leaves from the finite breakaway point $P$ between every two adjacent "beginning segment common root loci".
\end{theorem}

The zeros which satisfy the former equation are the demarcation points which the imaginary part variable $t$ is as the variable. The zeros which satisfy the latter equation are the demarcation points which the real part variable $\sigma$ is as the variable. Namely, the points which the partial derivatives of the phase angle of the root locus are equal to 0 are the demarcation points of the root locus equation (1.1).

For the demarcation points of meromorphic function $W(s)$, we have: the points which the partial derivatives of the phase angle of meromorphic function $W(s)$ are equal to 0 are the demarcation points of meromorphic function $W(s)$.

At their two sides of the finite demarcation points, the monotonic of the degree numbers of the root locus generate the complete converse change. Two same degree number "common root locus" intersect at the finite breakaway points. After they intersect latter, and then they leave again. So, around the finite breakaway points, the shape which four "common root locus" intersect at a finite point forms. For the four "common root locus", the two "common root locus" which are from poles to the finite breakaway points must appear, the two "common root locus" which are from the finite breakaway points to zeros must appear. Every two adjacent "common root locus" of four "common root locus" separately form themselves a corner at the finite breakaway point.

The different degree numbers root locus cannot intersect. The others approach degree numbers root locus all distributed along the four "common root locus". The others approach degree numbers root locus all distributed along every a corner which every two adjacent "common root locus" form. Namely, around the finite breakaway point, the others approach degree numbers root locus distributed along four "common root locus". And they are "parallel" to the four "common root locus". Their shapes are almost same as the four "common root locus". Because a corner itself has the features of geometric figures. Around the finite breakaway point, the others approach degree numbers root locus all need to turn around to distributed along the four "common root locus". Because these different degree numbers root locus turn around to distributed. Such distributed feature let the root locus form the following distributed rule around the finite breakaway point.

We do two straight lines, the two straight lines are used to analytic the change rule of degree numbers of the root locus around the finite breakaway point. The two straight lines are separately horizontal line and vertical line. We analytic the change rule of degree numbers of the root locus on the horizontal line or vertical line. If the degree numbers of the root locus which enter into the corner that the "common root locus" form increase gradually, then, the degree numbers of these root locus which turn out from this corner decrease gradually. Conversely, if the degree numbers of the root locus which enter into the corner that the "common root locus" form decrease gradually, then, the degree numbers of these root locus which turn out from this corner increase gradually.

But, no matter which one rule in the above two rules appears. In every corner that "the common root locus" form around the finite breakaway point, the change rule of degree numbers of the approach degree numbers root locus have the converse same feature. So, there exist the balance points which the degree numbers of the root locus cannot change in the zone which the degree numbers of the root locus change from the increasing gradually into the decreasing gradually. A converse situation is as: there exist the finite points which the change values of degree numbers of the root locus are equal to 0 in the zone which the degree numbers of the root locus change from the decreasing gradually into the increasing gradually. When the two horizontal line and vertical line continuously are done, the continuous curves of demarcation points can be obtained. We have already given the result of the single finite breakaway point. For the finite repeated breakaway points, namely more than two "common root locus" intersect at one finite point. Their situation is the same as the situation which two "common root locus" intersect at one finite point. Here, we needn't to give them.

\begin{lemma}
If there exist $l$ same degree number root locus that they intersect at one finite point, in which there is at least one "common root locus". When the finite breakaway point acts as one beginning point of this one "common root loci" line segment. This "common root loci" line segment is paralleled to the imaginary axis. When the imaginary part $t$ acts as a variable, all points on the "common root loci" line segment which is paralleled to the imaginary axis are all demarcation points of Eq(1.1).
\end{lemma}

\begin{proof}
By the condition of this lemma, when the finite breakaway point acts as its beginning point of one "common root loci" line segment. This root loci line segment is paralleled to the imaginary axis. (The result and proof are both suitable to the special case that the finite breakaway point and the root loci line segment are all on the imaginary axis.)

The "common root loci" line segment has the same degree number since it is a part of a complete root loci. And it is paralleled to the imaginary axis. The value of the phase angle function $\varphi(\sigma,t)$ on the full line segment of the "common root loci" obtains a same degree number. The degree number of the "common root loci" is a constant about the variable $t$. So, when the imaginary part $t$ acts as a variable, all points on the "common root loci" line segment which is paralleled to the imaginary axis are all demarcation points of (1.1).
\end{proof}

Passing the finite breakaway point, do a straight line parallel to the imaginary axis. We call this straight line as the first straight line $B1$.

In the following proof it is necessary to distinguish clearly between two different concepts: the root locus on the left side of the straight line $B1$ and the root locus on the right side of the straight line $B1$. The "beginning segment common root loci" all terminate at the breakaway point. The "ending segment common root loci" all begin at the breakaway point. Therefore, the "beginning segment common root loci" and the "ending segment common root loci" must all be connected with the breakaway point. If, after a "beginning segment common root locus" leaves the breakaway point, the segments of it closest to the breakaway point all lie on the left side of the straight line $B1$, then this "beginning segment common root locus" is considered to lie on the left side of the straight line $B1$. The earlier part of this root locus may lie on the right side of the straight line $B1$ and then cross the straight line $B1$ and enter the left side of the straight line $B1$. If, after an "ending segment common root locus" leaves the breakaway point, the segments of it closest to the breakaway point all lie on the left side of the straight line $B1$, then this "ending segment common root locus" is considered to lie on the left side of the straight line $B1$. This root locus may, as it continues to extend afterwards, cross the straight line $B1$ and enter the right side of the straight line $B1$.

If, after a "beginning segment common root locus" leaves the breakaway point, the segments of it closest to the breakaway point all lie on the right side of the straight line $B1$, then this "beginning segment common root locus" is considered to lie on the right side of the straight line $B1$. The earlier part of this root locus may lie on the left side of the straight line $B1$ and then cross the straight line $B1$ and enter the right side of the straight line $B1$. If, after an "ending segment common root locus" leaves the breakaway point, the segments of it closest to the breakaway point all lie on the right side of the straight line $B1$, then this "ending segment common root locus" is considered to lie on the right side of the straight line $B1$. This root locus may afterwards cross the straight line $B1$ and enter the left side of the straight line $B1$.

\begin{lemma}
If there exist $l$ same degree number root locus that they intersect at one finite point. Among all of root locus segments of $2l$ "common root locus", there doesn't exist any one "common root locus" segment that comes from the finite breakaway point, it is paralleled to the imaginary axis. Under these conditions, then, inside any small adjacent domain of the finite breakaway point, when the imaginary part $t$ acts as a variable, there surely exists at least one demarcation point of (1.1).
\end{lemma}

\begin{proof}
The function $W(s)$ which is studied in this paper is the meromorphic function. The finite zeros of the meromorphic function are isolate. Therefore, around every finite breakaway point of (1.1), there surely exists an adjacent domain which there don't exist the finite zeros of (1.1). Let the adjacent domain be $Q_1$.

The function $1/W(s)$ is also meromorphic function. The finite zeros of the meromorphic function are isolate. The finite poles of the meromorphic function $W(s)$ are the zeros of the meromorphic function $1/W(s)$. So, the finite poles of the meromorphic function $W(s)$ are isolate. Therefore, around every finite breakaway point of (1.1), there surely exist an adjacent domain which there don't exist the finite poles of (1.1). Let the adjacent domain be $Q_2$.

We take the intersection set of the above two adjacent domains $Q_1$ and $Q_2$, namely, the new adjacent domain of the finite breakaway point is as: $Q=Q_1\cap Q_2$. The new adjacent domain $Q$ has no zeros and poles of (1.1). According to the condition of this lemma, there exist a finite breakaway point and $2l$ "common root locus" in the adjacent domain $Q$.

In the adjacent domain $Q$, if there are $g$ segments of "the common root locus" at the left side of the first straight line $B1$. Then, at the right side of the first straight line $B1$, there at least exist $2l-g$ segments of "the common root locus".

When $g<2$, at the right side of the first straight line $B1$, there at least exist $2l-g$ segments of "the common root locus", $g<2$, $-g>-2$, $l\ge 2$, $2l \ge 4$, so, $2l-g> 2$. Do a straight line parallel to the imaginary axis at the right side of the first straight line $B1$, and it is called as the second straight line $B2$.

When $g\ge2$, at the left side of the first straight line $B1$, do a second straight line $B2$ parallel to the imaginary axis.

For the quantity $g$ which obtains the arbitrary natural number, according to the method which we do the second straight line, on the second straight line $B2$ inside the adjacent domain $Q$, no matter whether the second straight line $B2$ distributed at the left side of the first straight line or the right side of the first straight line. At least two intersection points can be obtained on the second straight line $B2$ in the adjacent domain $Q$. They are the intersection points of "the common root locus" and the second straight line $B2$. The phase angles of every two adjacent intersection points on the second straight line $B2$ are both same.

The phase angle function $\varphi(\sigma,t)$ is a real values function of two real variables. And it can use Rolle's theorem to prove results we need here. On the second straight line $B2$ we do, there are at least two points whose their phase angle are same. On the second straight line $B2$, the real variable $\sigma$ should be considered as a constant while the imaginary variable $t$ should be considered as a variable. According to Rolle's theorem in textbooks of mathematics analytic, we can obtain that there exists at least a demarcation point of (1.1) on the second straight line $B2$ in the adjacent domain $Q$. When the imaginary part $t$ is as a variable, this demarcation point is a point that the partial differentiate of phase angle function is 0.

Because the second straight line $B2$ parallels to the imaginary axis can arbitrarily and sufficiently be close to the finite breakaway point, based on proof given before, we prove: No matter how small adjacent domain of the finite breakaway point, when the imaginary part $t$ acts as a variable, there surely exists at least one demarcation point of (1.1).
\end{proof}
Sum up Lemma 3.16 and Lemma 3.17, we can obtain Theorem 3.18.
\begin{theorem}
If there exist $l$ same degree number root locus that they intersect at one finite point. Inside any small adjacent domain of the finite breakaway point, when the imaginary part $t$ acts as a variable and the real part $\sigma$ acts as a constant, there surely exists at least one demarcation point of (1.1).
\end{theorem}

Using the complete same proof with Lemma 3.16 and Lemma 3.17, select and do two straight lines parallel to the real axis, separately as the first straight line and the second straight line. we can prove the existence of the demarcation points when the real part $\sigma$ acts as a variable and the imaginary part $t$ acts a constant. Here, we only give the total result.

\begin{theorem}
If there exist $l$ same degree number root locus that they intersect at one finite point. Inside any small adjacent domain of the finite breakaway point, when the real part $\sigma$ acts as a variable and the imaginary part $t$ acts as a constant, there surely exists at least one demarcation point of (1.1).
\end{theorem}

At finite positions of the complex plane, the root loci emitted by two poles come into contact. After the root loci have come into contact, at finite positions of the complex plane the root loci of the two poles should produce intersecting "common root loci". The intersecting "common root loci" separate all the other root loci of the two poles. This should be the only correct form of distribution of the root loci after the root loci between two poles have come into contact. However, for this only correct result a justification must be given. That is, for the case in which the root loci of two poles have come into contact, all possible situations of the distribution of the root loci at finite positions of the complex plane must be proved. Intersecting "common root loci" constitute one large class among all possible situations, and the proof for it has already been given above. At finite positions of the complex plane, the root loci emitted by two poles come into contact, and the root loci are only adjacent but do not intersect. This differs from the intersection of root loci, and it is another large class of possible ways in which the root loci may be distributed. Such possible situations can in turn be divided into the following two large classes of possible situations.

The first class of possible situations is: at finite positions of the complex plane, two poles emit two root loci of the same degree; the two root loci of the same degree are adjacent and do not intersect; from the two adjacent points they separate and then continue to extend to the zeros. For the first class of possible situations, since between the two adjacent points there is neither a blank point nor any point on other root loci besides the two points, the two adjacent points of the two root loci of the same degree can be taken as the breakaway point of the intersecting "common root loci". Here the two adjacent finite points are regarded as one point, called the central point and denoted by $bp_{12}$. Except that one finite breakaway point is replaced by two points, the distribution of the root loci is exactly the same as the case in which the root loci intersect at a breakaway point in Lemma 3.16 and Lemma 3.17. Therefore, repeating the proofs of Lemma 3.16, Lemma 3.17, Theorem 3.18 and Theorem 3.19, we obtain:

\begin{theorem}
For the first class of possible situations, inside any small adjacent domain of the central point $bp_{12}$, when the imaginary part $t$ acts as a variable and the real part $\sigma$ acts as a constant, there surely exists at least one demarcation point of the imaginary part variable of the root locus equation (1.1).
\end{theorem}

\begin{theorem}
For the first class of possible situations, inside any small adjacent domain of the central point $bp_{12}$, when the real part $\sigma$ acts as a variable and the imaginary part $t$ acts as a constant, there surely exists at least one demarcation point of the real part variable of the root locus equation (1.1).
\end{theorem}

\begin{lemma}
The first class of possible situations cannot exist.
\end{lemma}
\begin{proof}
Under the condition of this lemma, when the first class of possible situations arises, by Theorem 3.20 and Theorem 3.21:

Inside any small adjacent domain of central point $bp_{12}$, when the imaginary part $t$ acts as a variable and the real part $\sigma$ acts as a constant, there surely exists at least one demarcation point of the imaginary part variable of the root locus equation (1.1). Namely, $\frac{\partial{\varphi}}{\partial{t}}(\sigma,t)\mid{_{s=bp_{12}}}=0$.

Inside any small adjacent domain of central point $bp_{12}$, when the real part $\sigma$ acts as a variable and the imaginary part $t$ acts as a constant, there surely exists at least one demarcation point of the real part variable of the root locus equation (1.1). Namely,  $\frac{\partial{\varphi}}{\partial{\sigma}}(\sigma,t)\mid{_{s=bp_{12}}}=0$ .

By the above two results, when we select the more and more small adjacent domain of central point $bp_{12}$, finally, the adjacent domain converge at central point $bp_{12}$, we can obtain:  $s_0=bp_{12}$. $\frac{\partial{\varphi}}{\partial{t}}(\sigma,t)\mid{_{s=s_0}}=0$ and $\frac{\partial{\varphi}}{\partial{\sigma}}(\sigma,t)\mid{_{s=s_0}}=0$.

By Theorem 2.7, we can obtain  central point $bp_{12}$ is a finite breakaway point. But, central point $bp_{12}$ cannot be the finite breakaway point. Therefore, the first class of possible situations cannot exist.
\end{proof}

At finite positions of the complex plane, if $l$ root loci of the same degree, as in Lemma 3.22, meet at $l$ points but do not intersect, then, since the situation of every adjacent pair is exactly the same as in Lemma 3.22, by Lemma 3.22 none of the pairs can hold. Hence the situation in which $l$ root loci of the same degree meet at $l$ points cannot hold either. Therefore, by Lemma 3.22, we have:

\begin{lemma}
At finite positions of the complex plane, root loci starting from $l$ distinct poles come into contact, and $l$ root loci of the same degree among them are adjacent and do not intersect. Then this possible situation cannot exist.
\end{lemma}

At finite positions of the complex plane, root loci of the same degree do not intersect; they are adjacent. It was proved above that the first class of situations is unreasonable. The second class of possible situations is: at finite positions of the complex plane, two poles emit two root loci of adjacent degree, and the two root loci of adjacent degree are adjacent. From the two adjacent points they separate and then continue to extend to the zeros.

Let the root loci of adjacent degree be $NLR_1$, $NLR_2$. Let the poles emitting the root loci $NLR_1$, $NLR_2$ be $np_1$ or $np_2$. Here adjacent degree means that the degrees of $NLR_1$ and $NLR_2$ satisfy the continuity requirement. Since the degrees of $NLR_1$ and $NLR_2$ are different, $NLR_1$ and $NLR_2$ certainly cannot intersect; however, $NLR_1$ and $NLR_2$ are adjacent. For the second class of possible situations, in order to express matters more clearly in the following proof, the root locus segments of $NLR_1$ and $NLR_2$ in the neighbourhood of the two finite adjacent points are divided into two classes, and these two classes of root locus segments are given special names. A root locus segment of adjacent degree that is emitted from the pole ($np_1$ or $np_2$), extends to one of the finite adjacent points ($ap_1$ or $ap_2$), and lies within a sufficiently small neighbourhood of this finite adjacent point ($ap_1$ or $ap_2$) is called an initial root locus segment of adjacent degree ($BNLR_1$ or $BNLR_2$). A root locus segment of adjacent degree that passes through one of the finite adjacent points ($ap_1$ or $ap_2$), extends to a zero, and lies within a sufficiently small neighbourhood of this finite adjacent point is called a terminal root locus segment of adjacent degree ($ENLR_1$ or $ENLR_2$).

Repeating the proof of Lemma 3.16, we obtain:

\begin{lemma}
At finite positions of the complex plane, the root loci $NLR_1$ and $NLR_2$ starting from two distinct poles ($np_1$ and $np_2$) come into contact. $NLR_1$ and $NLR_2$ are adjacent at the two finite points $ap_1$ and $ap_2$. $NLR_1$ and $NLR_2$ do not intersect. The two finite adjacent points $ap_1$ and $ap_2$ serve as the starting points. $NLR_1$ and $NLR_2$ are divided by the two adjacent points $ap_1$ and $ap_2$ respectively into two root locus segments ($BNLR_1$ and $ENLR_1$; $BNLR_2$ and $ENLR_2$). Such a situation thus produces two initial root locus segments of adjacent degree $BNLR_1$ and $BNLR_2$ and two terminal root locus segments of adjacent degree $ENLR_1$ and $ENLR_2$. When one of the four root locus segments ($BNLR_1$ and $ENLR_1$; $BNLR_2$ and $ENLR_2$) is parallel to the imaginary axis, then, with the imaginary part $t$ acting as a variable, all points on this root locus segment which is parallel to the imaginary axis are all demarcation points of the imaginary part variable of the root locus equation (1.1). 
\end{lemma}
The function $W(s)$ which is studied in this paper is the meromorphic function. The finite zeros of the meromorphic function are isolate. Therefore, around the two adjacent finite points $ap_1$ and $ap_2$, there surely exists an adjacent domain in which the finite zeros of (1.1) do not exist. Let the adjacent domain be $Q_3$.

The function $1/W(s)$ is also meromorphic function. The finite zeros of the meromorphic function are isolate. The finite poles of the meromorphic function $W(s)$ are the zeros of the meromorphic function $1/W(s)$. So, the finite poles of the meromorphic function $W(s)$ are isolate. Therefore, around the two adjacent finite points $ap_1$ and $ap_2$, there surely exist an adjacent domain which there don't exist the finite poles of (1.1). Let the adjacent domain be $Q_4$.

We take the intersection set of the above two adjacent domains $Q_3$ and $Q_4$; namely, the new adjacent domain of the two adjacent finite points $ap_1$ and $ap_2$ is $Q=Q_3\cap Q_4$. The new adjacent domain $Q$ has no zeros and no poles of (1.1). By the condition of this lemma, in the adjacent domain $Q$ there exist the two initial root locus segments of adjacent degree $BNLR_1$ and $BNLR_2$ and the two terminal root locus segments of adjacent degree $ENLR_1$ and $ENLR_2$. At the same time, the two adjacent finite points $ap_1$ and $ap_2$ also exist.

In order to prove Lemma 3.25, one needs to give a classification according to the case in which the two root loci of adjacent degree $NLR_1$ and $NLR_2$ are adjacent at the two finite points $ap_1$ and $ap_2$.

Passing the two adjacent finite points $ap_1$ and $ap_2$ or the central of the two adjacent finite points $ap_1$ and $ap_2$, do a straight line parallel to the imaginary axis. We call this straight line as the first straight line $B3$.

The case in which, on one side of the first straight line $B3$, the initial root locus segment ($BNLR_1$ or $BNLR_2$) and the terminal root locus segment ($ENLR_1$ or $ENLR_2$) of the same root locus of adjacent degree ($NLR_1$ or $NLR_2$) exist simultaneously is taken as the first class of the two root loci of adjacent degree $NLR_1$ and $NLR_2$ being adjacent at the two finite points $ap_1$ and $ap_2$.

If, on neither side of the first straight line $B3$, the initial root locus segment ($BNLR_1$ or $BNLR_2$) and the terminal root locus segment ($ENLR_1$ or $ENLR_2$) of the same root locus of adjacent degree ($NLR_1$ or $NLR_2$) exist simultaneously, then two other situations arise. One situation is: on one side of the first straight line $B3$, the initial root locus segment ($BNLR_1$ or $BNLR_2$) and the terminal root locus segment ($ENLR_1$ or $ENLR_2$) of non-identical root loci of adjacent degree exist simultaneously; on the other side, likewise, the initial root locus segment ($BNLR_1$ or $BNLR_2$) and the terminal root locus segment ($ENLR_1$ or $ENLR_2$) of non-identical root loci of adjacent degree exist simultaneously. This situation is taken as the second class of the two root loci of adjacent degree $NLR_1$ and $NLR_2$ being adjacent at the two finite points $ap_1$ and $ap_2$.

The third situation is: on one side of the first straight line $B3$, the initial root locus segments ($BNLR_1$ and $BNLR_2$) of the two root loci of adjacent degree exist simultaneously; on the other side, the terminal root locus segments ($ENLR_1$ and $ENLR_2$) of the two root loci of adjacent degree exist simultaneously. This situation is taken as the third class of the two root loci of adjacent degree $NLR_1$ and $NLR_2$ being adjacent at the two finite points $ap_1$ and $ap_2$.

\begin{lemma}
At finite positions of the complex plane, the root loci $NLR_1$ and $NLR_2$ starting from two distinct poles ($np_1$ or $np_2$) come into contact. $NLR_1$ and $NLR_2$ are adjacent at the two finite points $ap_1$ and $ap_2$, and $NLR_1$ and $NLR_2$ cannot intersect. With $ap_1$ and $ap_2$ as starting points, the two root loci $NLR_1$ and $NLR_2$ are divided by the two adjacent finite points $ap_1$ and $ap_2$ respectively into two root locus segments ($BNLR_1$ and $ENLR_1$; $BNLR_2$ and $ENLR_2$). Such a situation thus produces two initial root locus segments of adjacent degree $BNLR_1$ and $BNLR_2$ and two terminal root locus segments of adjacent degree $ENLR_1$ and $ENLR_2$. None of the four root locus segments ($BNLR_1$ and $ENLR_1$; $BNLR_2$ and $ENLR_2$) is parallel to the imaginary axis, and the first class of the two root loci of adjacent degree $NLR_1$ and $NLR_2$ being adjacent at the two finite points $ap_1$ and $ap_2$ arises. Then, inside any small adjacent domain of the two adjacent finite points $ap_1$ and $ap_2$, when the imaginary part $t$ acts as a variable, there surely exists at least one demarcation point of the imaginary part variable of the root locus equation (1.1).
\end{lemma}

\begin{proof}
When, on one side of the first straight line $B3$, the initial root locus segment ($BNLR_1$ or $BNLR_2$) and the terminal root locus segment ($ENLR_1$ or $ENLR_2$) of the same root locus of adjacent degree ($NLR_1$ or $NLR_2$) exist simultaneously, then on this side, where the initial root locus segment ($BNLR_1$ or $BNLR_2$) and the terminal root locus segment ($ENLR_1$ or $ENLR_2$) of the same root locus of adjacent degree ($NLR_1$ or $NLR_2$) exist, draw a second straight line $B4$ parallel to the imaginary axis.

By the method with which we draw the second straight line, on this side where the initial root locus segment ($BNLR_1$ or $BNLR_2$) and the terminal root locus segment ($ENLR_1$ or $ENLR_2$) of the same root locus of adjacent degree ($NLR_1$ or $NLR_2$) exist, no matter how close the second straight line $B4$ is to the first straight line $B3$, at least two intersection points can be obtained on the second straight line $B4$ within the adjacent domain $Q$. They are the intersection points of the same root locus of adjacent degree ($NLR_1$ or $NLR_2$) with the second straight line $B4$. The phase angles of the two adjacent intersection points on the second straight line $B4$ are the same.

The phase angle function $\varphi(\sigma,t)$ is a real values function of two real variables. And it can use Rolle's theorem to prove results we need here. On the second straight line $B4$ we do, there are at least two points whose their phase angle are same. On the second straight line $B4$, the real variable $\sigma$ should be considered as a constant while the imaginary variable $t$ should be considered as a variable. By Rolle's theorem, we can obtain that there exists at least a demarcation point of the imaginary part variable of the root locus equation (1.1) on the second straight line $B4$ in the adjacent domain $Q$. When the imaginary part $t$ is as a variable, this demarcation point is a point that the partial differentiate of phase angle function is 0.

Because the second straight line $B4$ parallels to the imaginary axis can arbitrarily and sufficiently be close to the two adjacent finite points $ap_1$ and $ap_2$, based on proof given before, we prove: No matter how small adjacent domain of the two adjacent finite points $ap_1$ and $ap_2$, when the imaginary part $t$ acts as a variable, there surely exists at least one demarcation point of the imaginary part variable of the root locus equation (1.1).
\end{proof}

\begin{lemma}
At finite positions of the complex plane, the root loci $NLR_1$ and $NLR_2$ starting from two distinct poles ($np_1$ or $np_2$) come into contact. $NLR_1$ and $NLR_2$ are adjacent at the two finite points $ap_1$ and $ap_2$. $NLR_1$ and $NLR_2$ cannot intersect. With $ap_1$ and $ap_2$ as starting points, $NLR_1$ and $NLR_2$ are divided by $ap_1$ and $ap_2$ respectively into two root locus segments ($BNLR_1$ and $ENLR_1$; $BNLR_2$ and $ENLR_2$). Such a situation thus produces two initial root locus segments of adjacent degree $BNLR_1$ and $BNLR_2$ and two terminal root locus segments of adjacent degree $ENLR_1$ and $ENLR_2$. None of the four root locus segments ($BNLR_1$ and $ENLR_1$; $BNLR_2$ and $ENLR_2$) is parallel to the imaginary axis, and the second class of the two root loci of adjacent degree $NLR_1$ and $NLR_2$ being adjacent at the two finite points $ap_1$ and $ap_2$ arises. Then, inside any small adjacent domain of the two adjacent finite points $ap_1$ and $ap_2$, when the imaginary part $t$ acts as a variable, there surely exists at least one demarcation point of the imaginary part variable of the root locus equation (1.1).
\end{lemma}

\begin{proof}
On one side of the first straight line $B3$, the initial root locus segment ($BNLR_1$ or $BNLR_2$) and the terminal root locus segment ($ENLR_1$ or $ENLR_2$) of non-identical root loci of adjacent degree ($NLR_1$ or $NLR_2$) exist simultaneously; on the other side of the straight line $B3$, likewise, the initial root locus segment ($BNLR_1$ or $BNLR_2$) and the terminal root locus segment ($ENLR_1$ and $ENLR_2$) of non-identical root loci of adjacent degree ($NLR_1$ or $NLR_2$) exist simultaneously. Under such conditions, on the right side of the first straight line $B3$, there certainly exists an initial root locus segment ($BNLR_1$ or $BNLR_2$) of one root locus of adjacent degree ($NLR_1$ or $NLR_2$). From the pole ($np_1$ or $np_2$) emitting this initial root locus segment ($NLR_1$ or $NLR_2$), root loci $RL_r$ of other degrees are certainly emitted; these root loci $RL_r$ extend along the root locus ($NLR_1$ or $NLR_2$) of this initial root locus segment of adjacent degree ($BNLR_1$ or $BNLR_2$) towards one finite point ($ap_1$ or $ap_2$) of the two adjacent points. The other of the two adjacent finite points ($ap_2$ or $ap_1$) lies on a root locus of adjacent degree ($NLR_1$ and $NLR_2$) coming from the other pole ($np_2$ or $np_1$).

Thus, these root loci $RL_r$ must change their direction of extension in the neighbourhood of the two adjacent finite points ($ap_1$ or $ap_2$). $RL_r$ then extends along the root locus of adjacent degree ($NLR_2$ or $NLR_1$) coming from the other pole ($np_2$ or $np_1$). These root loci $RL_r$ are all root loci emitted from the same pole ($np_1$ or $np_2$). Each $RL_r$ is the same root locus. In the neighbourhood $Q$ of the two adjacent finite points ($ap_2$ or $ap_1$), on the right side of the first straight line $B3$, just as for the two root loci of adjacent degree, these root loci $RL_r$ also have initial root locus segments and terminal root locus segments.

In the adjacent domain $Q$, there simultaneously exist many initial root locus segments and terminal root locus segments of the same root locus $RL_r$. Therefore, repeating the proof of Lemma 3.25 proves: no matter how small the adjacent domain of the finite breakaway point is, when the imaginary part $t$ acts as a variable, there surely exists at least one demarcation point of the imaginary part variable of the root locus equation (1.1).

In the adjacent domain $Q$, these root loci $RL_r$, which are emitted from the same pole ($np_1$ or $np_2$) and extend along the root loci of adjacent degree ($NLR_1$ or $NLR_2$), cannot intersect the root loci coming from the other pole ($np_2$ or $np_1$). Even if the case of intersecting root loci must be considered, after intersection these root loci $RL_r$ must still continue to extend along the root loci of adjacent degree ($NLR_1$ or $NLR_2$). Therefore, the initial root locus segments and the terminal root locus segments of these root loci $RL_r$ must both exist.
\end{proof}

\begin{lemma}
At finite positions of the complex plane, the root loci $NLR_1$ and $NLR_2$ starting from two distinct poles ($np_1$ or $np_2$) come into contact. $NLR_1$ and $NLR_2$ are adjacent at the two finite points $ap_1$ and $ap_2$. $NLR_1$ and $NLR_2$ cannot intersect. With $ap_1$ and $ap_2$ as starting points, $NLR_1$ and $NLR_2$ are divided by $ap_1$ and $ap_2$ respectively into two root locus segments ($BNLR_1$ and $ENLR_1$; $BNLR_2$ and $ENLR_2$). Such a situation thus produces two initial root locus segments of adjacent degree $BNLR_1$ and $BNLR_2$ and two terminal root locus segments of adjacent degree $ENLR_1$ and $ENLR_2$. None of the four root locus segments ($BNLR_1$ and $ENLR_1$; $BNLR_2$ and $ENLR_2$) is parallel to the imaginary axis. The third class of the two root loci of adjacent degree $NLR_1$ and $NLR_2$ being adjacent at the two finite points $ap_1$ and $ap_2$ arises. Then, inside any small adjacent domain of the two adjacent finite points $ap_1$ and $ap_2$, when the imaginary part $t$ acts as a variable, there surely exists at least one demarcation point of the imaginary part variable of the root locus equation (1.1).
\end{lemma}

\begin{proof}
On one side of the first straight line $B3$, the initial root locus segments $BNLR_1$ and $BNLR_2$ of the non-identical root loci of adjacent degree ($NLR_1$ or $NLR_2$) exist simultaneously. On the other side of the straight line $B3$, the terminal root locus segments $ENLR_1$ and $ENLR_2$ of the non-identical root loci of adjacent degree ($NLR_1$ or $NLR_2$) exist simultaneously. Under such conditions, on the side of the first straight line $B3$ where the initial root locus segments $BNLR_1$ and $BNLR_2$ of the two root loci of adjacent degree ($NLR_1$ and $NLR_2$) exist simultaneously, the initial root locus segments $BNLR_1$ and $BNLR_2$ of the two root loci of adjacent degree ($NLR_1$ and $NLR_2$) exist. Here this side is chosen in order to carry out the following proof.

From the pole ($np_1$ or $np_2$) emitting this initial root locus segment ($BNLR_1$ or $BNLR_2$), root loci $RL_e$ of other degrees are certainly emitted; these root loci $RL_e$ extend along the root loci ($NLR_1$ and $NLR_2$) of this initial root locus segment of adjacent degree ($BNLR_1$ or $BNLR_2$) towards the two adjacent finite points $ap_1$ and $ap_2$. This is because the other of the two adjacent finite points $ap_1$ and $ap_2$ lies on a root locus of adjacent degree ($NLR_2$ or $NLR_1$) coming from the other pole ($np_2$ or $np_1$).

Thus, these root loci $RL_e$ must change their direction of extension in the neighbourhood of the two adjacent finite points ($ap_1$ and $ap_2$). $RL_e$ then extends along the root locus ($NLR_2$ or $NLR_1$) of the initial root locus segment ($BNLR_2$ or $BNLR_1$) coming from the other pole ($np_2$ or $np_1$). These root loci $RL_e$ are all root loci emitted from the same pole ($np_2$ or $np_1$); each $RL_e$ is the same root locus. In the neighbourhood $Q$ of the two adjacent finite points $ap_1$ and $ap_2$, on the side of the first straight line $B3$ where the initial root locus segment ($BNLR_2$ or $BNLR_1$) exists, just as for the two root loci of adjacent degree, these root loci $RL_e$ also have initial root locus segments and terminal root locus segments.

Therefore, in the adjacent domain $Q$, there simultaneously exist many initial root locus segments and terminal root locus segments of the same root locus $RL_e$. Therefore, repeating the proof of Lemma 3.25 proves: no matter how small the adjacent domain of the finite breakaway point is, when the imaginary part $t$ acts as a variable, there surely exists at least one demarcation point of the imaginary part variable of the root locus equation (1.1).

In the adjacent domain $Q$, these root loci $RL_e$, which are emitted from the same pole ($np_1$ or $np_2$) and extend along the root loci of adjacent degree ($NLR_1$ or $NLR_2$), cannot intersect the root loci coming from the other pole ($np_2$ or $np_1$). Even if the case of intersecting root loci must be considered, after intersection these root loci $RL_e$ must still continue to extend along the root loci of adjacent degree ($NLR_1$ or $NLR_2$). Therefore, the initial root locus segments and the terminal root locus segments of these root loci $RL_e$ must also exist.
\end{proof}

Sum up Lemma 3.25, Lemma 3.26 and Lemma 3.27, we have:

\begin{lemma}
At finite positions of the complex plane, the root loci $NLR_1$ and $NLR_2$ starting from two distinct poles ($np_1$ or $np_2$) come into contact. $NLR_1$ and $NLR_2$ are adjacent at the two finite points $ap_1$ and $ap_2$. $NLR_1$ and $NLR_2$ cannot intersect. With $ap_1$ and $ap_2$ as starting points, $NLR_1$ and $NLR_2$ are divided by $ap_1$ and $ap_2$ respectively into two root locus segments ($BNLR_1$ and $ENLR_1$; $BNLR_2$ and $ENLR_2$). Such a situation thus produces two initial root locus segments of adjacent degree $BNLR_1$ and $BNLR_2$ and two terminal root locus segments of adjacent degree $ENLR_1$ and $ENLR_2$. None of the four root locus segments ($BNLR_1$ and $ENLR_1$; $BNLR_2$ and $ENLR_2$) is parallel to the imaginary axis. Then, inside any small adjacent domain of the two adjacent finite points $ap_1$ and $ap_2$, when the imaginary part $t$ acts as a variable, there surely exists at least one demarcation point of the imaginary part variable of the root locus equation (1.1).
\end{lemma}

Combining Lemma 3.24 and Lemma 3.28, we obtain:

\begin{lemma}
At finite positions of the complex plane, the root loci $NLR_1$ and $NLR_2$ starting from two distinct poles ($np_1$ or $np_2$) come into contact. $NLR_1$ and $NLR_2$ are adjacent at the two finite points $ap_1$ and $ap_2$. $NLR_1$ and $NLR_2$ cannot intersect. Then, inside any small adjacent domain of the two adjacent finite points $ap_1$ and $ap_2$, when the imaginary part $t$ acts as a variable, there surely exists at least one demarcation point of the imaginary part variable of the root locus equation (1.1).
\end{lemma}

Repeating the proof of Lemma 3.29, we obtain:

\begin{lemma}
At finite positions of the complex plane, the root loci $NLR_1$ and $NLR_2$ starting from two distinct poles ($np_1$ or $np_2$) come into contact. $NLR_1$ and $NLR_2$ are adjacent at the two finite points $ap_1$ and $ap_2$. $NLR_1$ and $NLR_2$ cannot intersect. Then, inside any small adjacent domain of the two adjacent finite points $ap_1$ and $ap_2$, when the real part $\sigma$ acts as a variable, there surely exists at least one demarcation point of the real part variable of the root locus equation (1.1).
\end{lemma}

Lemma 3.29 and Lemma 3.30 are consistent with Theorem 3.20 and Theorem 3.21. Therefore, by Lemma 3.29 and Lemma 3.30, repeating the proof of Lemma 3.22, we can prove: 

\begin{lemma}
The second class of possible situations cannot exist.
\end{lemma}
At finite positions of the complex plane, if $l$ root loci of adjacent degree, as in Lemma 3.31, meet at $l$ points but do not intersect, then, since the situation of every adjacent pair is exactly the same as in Lemma 3.31, by Lemma 3.31 none of the adjacent pairs can hold. Hence the situation in which $l$ root loci of adjacent degree meet at $l$ points cannot hold either. Therefore, by Lemma 3.31, we have:

\begin{lemma}
At finite positions of the complex plane, root loci starting from $l$ distinct poles come into contact, and $l$ root loci of adjacent degree among them are adjacent and do not intersect. Then this possible situation cannot exist.
\end{lemma}

\begin{theorem}
For the root locus equation (1.1) which has no finite zeros, the same one root loci cannot intersect with itself in the finite position of $\mathbb{C}\cup\{\infty\}$.
\end{theorem}
\begin{proof}
For the root locus equation (1.1) which has no any finite zeros, if the same one root loci  intersect itself in the finite position of $\mathbb{C}\cup\{\infty\}$. The same one root loci itself constitute a closed curve. The same one root loci form a closed zone. In the closed zone of the finite position of $\mathbb{C}\cup\{\infty\}$, there don't exist any finite zeros. If there exist the finite poles of Eq(1.1) in the closed zone. The infinite numbers of the different degree numbers root locus are originated from these poles. The infinite numbers of the different degree numbers  root locus cannot pass through the closed curve which is constituted by the same one root loci. These different degree numbers root locus cannot leave the closed zone in $\mathbb{C}\cup\{\infty\}$. These different degree numbers root locus cannot end at their ending-points. The continuity of the root locus are broken. The situation cannot be true.

If there don't exist the poles of Eq(1.1) in the closed zone. Because Eq(1.1) has no any finite zeros. In the closed zone, there are no finite zeros and poles. So, there are two possibilities. One is that the closed zone is a "vacuum" which has no any root locus. Theorem 1.24 proves: the root locus of Eq(1.1) fully fill the whole complex plane, there don't exist the "vacuum" in $\mathbb{C}\cup\{\infty\}$. The situation cannot be true.

If there exist the root locus in the closed zone of the finite position of $\mathbb{C}\cup\{\infty\}$. These root locus must enter the closed zone inner from the closed zone outer. The different degree numbers root locus cannot intersect. Only there is a possible situation,  the same degree one number root locus pass through the closed curve which the same degree number root loci constitute, and they enter into the closed zone of $\mathbb{C}\cup\{\infty\}$.

In this closed zone of $\mathbb{C}\cup\{\infty\}$, all of points which is in this closed zone are on the same one degree number root locus, in which the degree numbers of these points are the same as the degree number of the closed curve. Namely, the same degree number root locus fully fill the whole closed zone of $\mathbb{C}\cup\{\infty\}$.

All points on the horizontal and vertical line which pass the closed zone obtain the same phase angle  $\varphi{(\sigma,t)}$. So, the phase angle  $\varphi{(\sigma,t)}$ concerning the variable $\sigma$ and variable  $t$, $\frac{\partial{\varphi}}{\partial{t}}(\sigma,t)\mid{_{s=s_0}}=0$ and $\frac{\partial{\varphi}}{\partial{\sigma}}(\sigma,t)\mid{_{s=s_0}}=0$. For all points in the closed zone of the finite position of the generalized complex plane $\mathbb{C}\cup\{\infty\}$, the two partial derivative function equations are true simultaneously. According to Theorem 2.7, all of points in the closed zone are zeros of derivative of analytic function in Eq(1.1). But, the zeros of analytic function need to distributed discretely in the finite position of the complex plane $\mathbb{C}\cup\{\infty\}$. They cannot continuously distributed in the finite position of the complex plane $\mathbb{C}\cup\{\infty\}$. So, this situation cannot be true.

Sum up the above results, we have proved that such closed zone cannot exist and the same one root loci cannot intersect to itself in the finite position of $\mathbb{C}\cup\{\infty\}$.
\end{proof}

\section{The Zeros of Derivatives of Exponential Series}
For a exponential series which is defined in its definition domain $\Upsilon$ as:
$FF(s)=\sum^{+\infty}_{n=1}\frac{a_{n}(n)}{(b_{n}(n))^{s}}$, $(n\in N, s\in\mathbb{C}\cup\{\infty\}, a_{n}(n)\in R, b_{n}(n)\in R)$. All $a_{n}(n)$ and $b_{n}(n)$ obtain the real values.
It is obvious that $FF(s)$ is a conjugate meromorphic function.
The root locus equation of the $FF(s)$ function is
\begin{equation}
K\frac{1}{FF(s)}=a+ib
\end{equation}
\begin{lemma}
For $FF(s)$, its real part is $\sum^{+\infty}_{n=1}\frac{a_{n}(n)\cos(t\log (b_{n}(n)))}{(b_{n}(n))^{\sigma}}$,
its imaginary part is $-\sum^{+\infty}_{n=1}\frac{a_{n}(n)\sin(t\log (b_{n}(n)))}{(b_{n}(n))^{\sigma}}$.
\end{lemma}
\begin{proof}
Let $s=\sigma+it$, and we substitute it into the expression of the $FF(s)$ function , we can obtain:
$\sum^{+\infty}_{n=1}\frac{a_{n}(s)}{(b_{n}(n))^{s}}=
\sum^{+\infty}_{n=1}\frac{a_{n}(n)}{(b_{n}(n))^{\sigma}(\cos(t\log (b_{n}(n)))+i\sin(t\log (b_{n}(n))))}
=\sum^{+\infty}_{n=1}\frac{a_{n}(n)\cos(t\log (b_{n}(n)))}{(b_{n}(n))^{\sigma}}
-i\sum^{+\infty}_{n=1}\frac{a_{n}(n)\sin(t\log (b_{n}(n)))}{(b_{n}(n))^{\sigma}}$.
\end{proof}

\begin{lemma}
For $FF(s)$, the real part of its derivative is

$-\sum^{+\infty}_{n=1}\frac{a_{n}(n)\log (b_{n}(n))\cos(t\log (b_{n}(n)))}{(b_{n}(n))^{\sigma}}$,
the imaginary part of its derivatives is

$\sum^{+\infty}_{n=1}\frac{a_{n}(n)\log (b_{n}(n))\sin(t\log (b_{n}(n)))}{(b_{n}(n))^{\sigma}}$.
\end{lemma}
\begin{proof}
The derivative of the $FF(s)$ function is: $FF^{'}(s)=
(\sum^{+\infty}_{n=1}\frac{a_{n}(n)}{(b_{n}(n))^{s}})^{'}
=\sum^{+\infty}_{n=1}\frac{-a_{n}(n)\log (b_{n}(n))}{(b_{n}(n))^{s}}$.
Let $s=\sigma+it$, and we substitute it into the expression of the derivatives of the $FF(s)$ function, we can obtain:

$\sum^{+\infty}_{n=1}\frac{-a_{n}(n)\log (b_{n}(n))}{(b_{n}(n))^{s}}=
\sum^{+\infty}_{n=1}\frac{-a_{n}(n)\log (b_{n}(n))}{(b_{n}(n))^{\sigma}(\cos(t\log (b_{n}(n)))+i\sin(t\log (b_{n}(n)))}$

$=-\sum^{+\infty}_{n=1}\frac{a_{n}(n)\log (b_{n}(n))\cos(t\log (b_{n}(n)))}{(b_{n}(n))^{\sigma}}+
i\sum^{+\infty}_{n=1}\frac{a_{n}(n)\log (b_{n}(n))\sin(t\log (b_{n}(n)))}{(b_{n}(n))^{\sigma}}$.
\end{proof}

Let $E(\sigma,t)=\sum^{+\infty}_{n=1}\frac{a_{n}(n)\cos(t\log (b_{n}(n)))}{(b_{n}(n))^{\sigma}}$,
$F(\sigma,t)=\sum^{+\infty}_{n=1}\frac{a_{n}(n)\sin(t\log (b_{n}(n)))}{(b_{n}(n))^{\sigma}}$,
$G(\sigma,t)=\sum^{+\infty}_{n=1}\frac{a_{n}(n)\log (b_{n}(n))\cos(t\log (b_{n}(n)))}{(b_{n}(n))^{\sigma}}$,

and $H(\sigma,t)=\sum^{+\infty}_{n=1}\frac{a_{n}(n)\log (b_{n}(n))\sin(t\log (b_{n}(n)))}{(b_{n}(n))^{\sigma}}$.
So, the $FF(s)$ function can be rewritten as: $FF(s)=\sum^{+\infty}_{n=1}\frac{a_{n}(n)}{(b_{n}(n))^{s}}=E(\sigma,t)-iF(\sigma,t)$.
The derivatives of  the $FF(s)$ function can be rewritten as: $FF^{'}(s)=
(\sum^{+\infty}_{n=1}\frac{a_{n}(n)}{(b_{n}(n))^{s}})^{'}
=-G(\sigma,t)+iH(\sigma,t)$.

\begin{lemma}
For any finite zero of $FF^{'}(s)$, the two partial derivatives of argument of the $FF(s)$ function are
$\frac{\partial}{\partial{t}}\{\arg{FF(\sigma,t)}\}=\frac{-G(\sigma,t)E(\sigma,t)-H(\sigma,t)F(\sigma,t)}{E^{2}(\sigma,t)+F^{2}(\sigma,t)}=0$ and
$\frac{\partial}{\partial{\sigma}}\{\arg{FF(\sigma,t)}\}=
\frac{H(\sigma,t)E(\sigma,t)-G(\sigma,t)F(\sigma,t)}{E^{2}(\sigma,t)+F^{2}(\sigma,t)}=0$.
\end{lemma}
\begin{proof}
$\frac{\partial}{\partial{t}}\{\arg{FF(\sigma,t)}\}=$ $\frac{\partial}{\partial{t}}
\arg\{\frac{-\sum^{+\infty}_{n=1}\frac{a_{n}(n)\sin(t\log (b_{n}(n)))}{(b_{n}(n))^{\sigma}}}
{\sum^{+\infty}_{n=1}\frac{a_{n}(n)\cos(t\log (b_{n}(n)))}{(b_{n}(n))^{\sigma}}}\}
=$

$\frac{-\sum^{+\infty}_{n=1}\frac{a_{n}(n)\log (b_{n}(n))\cos(t\log (b_{n}(n)))}{(b_{n}(n))^{\sigma}}
\sum^{+\infty}_{n=1}\frac{a_{n}(n)\cos(t\log (b_{n}(n)))}{(b_{n}(n))^{\sigma}}}
{(\sum^{+\infty}_{n=1}\frac{a_{n}(n)\cos(t\log (b_{n}(n)))}{(b_{n}(n))^{\sigma}})^{2}+
}
$

$\frac{-\sum^{+\infty}_{n=1}\frac{a_{n}(n)\log (b_{n}(n))\sin(t\log (b_{n}(n)))}{(b_{n}(n))^{\sigma}}
\sum^{+\infty}_{n=1}\frac{a_{n}(n)\sin(t\log (b_{n}(n)))}{(b_{n}(n))^{\sigma}}}
{(\sum^{+\infty}_{n=1}\frac{a_{n}(n)\sin(t\log (b_{n}(n)))}{(b_{n}(n))^{\sigma}})^{2}}
$

$=\frac{-G(\sigma,t)E(\sigma,t)-H(\sigma,t)F(\sigma,t)}{E^{2}(\sigma,t)+F^{2}(\sigma,t)}=0$.
And

$\frac{\partial}{\partial{\sigma}}\{\arg{FF(\sigma,t)}\}=
\frac{\partial}{\partial{\sigma}}
\arg\{\frac{-\sum^{+\infty}_{n=1}\frac{a_{n}(n)\sin(t\log (b_{n}(n)))}{(b_{n}(n))^{\sigma}}}
{\sum^{+\infty}_{n=1}\frac{a_{n}(n)\cos(t\log (b_{n}(n)))}{(b_{n}(n))^{\sigma}}}\}
=$

$\frac{\sum^{+\infty}_{n=1}\frac{a_{n}(n)\log (b_{n}(n))\sin(t\log (b_{n}(n)))}{(b_{n}(n))^{\sigma}}
\sum^{+\infty}_{n=1}\frac{a_{n}(n)\cos(t\log (b_{n}(n)))}{(b_{n}(n))^{\sigma}}-
}
{(\sum^{+\infty}_{n=1}\frac{a_{n}(n)\cos(t\log (b_{n}(n)))}{(b_{n}(n))^{\sigma}})^{2}+
}$

$\frac{
\sum^{+\infty}_{n=1}\frac{a_{n}(n)\log (b_{n}(n))\cos(t\log (b_{n}(n)))}{(b_{n}(n))^{\sigma}}
\sum^{+\infty}_{n=1}\frac{a_{n}(n)\sin(t\log (b_{n}(n)))}{(b{n}(n))^{\sigma}}}
{(\sum^{+\infty}_{n=1}\frac{a_{n}(n)\sin(t\log (b_{n}(n)))}{(b_{n}(n))^{\sigma}})^{2}}
=\frac{H(\sigma,t)E(\sigma,t)-G(\sigma,t)F(\sigma,t)}{E^{2}(\sigma,t)+F^{2}(\sigma,t)}=0$.
\end{proof}

\begin{lemma}
In the upper half of $\mathbb{C}\cup\{\infty\}$, if there exist the breakaway points of Eq(4.1) which satisfy the following condition. These breakaway points are generated by the "common root locus" satisfy the following degree condition. The degree condition is: the degree of "common root locus" are between $2q\pi+\frac{\pi}{18}$ degree and $2q\pi+\frac{8\pi}{18}$ degree, $2q\pi+\frac{10\pi}{18}$ degrees and $2q\pi+\frac{17\pi}{18}$ degree, $2q\pi+\frac{19\pi}{18}$ degree and $2q\pi+\frac{26\pi}{18}$ degree, or $2q\pi+\frac{28\pi}{18}$ degree and $2q\pi+\frac{35\pi}{18}$ degree. Then, there surely exist the adjacent domains $Y(l)$ $l=1,2,\cdots,$ of the breakaway points. In the adjacent domain $Y(l)$ of every breakaway point, there only exists one breakaway point. And the real parts and imaginary parts of the $FF(s)$ function on all points in the adjacent domain $Y(l)$ of every breakaway point cannot both equal to $0$.
\end{lemma}
\begin{proof}
The breakaway points of Eq(4.1) are the zeros of derivative of the $FF(s)$ function. And the $FF(s)$ function is meromorphic function. So, the derivative function of the $FF(s)$ function is also meromorphic function. The finite zeros of the meromorphic function is isolate. Therefore, around every breakaway point, there all exist only one adjacent domain which there don't exist other breakaway point. Let the adjacent domain which satisfies this condition be $Y_{1}(l)$, $l=1,2,\cdots,$.

In the upper half of $\mathbb{C}\cup\{\infty\}$, if there exist the finite breakaway points of Eq(4.1), if the degree of "common root locus" which they generate these breakaway points are between $2q\pi+\frac{\pi}{18}$ degree and $2q\pi+\frac{8\pi}{18}$ degree, $2q\pi+\frac{10\pi}{18}$ degree and $2q\pi+\frac{17\pi}{18}$ degree, $2q\pi+\frac{19\pi}{18}$ degree and $2q\pi+\frac{26\pi}{18}$ degree, or $2q\pi+\frac{28\pi}{18}$ degree and $2q\pi+\frac{35\pi}{18}$ degree. Around the "common root locus", there exist the similar degree root locus, so, if the degree of "common root locus" are between the above these degrees, this ensures that the real part and imaginary part of the $FF(s)$ function both cannot be equal to 0. So, under this condition, around the breakaway point $s_0=\sigma_0+it_0$, there exist a neighborhood, let the adjacent domain be $Y_{2}(l)$, $l=1,2,\cdots,$ the real parts and imaginary parts of the $FF(s)$ function on all points in this a neighborhood cannot both equal to $0$.

Let $ Y(l)=Y_{1}(l)\cap Y_{2}(l)$, in the adjacent domain $Y(l)$, there only exists one breakaway point, and the real parts and imaginary parts of the $FF(s)$ function on all points in the adjacent domain cannot both equal to $0$.
\end{proof}

\begin{lemma}
In the adjacent domain $Y(l)$ of every breakaway point, for all points on the curves that the two kinds of demarcation points of Eq(4.1) are on, except the breakaway point $s_0=\sigma_0+it_0$, the four function values $E(\sigma,t)$, $F(\sigma,t)$, $G(\sigma,t)$ and $H(\sigma,t)$ all cannot equal to 0.
\end{lemma}

\begin{proof}
According to Definition 3.12, we have: on the demarcation point $s_1=\sigma_1+it_1$ of Eq(4.1), which the imaginary part variable $t$ is variable, the following expression is true. $\frac{\partial}{\partial{t}}\{\arg{FF(\sigma,t)}\}\mid{_{s=\sigma_1+it_1}}=
\frac{-G(\sigma_1,t_1)E(\sigma_1,t_1)-H(\sigma_1,t_1)F(\sigma_1,t_1)}{E^{2}(\sigma_1,t_1)+F^{2}(\sigma_1,t_1)}=0$.

According to Definition 3.12, we have: on the demarcation point $s_2=\sigma_2+it_2$ of Eq(4.1), which the real part variable $\sigma$ is the variable, the following expression is true. $\frac{\partial}{\partial{\sigma}}\{\arg{FF(\sigma,t)}\}\mid{_{s=\sigma_2+it_2}}=
\frac{H(\sigma_2,t_2)E(\sigma_2,t_2)-G(\sigma_2,t_2)F(\sigma_2,t_2)}{E^{2}(\sigma_2,t_2)+F^{2}(\sigma_2,t_2)}=0$.

In the adjacent domain $Y(l)$ of the breakaway point $s_0=\sigma_0+it_0$, for the demarcation point $s_1=\sigma_1+it_1$ of Eq(4.1), which the imaginary part variable $t$ is as the variable constitute the curves of the demarcation points. For the demarcation point $s_2=\sigma_2+it_2$ of Eq(4.1), which the real part variable $\sigma$ is the variable constitute the curves of the demarcation points. According to Lemma 4.4,
 $E(\sigma_1,t_1)\neq0$, $F(\sigma_1,t_1)\neq0$, $E(\sigma_2,t_2)\neq0$, and $F(\sigma_2,t_2)\neq0$,
we have $E^{2}(\sigma_1,t_1)+F^{2}(\sigma_1,t_1)\neq0$, $E^{2}(\sigma_2,t_2)+F^{2}(\sigma_2,t_2)\neq0$. So,
$-G(\sigma_1,t_1)E(\sigma_1,t_1)-H(\sigma_1,t_1)F(\sigma_1,t_1)=0$, $H(\sigma_2,t_2)E(\sigma_2,t_2)-G(\sigma_2,t_2)F(\sigma_2,t_2)=0$.

Because in the adjacent domain $Y(l)$ of the breakaway point $s_0$, on the points on the curves that the two kinds of demarcation points are on,  $E(\sigma_1,t_1)$, $F(\sigma_1,t_1)$, $E(\sigma_2,t_2)$, and $F(\sigma_2,t_2)$ cannot equal to 0. Without loss of generality, let $G(\sigma_1,t_1)=0$, then $H(\sigma_1,t_1)F(\sigma_1,t_1)=0$. And $F(\sigma_1,t_1)\neq0$ so, there must be $H(\sigma_1,t_1)=0$.
Obviously, the result proves: on the curves of the demarcation point $s_1=\sigma_1+it_1$ of Eq(4.1) which the imaginary part variable $t$ is the variable, there exist the other breakaway point of Eq(4.1) except the breakaway point $s_0=\sigma_0+it_0$. But, in the adjacent domain $Y(l)$ of the breakaway point $s_0$, we have excluded the existence of the other breakaway point except the breakaway point $s_0=\sigma_0+it_0$. Hence, in the neighborhood of the breakaway point $s_0=\sigma_0+it_0$, $G(\sigma_1,t_1)=0$ cannot appear.

So, we prove: For the points on the curves that the two kinds of demarcation points of Eq(4.1) are on, except the breakaway point $s_0=\sigma_0+it_0$,
all of the four function values $G(\sigma_1,t_1)$, $H(\sigma_1,t_1)$, $G(\sigma_2,t_2)$, and $H(\sigma_2,t_2)$ cannot equal to 0.
According to Lemma 4.4, we prove; In the adjacent domain $Y(l)$ of every breakaway point, on all points on the curves that the two kinds of demarcation points of Eq(4.1) are on, except the breakaway point $s_0=\sigma_0+it_0$, the four function values $E(\sigma,t)$, $F(\sigma,t)$, $G(\sigma,t)$ and $H(\sigma,t)$ all cannot equal to 0.
\end{proof}

\begin{lemma}
In the upper half of $\mathbb{C}\cup\{\infty\}$, there cannot exist the breakaway points of Eq(4.1) which satisfy the following condition. These breakaway points are generated by the "common root locus" satisfy the following degree condition. The degree condition is: the degree of "common root locus" are between $2q\pi+\frac{\pi}{18}$ degree and $2q\pi+\frac{8\pi}{18}$ degree, $2q\pi+\frac{10\pi}{18}$ degrees and $2q\pi+\frac{17\pi}{18}$ degree, $2q\pi+\frac{19\pi}{18}$ degree and $2q\pi+\frac{26\pi}{18}$ degree, or $2q\pi+\frac{28\pi}{18}$ degree and $2q\pi+\frac{35\pi}{18}$ degree.
\end{lemma}

\begin{proof}
$-G(\sigma_1,t_1)E(\sigma_1,t_1)-H(\sigma_1,t_1)F(\sigma_1,t_1)=0$. In Lemma 4.5, we have proved that the four quantities in the equation all cannot equal to 0, so, we can transform this equation into the following equation, namely, we can obtain the following equation,
$\frac{G(\sigma_1,t_1)}{H(\sigma_1,t_1)}=-\frac{F(\sigma_1,t_1)}{E(\sigma_1,t_1)}$.

Because $H(\sigma_2,t_2)E(\sigma_2,t_2)-G(\sigma_2,t_2)F(\sigma_2,t_2)=0$, based on the same reason, we can obtain the equation $\frac{G(\sigma_2,t_2)}{H(\sigma_2,t_2)}=\frac{F(\sigma_2,t_2)}{E(\sigma_2,t_2)}$.

On the curves of the demarcation point $s_1=\sigma_1+it_1$ of Eq(4.1) which the imaginary part variable $t$ is the variable, when the demarcation point $s_1=\sigma_1+it_1$ is more and more close to the breakaway point $s_0=\sigma_0+it_0$, we can obtain: $\frac{G(\sigma_0,t_0)}{H(\sigma_0,t_0)}=-\frac{F(\sigma_0,t_0)}{E(\sigma_0,t_0)}$.

On the curves of the demarcation point $s_2=\sigma_2+it_2$ of Eq(4.1) which the real part variable $\sigma$ is the variable, when the demarcation point $s_2=\sigma_2+it_2$ is more and more close to the breakaway point $s_0=\sigma_0+it_0$, we can obtain: $\frac{G(\sigma_0,t_0)}{H(\sigma_0,t_0)}=\frac{E(\sigma_0,t_0)}{F(\sigma_0,t_0)}$.

So, we can obtain: $\frac{E(\sigma_0,t_0)}{F(\sigma_0,t_0)}=-\frac{F(\sigma_0,t_0)}{E(\sigma_0,t_0)}$. We have: $E^{2}(\sigma_0,t_0)+F^{2}(\sigma_0,t_0)=0$, so, $E(\sigma_0,t_0)=F(\sigma_0,t_0)=0$. This is contradiction which the two function values $E(\sigma_0,t_0)$, $F(\sigma_0,t_0)$ can not both equal to 0. So,  the original assuming cannot be true, namely, it isn't true that there exist the finite truly breakaway point of Eq(4.1) In the upper half of $\mathbb{C}\cup\{\infty\}$, it cannot be true. So, we have proved that there don't exist the finite truly breakaway point of Eq(4.1) in the upper half of $\mathbb{C}\cup\{\infty\}$.
\end{proof}

\begin{lemma}
In the upper half of $\mathbb{C}\cup\{\infty\}$, there don't exist the breakaway points of Eq(4.1) which satisfy the following condition. These breakaway points are generated by the "common root locus" which satisfy the following degree condition. The degree condition is: the degree of "common root locus" are between $2q\pi+\frac{35\pi}{18}$ degree and $2q\pi+\frac{\pi}{18}$ degree, $2q\pi+\frac{8\pi}{18}$ degree and $2q\pi+\frac{10\pi}{18}$ degree, $2q\pi+\frac{17\pi}{18}$ degree and $2q\pi+\frac{19\pi}{18}$ degree, $2q\pi+\frac{26\pi}{18}$ degree and $2q\pi+\frac{28\pi}{18}$ degree.
\end{lemma}
\begin{proof}
In the upper half of $\mathbb{C}\cup\{\infty\}$, if there exist the truly finite breakaway points of Eq(4.1), and if the degree of "common root locus" which they generate these breakaway points are not between $2q\pi+\frac{\pi}{18}$ degree and $2q\pi+\frac{8\pi}{18}$ degree, $2q\pi+\frac{10\pi}{18}$ degreea and $2q\pi+\frac{17\pi}{18}$ degree, $2q\pi+\frac{19\pi}{18}$ degree and $2q\pi+\frac{26\pi}{18}$ degree, or $2q\pi+\frac{28\pi}{18}$ degree and $2q\pi+\frac{35\pi}{18}$ degree. However, they are between $2q\pi+\frac{35\pi}{18}$ degree and $2q\pi+\frac{\pi}{18}$ degree, $2q\pi+\frac{8\pi}{18}$ degree and $2q\pi+\frac{10\pi}{18}$ degree, $2q\pi+\frac{17\pi}{18}$ degree and $2q\pi+\frac{19\pi}{18}$ degree, $2q\pi+\frac{26\pi}{18}$ degree and $2q\pi+\frac{28\pi}{18}$ degree.

Around the "common root locus", there exist the similar degree root locus, so, if the degree of "common root locus" are between the above these degrees, this doesn't ensure that the real part and imaginary part of the $FF(s)$ function both cannot be equal to 0. So, under this condition, around the breakaway point $s_0=\sigma_0+it_0$, perhaps, there doesn't exist a neighborhood, the real parts and imaginary parts of the $FF(s)$ function on all of the points in this  neighborhood both cannot both equal to 0.

When the "common root locus" which have such condition appear, and they generate the breakaway points, for the $FF(s)$ function, we do the aid function, let: $FF^{*}(s)=(1+i)FF(s)=(1+i)\sum^{+\infty}_{n=1}\frac{a_{n}(n)}{(b_{n}(n))^{s}}=(E(\sigma,t)+F(\sigma,t))+i(E(\sigma,t)-F(\sigma,t))$. We have: $\arg\{FF^{*}(s)\} =\frac{\pi}{4}+\arg\{FF(s)\}$, $\frac{dFF^{*}(s)}{ds}=\frac{dFF(s)}{ds}$, $\frac{\partial}{\partial{t}}\{\arg(FF^{*}(s))\}=\frac{\partial}{\partial{t}}\{\arg(FF(s))\}$, $\frac{\partial}{\partial{\sigma}}\{\arg(FF^{*}(s))\}=\frac{\partial}{\partial{\sigma}}\{\arg(FF(s))\}$. So, the zeros of derivative and two demarcation points of the aid function $FF^{*}(s)$ of the $FF(s)$ function are completely same to the zeros of derivative and two categories demarcation points of the $FF(s)$.

But, the degree of "common root locus" change, if the degree of originate "common root locus" is $\arg\{FF(s)\}$. Then, the degree of "common root locus" of the aid function is $\frac{\pi}{4}+\arg\{FF(s)\}$.

Because $2q\pi+\frac{35\pi}{18}+\frac{\pi}{4}=2q\pi+\frac{7\pi}{36}$, $2q\pi+\frac{\pi}{18}+\frac{\pi}{4}=2q\pi+\frac{11\pi}{36}$.
$\frac{\pi}{18}=\frac{2\pi}{36}<\frac{7\pi}{36}$, $\frac{11\pi}{36}<\frac{8\pi}{18}=\frac{16\pi}{36}$. So, the "common root locus" which their degrees are between $2q\pi+\frac{35\pi}{18}$ degree and $2q\pi+\frac{\pi}{18}$ degree are transformed into the "common root locus" which their degrees are between $2q\pi+\frac{7\pi}{36}$ degree and $2q\pi+\frac{11\pi}{36}$ degree. The degree number in the interval $[2q\pi+\frac{7\pi}{36},2q\pi+\frac{11\pi}{36}]$ are all between $2q\pi+\frac{\pi}{18}$ degree and $2q\pi+\frac{8\pi}{18}$ degree.

Because $2q\pi+\frac{8\pi}{18}+\frac{\pi}{4}=2q\pi+\frac{25\pi}{36}$, $2q\pi+\frac{10\pi}{18}+\frac{\pi}{4}=2q\pi+\frac{29\pi}{36}$. $\frac{10\pi}{18}=\frac{20\pi}{36}<\frac{25\pi}{36}$, $\frac{29\pi}{36}<\frac{17\pi}{18}=\frac{34\pi}{36}$. So, the "common root locus" which their degrees are between $2q\pi+\frac{8\pi}{18}$ degree and $2q\pi+\frac{10\pi}{18}$ degree are transformed into the "common root locus" which their degrees are between $2q\pi+\frac{25\pi}{36}$ degree and $2q\pi+\frac{29\pi}{36}$ degree. The degree number in the interval $[2q\pi+\frac{25\pi}{36},2q\pi+\frac{29\pi}{36}]$ are all between $2q\pi+\frac{10\pi}{18}$ degree and $2q\pi+\frac{17\pi}{18}$ degree.

Because $2q\pi+\frac{17\pi}{18}+\frac{\pi}{4}=2q\pi+\frac{43\pi}{36}$, $2q\pi+\frac{19\pi}{18}+\frac{\pi}{4}=2q\pi+\frac{47\pi}{36}$. $\frac{19\pi}{18}=\frac{38\pi}{36}<\frac{43\pi}{36}$, $\frac{47\pi}{36}<\frac{26\pi}{18}=\frac{52\pi}{36}$. So, the "common root locus" which their degrees are between $2q\pi+\frac{17\pi}{18}$ degree and $2q\pi+\frac{19\pi}{18}$ degree are transformed into the "common root locus" which their degrees are between $2q\pi+\frac{43\pi}{36}$ degree and $2q\pi+\frac{47\pi}{36}$ degree. The degree number in the interval $[2q\pi+\frac{43\pi}{36},2q\pi+\frac{47\pi}{36}]$ are all between $2q\pi+\frac{19\pi}{18}$ degree and $2q\pi+\frac{26\pi}{18}$ degree.

Because $2q\pi+\frac{26\pi}{18}+\frac{\pi}{4}=2q\pi+\frac{61\pi}{36}$, $2q\pi+\frac{28\pi}{18}+\frac{\pi}{4}=2q\pi+\frac{65\pi}{36}$. $\frac{28\pi}{18}=\frac{56\pi}{36}<\frac{61\pi}{36}$, $\frac{65\pi}{36}<\frac{35\pi}{18}=\frac{70\pi}{36}$. So, the "common root locus" which their degrees are between $2q\pi+\frac{26\pi}{18}$ degree and $2q\pi+\frac{28\pi}{18}$ degree are transformed into the "common root locus" which their degrees are between $2q\pi+\frac{61\pi}{36}$ degree and $2q\pi+\frac{65\pi}{36}$ degree. The degree number in the interval $[2q\pi+\frac{61\pi}{36},2q\pi+\frac{65\pi}{36}]$ are all between $2q\pi+\frac{28\pi}{18}$ degree and $2q\pi+\frac{35\pi}{18}$ degree.

After utilizing the transformation which multiply a constant factor, completely repeat the proofs from Lemma 4.4 to Lemma 4.6, we have:
$(E(\sigma_0,t_0)+F(\sigma_0,t_0))^{2}+(E(\sigma_0,t_0)-F(\sigma_0,t_0))^{2}=0$.

So, we can obtain: $E(\sigma_0,t_0)+F(\sigma_0,t_0)=0$ and $E(\sigma_0,t_0)-F(\sigma_0,t_0)=0$. From the two equations, we have: $E(\sigma_0,t_0)=F(\sigma_0,t_0)=0$. Based on the same reason with the ending partial in proof of Lemma 4.6, we prove that there don't exist the truly breakaway points of Eq(4.1) in the upper half of $\mathbb{C}\cup\{\infty\}$.
\end{proof}

\begin{theorem}
In the upper half of $\mathbb{C}\cup\{\infty\}$, there don't exist the finite breakaway points of Eq(4.1). Namely, $FF^{'}(s)$ has no finite zeros in $\Upsilon$, $Im(s)>0$.
\end{theorem}
\begin{proof}
When the degree of "common root locus" are between $2q\pi+\frac{\pi}{18}$ degree and $2q\pi+\frac{8\pi}{18}$ degree, $2q\pi+\frac{10\pi}{18}$ degree and $2q\pi+\frac{17\pi}{18}$ degree, $2q\pi+\frac{19\pi}{18}$ degree and $2q\pi+\frac{26\pi}{18}$ degree, or $2q\pi+\frac{28\pi}{18}$ degree and $2q\pi+\frac{35\pi}{18}$ degree, we have proved Lemma 4.6. According to Lemma 4.6, there don't exist the truly finite breakaway points of Eq(4.1) In the upper half of $\mathbb{C}\cup\{\infty\}$.

When the degree of "common root locus" are between $2q\pi+\frac{35\pi}{18}$ degree and $2q\pi+\frac{\pi}{18}$ degree, $2q\pi+\frac{8\pi}{18}$ degree and $2q\pi+\frac{10\pi}{18}$ degree, $2q\pi+\frac{17\pi}{18}$ degree and $2q\pi+\frac{19\pi}{18}$ degree, $2q\pi+\frac{26\pi}{18}$  degree and $2q\pi+\frac{28\pi}{18}$ degree, we have proved Lemma 4.7. According to Lemma 4.7, there don't exist the truly finite breakaway points of Eq(4.1) in the upper half of $\mathbb{C}\cup\{\infty\}$.

The $2q\pi$ degree is contained between $2q\pi+\frac{35\pi}{18}$ degree and $2q\pi+\frac{\pi}{18}$ degree.
The $2q\pi+\frac{\pi}{2}$ degree is contained between $2q\pi+\frac{8\pi}{18}$ degree and $2q\pi+\frac{10\pi}{18}$ degree.
The $2q\pi+\pi$ degree is contained between $2q\pi+\frac{17\pi}{18}$ degree and $2q\pi+\frac{19\pi}{18}$ degree.
The $2q\pi+\frac{3\pi}{2}$ degree is contained between $2q\pi+\frac{26\pi}{18}$ degree and $2q\pi+\frac{28\pi}{18}$ degree.

Lemma 4.7 proves the result from $2q\pi$ degree to $2q\pi+\frac{\pi}{18}$ degree, Lemma 4.6 proves the result from $2q\pi+\frac{\pi}{18}$ degree to $2q\pi+\frac{8\pi}{18}$ degree.
Lemma 4.7 proves the result from $2q\pi+\frac{8\pi}{18}$ degree to $2q\pi+\frac{\pi}{2}$ degree and again from $2q\pi+\frac{\pi}{2}$ degree to $2q\pi+\frac{10\pi}{18}$ degree, Lemma 4.6 proves the result from $2q\pi+\frac{10\pi}{18}$ degree to $2q\pi+\frac{17\pi}{18}$ degree.
Lemma 4.7 proves the result from $2q\pi+\frac{17\pi}{18}$ degree to $2q\pi+\pi$ degree and again from $2q\pi+\pi$ degree to $2q\pi+\frac{19\pi}{18}$ degree, Lemma 4.6 proves the result from $2q\pi+\frac{19\pi}{18}$ degree to $2q\pi+\frac{26\pi}{18}$ degree.
Lemma 4.7 proves the result from $2q\pi+\frac{26\pi}{18}$ degree to $2q\pi+\frac{3\pi}{2}$ degree and again from $2q\pi+\frac{3\pi}{2}$ degree to $2q\pi+\frac{28\pi}{18}$ degree, Lemma 4.6 proves the result from $2q\pi+\frac{28\pi}{18}$ degree to $2q\pi+\frac{35\pi}{18}$ degree.
Lemma 4.7 proves the result from $2q\pi+\frac{35\pi}{18}$ degree to $2q\pi$ degree.

Beginning from the first $2q\pi$ degree and ending at the second $2q\pi$ degree, they constitute a whole cycle from $2q\pi$ degree to $2q\pi+2\pi$ degree. So, in two lemmas Lemma 4.6 and Lemma 4.7, the situation of "common root locus " of all of degree are all proved. Therefore, when we sum up results of two lemmas, and according to the definition of breakaway points, we can prove the result of this theorem.
\end{proof}

Because the $FF(s)$ function and $FF^{'}(s)$ function are  both conjugate function, and according to Theorem 4.8, we can obtain the following theorem.

\begin{theorem}
The $FF^{'}(s)$ has no finite zeros in $\Upsilon$.
\end{theorem}
\section{The Results Of Dirichlet Ate Series Or Riemann zeta function}
In the study of Riemann hypothesis, because there exist a function equation, usually, the critical line is as the demarcation line, the generalized complex plane  $\mathbb{C}\cup\{\infty\}$ is divided into two parts. The independent variables of functions concerning Riemann zeta function need to have some differences in the left side of the critical line and the right side of the critical line. In this paper, we give the following rule. In order to express Riemann zeta function or Dirichlet ate series in the right side of the critical line, we need to use two independent variables $s$ and $\overline{s}$. But, the real parts $\sigma$ of two independent variables $s$ and $\overline{s}$ must obtain values in the interval $[1/2,+\infty)$. In order to express  Riemann zeta function or Dirichlet ate series in the left side of the critical line, we need to use two independent variables $1-s$ and $1-\overline{s}$. But, the real parts $\sigma$ of two independent variables $1-s$ and $1-\overline{s}$ must obtain the values in the interval $[1/2,+\infty)$ rather than $(-\infty,1/2]$.

Established the root locus equation of  $\frac{\xi^{'}(s)}{\xi(s)}$.
\begin{equation}
K\frac{\xi^{'}(s)}{\xi(s)}=a+ib
\end{equation}

Dirichlet eta function is defined in $Re(s)>0$ as:
 $\eta(s)=\sum^{+\infty}_{n=1}\frac{(-1)^{n+1}}{n^{s}}$, $(n\in N, s\in\mathbb{C}\cup\{\infty\}$.
$\eta(s)=(1-2^{1-s})\zeta(s)$\cite{Titchmarsh}.
Dirichlet ate series $\eta(s)$ is one of the concrete functions of meromorphic function $W(s)$ which is studied in this paper. So, we can use the results which we have obtained in the previous 3 chapters to solve the problem of distribution of zeros of derivative of Dirichlet eta function.

In the right half critical strip of the upper half $\mathbb{C}\cup\{\infty\}$. The following relationships can be true.

$\zeta(s)=\frac{\eta(s)}{1-2^{1-s}}$, $\xi(s)=\frac{1}{2}s(s-1)\Gamma(\frac{s}{2})\pi^{-\frac{s}{2}}\zeta(s)=\frac{1}{2}s(s-1)\Gamma(\frac{s}{2})\pi^{-\frac{s}{2}}\eta(s)/(1-2^{1-s})$.

$\xi^{'}(s)=\frac{1}{2}(s-1)\Gamma(\frac{s}{2})\pi^{-\frac{s}{2}}\zeta(s)+\frac{1}{2}s\Gamma(\frac{s}{2})\pi^{-\frac{s}{2}}\zeta(s)+
\frac{1}{4}s(s-1)\Gamma^{'}(\frac{s}{2})\pi^{-\frac{s}{2}}\zeta(s)-\frac{1}{4}s(s-1)\Gamma(\frac{s}{2})\pi^{-\frac{s}{2}}\zeta(s)\log{(\pi)}
+\frac{1}{2}s(s-1)\Gamma(\frac{s}{2})\pi^{-\frac{s}{2}}\zeta^{'}(s)$.

Namely, $\xi^{'}(s)=\frac{1}{s}\xi(s)+\frac{1}{s-1}\xi(s)+
\frac{1}{2}\frac{\Gamma^{'}(\frac{s}{2})}{\Gamma(\frac{s}{2})}\xi(s)-\frac{\log(\pi)}{2}\xi(s)
+\frac{\zeta^{'}(s)}{\zeta(s)}\xi(s)$. At two sides of this equation, both divided by xi-function $\xi(s)$, we can obtain: $\frac{\xi^{'}(s)}{\xi(s)}=\frac{1}{s}+\frac{1}{s-1}+
\frac{1}{2}\frac{\Gamma^{'}(\frac{s}{2})}{\Gamma(\frac{s}{2})}-\frac{\log(\pi)}{2}\xi(s)
+\frac{\zeta^{'}(s)}{\zeta(s)}$.

$\frac{d\zeta(s)/ds}{\zeta(s)}=\frac{\eta^{'}(s)/(1-2^{1-s})-(1-2^{1-s})^{'}\eta(s)/(1-2^{1-s})^2}{\eta(s)(1-2^{1-s})}=\frac{d\eta(s)/ds}{\eta(s)}-\frac{2^{1-s}log2}{1-2^{1-s}}$. so, we obtain:
\begin{equation}
\frac{\xi^{'}(s)}{\xi(s)}=\frac{1}{s}+\frac{1}{s-1}+
\frac{1}{2}\psi(\frac{s}{2})-\frac{\log(\pi)}{2}
+\frac{\eta^{'}(s)}{\eta(s)}-\frac{2^{1-s}log2}{1-2^{1-s}}
\end{equation}

Using the complete same proof, we can obtain: in the left half critical strip of the upper half  $\mathbb{C}\cup\{\infty\}$. The following relationships can be true.
\begin{equation}
\frac{\xi^{'}(1-\overline{s})}{\xi(1-\overline{s})}=\frac{1}{1-\overline{s}}-\frac{1}{\overline{s}}+
\frac{1}{2}\psi(\frac{1-\overline{s}}{2})-\frac{\log(\pi)}{2}
+\frac{\eta^{'}(1-\overline{s})}{\eta(1-\overline{s})}-\frac{2^{\overline{s}}log2}{1-2^{\overline{s}}}
\end{equation}

In the reference \cite{Levinson}, have: $\psi(s)=\log{s}-\frac{1}{2s}+R$, $\abs{R}<\frac{1}{10\abs{s}^2}$, $\abs{s}>2$, $\sigma>0$. So, when $t>100$, $1/2\le\sigma\le1$. $\psi(\frac{s}{2})\sim\log{\frac{s}{2}}$, $\frac{1}{2}\log{\frac{s}{2}}=\frac{1}{2}(\log{\sqrt{\frac{(\sigma)^2}{4}+\frac{t^2}{4}}}+i(\arctan{\frac{t}{\sigma}}+2q\pi))$. $\psi(\frac{1-\overline{s}}{2})\sim\log{\frac{1-\overline{s}}{2}}$, $\frac{1}{2}\log{\frac{1-\overline{s}}{2}}=\frac{1}{2}(\log{\sqrt{\frac{(1-\sigma)^2}{4}+\frac{t^2}{4}}}+i(\arctan{\frac{t}{1-\sigma}}+2q\pi))$. So, in the critical strip,

\begin{equation}
Re(\psi(\frac{s}{2}))\approx Re(\psi(\frac{1-\overline{s}}{2}))\approx \log{\frac{t}{2}}
\end{equation}

\begin{lemma}
When $t>100$. On all points of the right sideline $Re(s)=1$ of the critical strip in the upper half  $\mathbb{C}\cup\{\infty\}$, the real parts of the function $\frac{\xi^{'}(s)}{\xi(s)}$ all obtain the positive real values.
\end{lemma}
\begin{proof}
According to Eq(5.4), when the values of the imaginary part satisfy the inequality $t>100$. On all of points of the right side line $Re(s)=1$ of the critical strip in the upper half  $\mathbb{C}\cup\{\infty\}$, $Re(\psi(\frac{1+it}{2}))\approx \log{\frac{1+it}{2}}$, $t$ variable is the positive real values. So, the phase angles of the function $\log{\frac{1+it}{2}}$ are surely from $2q\pi+\frac{3\pi}{2}$ degree to $2q\pi$ degree, and from $2q\pi$ degree to $2q\pi+\frac{\pi}{2}$ degree. Namely, the phase angles of digamma function are surely from $2q\pi+\frac{3\pi}{2}$ degree to $2q\pi$ degree, and from $2q\pi$ degree to $2q\pi+\frac{\pi}{2}$ degree. So, when the values of the imaginary part satisfy the inequality $t>100$, on all of points of the right side line $Re(s)=1$ of the critical strip in the upper half  $\mathbb{C}\cup\{\infty\}$, digamma function obtains the complex values which its real part is the positive real value.

In reference \cite{Titchmarsh}, P112, Theorem 5.17. For $t>A1$, $\sigma\ge1-A1\frac{\log(\log(t))}{\log(t)}$. $\frac{\zeta^{'}(1+it)}{\zeta(1+it)}=O(\frac{\log(t)}{\log(\log(t))})$.
P132, $\sigma\ge1-A2\frac{1}{(\log(\log(t)))^{\frac{3}{4}}(\log(t))^{\frac{3}{4}}}$.

$\frac{\zeta^{'}(1+it)}{\zeta(1+it)}=O((\log(\log(t)))^{\frac{3}{4}}(\log(t))^{\frac{3}{4}})$.
P135, $\sigma\ge1-A3\frac{1}{(\log(\log(t)))^{\frac{1}{3}}(\log(t))^{\frac{2}{3}}}$. $\frac{\zeta^{'}(1+it)}{\zeta(1+it)}<<(\log(\log(t)))^{\frac{1}{3}}(\log(t))^{\frac{2}{3}}$.

According to three different results, obviously, we can give the following result.  $\frac{\zeta^{'}(1+it)}{\zeta(1+it)}<<Re(\psi(\frac{s}{2}))\approx Re(\frac{1}{2}\psi(\frac{s}{2}))\approx \frac{1}{2}Re(\log{\frac{s}{2}})\approx \frac{1}{2}\log{\frac{t}{2}}$.  $\frac{\xi^{'}(1+it)}{\xi(1+it)}=\frac{1}{1+it}+\frac{1}{it}+
\frac{1}{2}\psi(\frac{1+it}{2})-\frac{\log(\pi)}{2}
+\frac{\zeta^{'}(1+it)}{\zeta(1+it)}$, when $t>100$, $\frac{\xi^{'}(1+it)}{\xi(1+it)}\sim  \frac{1}{2}\psi(\frac{1+it}{2}) \sim \frac{1}{2}\log{\frac{1+it}{2}}$. So, the result which we have given the value of the function $\frac{1}{2}\log{\frac{1+it}{2}}$ in front is same suit to the value of the function $\frac{\xi^{'}(1+it)}{\xi(1+it)}$. For the function $\frac{\xi^{'}(1+it)}{\xi(1+it)}$, we have: the phase angles of the function $\frac{\xi^{'}(1+it)}{\xi(1+it)}$ are surely from $2q\pi+\frac{3\pi}{2}$ degree to $2q\pi$ degree, and from $2q\pi$ degree to $2q\pi+\frac{\pi}{2}$ degree. Namely, when $t>100$. On all of points of the right sideline $Re(s)=1$ of the critical strip in the upper half $\mathbb{C}\cup\{\infty\}$, the real parts of the function $\frac{\xi^{'}(1+it)}{\xi(1+it)}$ are all positive real values.
\end{proof}
According to the function equation $\frac{\xi^{'}(s)}{\xi(s)}=-\frac{\xi^{'}(1-s)}{\xi(1-s)}$ and by Lemma 5.1, we can obtain Lemma 5.2.
\begin{lemma}
When $t>100$. On all points of the left sideline $Re(s)=0$ of the critical strip in the lower half  $\mathbb{C}\cup\{\infty\}$, the real parts of the function $\frac{\xi^{'}(s)}{\xi(s)}$ all obtain the negative real values.
\end{lemma}
The xi-function $\xi(s)$ and the derivative $\xi^{'}(s)$ of xi-function are both conjugate functions, so the function $\frac{\xi^{'}(s)}{\xi(s)}$ is also a conjugate function. For the conjugate function, on the conjugate points, the real part values of their function values are equal, and the imaginary part values of their function values are converse. So, by Lemma 5.2, we can have Lemma 5.3:
\begin{lemma}
When $t>100$. On all points of the left sideline $Re(s)=0$ of the critical strip in the upper half  $\mathbb{C}\cup\{\infty\}$, the real parts of the function $\frac{\xi^{'}(s)}{\xi(s)}$ all obtain the negative real values.
\end{lemma}

\begin{lemma}
In the right half critical strip of the upper half $\mathbb{C}\cup\{\infty\}$, the real part and imaginary part of the function $\frac{2^{1-s}log2}{1-2^{1-s}}$ are separately as: $\frac{2^{1-\sigma} \log(2)(\cos(t\log(2))-2^{1-\sigma})}{1-2^{2-\sigma}\cos(t\log(2))+2^{2-2\sigma}} $,
$\frac{-2^{1-\sigma} \log(2)\sin(tlog(2))}{1-2^{2-\sigma}\cos(t\log(2))+2^{2-2\sigma}} $.
\end{lemma}
\begin{proof}
$1-2^{1-s}=1-2^{1-\sigma}\cos(t\log(2))+i2^{1-\sigma} \sin(t\log(2))$, the expression $(1-2^{1-\sigma}\cos(t\log(2))+i2^{1-\sigma} \sin(t\log(2)))$ multiplies its conjugate factor $(1-2^{1-\sigma}\cos(t\log(2))-i2^{1-\sigma} \sin(t\log(2)))$, obtain: $(1-2^{1-\sigma}\cos(t\log(2)))^2+2^{2-2\sigma} \sin^2(t\log(2))=1-2^{1-\sigma}\cos(t\log(2))+2^{2-2\sigma}$. The above expression of derivative $2^{1-\sigma} \log(2)(\cos(t\log(2))-i\sin(t\log(2)))$ multiplies its conjugate factor $(1-2^{1-\sigma}\cos(t\log(2))-i2^{1-\sigma} \sin(t\log(2)))$, obtain: $2^{1-\sigma} \log(2)(\cos(t\log(2))-i\sin(t\log(2)))((1-2^{1-\sigma}\cos(t\log(2))-i2^{1-\sigma} \sin(t\log(2))))=2^{1-\sigma} \log(2)(\cos(t\log(2))-2^{1-\sigma}\cos^{2}(t\log(2))-i2^{1-\sigma}\cos(t\log(2))\sin(t\log(2))-i\sin(t\log(2))+i2^{1-\sigma}\cos(t\log(2))\sin(t\log(2))-2^{1-\sigma}\sin^{2}(t\log(2)))
=2^{1-\sigma} \log(2)(\cos(t\log(2))-2^{1-\sigma}-i\sin(t\log(2)))$.
\end{proof}

Let $Q=2^{1-\sigma} \log(2)(\cos(t\log(2))-2^{1-\sigma})$, $ P=1-2^{2-\sigma}\cos(t\log(2))+2^{2-2\sigma}$. By Lemma 5.4, we can have $Re(\frac{2^{1-s}log2}{1-2^{1-s}})=\frac{Q}{P}$. In this paper, we only need the result when $1/2\le\sigma\le1$.  $ P=1-2^{2-\sigma}\cos(t\log(2))+2^{2-2\sigma}=\sin^{2}(tlog2)+\cos^{2}(tlog2)-2^{2-\sigma}\cos(t\log(2))+2^{2-2\sigma}=\sin^{2}(tlog2)+(\cos(tlog2)-2^{1-\sigma})^2>0$, $\cos(t\log(2))\le 1 \le 2^{1-\sigma}$. So, $Q=2^{1-\sigma} \log(2)(\cos(t\log(2))-2^{1-\sigma})\le 0$. $Re(\frac{2^{1-s}log2}{1-2^{1-s}})=\frac{Q}{P}\le 0$.

\begin{lemma}
In the left half critical strip of the upper half $\mathbb{C}\cup\{\infty\}$, the real part and imaginary part of the function $\frac{2^{\overline{s}}log(2)}{1-2^{\overline{s}}}$ are separately as:
$\frac{2^{\sigma} \log(2)(\cos(t\log(2))-2^{\sigma})}{1-2^{1+\sigma}\cos(t\log(2))+2^{2\sigma}} $. $-\frac{2^{\sigma} \log(2)\sin(t\log(2))}{1-2^{1+\sigma}\cos(t\log(2))+2^{2\sigma}} $.
\end{lemma}
\begin{proof}
$1-2^{\overline{s}}=1-2^{\sigma}\cos(t\log(2))+i2^{\sigma} \sin(t\log(2))$. The expression $(1-2^{\sigma}\cos(t\log(2))+i2^{\sigma} \sin(t\log(2)))$ multiplies its conjugate factor $(1-2^{\sigma}\cos(t\log(2))-i2^{\sigma} \sin(t\log(2)))$, obtain: $(1-2^{\sigma}\cos(t\log(2)))^2+2^{2\sigma} \sin^2(t\log(2))=1-2^{1+\sigma}\cos(t\log(2))+2^{2\sigma}$. The above expression of derivative $2^{\sigma} \log(2)(\cos(t\log(2))-i\sin(t\log(2)))$ multiplies its conjugate factor $(1-2^{\sigma}\cos(t\log(2))-i2^{\sigma} \sin(t\log(2)))$, obtain: $2^{\sigma} \log(2)(\cos(t\log(2))-i\sin(t\log(2)))(1-2^{\sigma}\cos(t\log(2))-i2^{\sigma} \sin(t\log(2)))=2^{\sigma} \log(2)(\cos(t\log(2))-2^{\sigma}\cos^{2}(t\log(2))-i2^{\sigma}\cos(t\log(2))\sin(t\log(2))-i\sin(t\log(2))+i2^{\sigma}\cos(t\log(2))\sin(t\log(2))-2^{\sigma}\sin^{2}(t\log(2)))
=2^{\sigma} \log(2)(\cos(t\log(2))-2^{\sigma}-i\sin(t\log(2)))$.
\end{proof}
Let $Q1=2^{\sigma} \log(2)(\cos(t\log(2))-2^{\sigma})$, $P1=1-2^{1+\sigma}\cos(t\log(2))+2^{2\sigma}$.
By Lemma 5.5, we can have $Re(\frac{2^{\overline{s}}log(2)}{1-2^{\overline{s}}})=\frac{Q1}{P1}$. In this paper, we only need the result when $1/2\le\sigma\le1$.  $ P1=1-2^{1+\sigma}\cos(t\log(2))+2^{2\sigma}=\sin^{2}(tlog2)+\cos^{2}(tlog2)-2^{1+\sigma}\cos(t\log(2))+2^{2\sigma}=\sin^{2}(tlog2)+(\cos(tlog2)-2^{\sigma})^2>0$, $\cos(t\log(2))\le 1 \le 2^{\sigma}$. So, $Q1=2^{\sigma} \log(2)(\cos(t\log(2))-2^{\sigma})\le 0$. $Re(\frac{2^{\overline{s}}log(2)}{1-2^{\overline{s}}})=\frac{Q1}{P1}\le 0$

Dirichlet ate series constitute the following root locus equation (5.5).

\begin{equation}
K\frac{1}{\eta(s)}=a+ib
\end{equation}
After Dirichlet ate series constitute itself root locus equation, its non-trivial zeros surely create "the common root locus" and the demarcation points in $\mathbb{C}\cup\{\infty\}$, so, we can define the set of the demarcation points of (5.5).

Let $\Phi_{t}=\{$ The demarcation points of (5.5) which the imaginary variable $t$ is as a variable in the upper half $\mathbb{C}\cup\{\infty\}\}$.

\begin{lemma}
When $t>100$, on the left side-line of the critical strip, there don't exist the demarcation points in set $\Phi_{t}$.
\end{lemma}
\begin{remark}
when  $t>100$, on the left side line of the critical strip of the upper half $\mathbb{C}\cup\{\infty\}$, there don't exist the zeros of partial derivative of the phase angle of Dirichlet ate series which the imaginary variable $t$ is as a variable.
\end{remark}
\begin{proof}
Assuming that there exists a demarcation point $it_1\in \Phi_{t}$ on the imaginary axis. Theorem 2.6 proves the demarcation points which let the imaginary variable $t$ as the variable are the points which let the real part of the function $\frac{\eta^{'}(s)}{\eta(s)}$ equal to 0. So, the point $it_1$ can let $Re(\frac{\eta^{'}(it_1)}{\eta(it_1)})=0$. The real part of Eq(5.3) is as: $Re(\frac{\xi^{'}(1-\overline{s})}{\xi(1-\overline{s})})=Re(\frac{1}{1-\overline{s}})-Re(\frac{1}{\overline{s}})+
Re(\frac{1}{2}\psi(\frac{1-\overline{s}}{2}))-\frac{\log(\pi)}{2}
+Re(\frac{\eta^{'}(1-\overline{s})}{\eta(1-\overline{s})})-Re(\frac{2^{\overline{s}}log2}{1-2^{\overline{s}}})$.

Especially, for the demarcation point $it_1\in \Phi_{t}$ on the imaginary axis, we have: $Re(\frac{\xi^{'}(it_1)}{\xi(it_1)})=Re(\frac{1}{it_1})+Re(\frac{1}{it_1-1})+
Re(\frac{1}{2}\psi(\frac{it_1}{2}))-\frac{\log(\pi)}{2}
-\frac{Q1}{P1}$.

$Re(\psi(\frac{1-\overline{s}}{2}))\approx\log{\frac{t}{2}}$.  $Re(\frac{2^{\overline{s}}log2}{1-2^{\overline{s}}})=\frac{Q1}{P1}<0$, so, when $t>100$, on the imaginary axis of the upper half $\mathbb{C}\cup\{\infty\}\}$, the real part of the function values of the function $\frac{\xi^{'}(s)}{\xi(s)}$ on the demarcation point $it_1\in \Phi_{t}$ surely obtains positive real value.

Lemma 5.2 proves, on the imaginary axis of the upper half $\mathbb{C}\cup\{\infty\}$, the real part values of the function $\frac{\xi^{'}(s)}{\xi(s)}$ obtain the negative real values.

The contradiction is generated, so, the original assuming cannot be true. On the left side line of the critical strip, there don't exist the demarcation points in set $\Phi_{t}$. Namely, there don't exist the zeros of partial derivative of the phase angle of Dirichlet ate series on the left side line of the critical strip of the upper half $\mathbb{C}\cup\{\infty\}$.
\end{proof}

\begin{lemma}
$Re{(\zeta^{'}(1/2+it)/\zeta(1/2+it))}=\frac{log \pi}{2}-\frac{1}{4}\psi(\frac{1/2+it}{2})-\frac{1}{4}\psi(\frac{1/2-it}{2})$
\end{lemma}
\begin{proof}
According to the relationship formula $\xi(s)=\frac{1}{2}s(s-1)\Gamma(\frac{s}{2}){\pi}^{-\frac{s}{2}}\zeta(s)$, we can obtain:  $\zeta^{'}(s)/\zeta(s)+\zeta^{'}(1-s)/\zeta(1-s)=log \pi-\frac{1}{2}\psi(\frac{s}{2})-\frac{1}{2}\psi(\frac{1-s}{2})$. In reference\cite{Titchmarsh}, P287,it is Eq(10.28.3).
In specially, let $s=1/2+it$, we can have: $\zeta^{'}(1/2+it)/\zeta(1/2+it)+\zeta^{'}(1/2-it)/\zeta(1/2-it)=log \pi-\frac{1}{2}\psi(\frac{1/2+it}{2})-\frac{1}{2}\psi(\frac{1/2-it}{2})$. $\zeta^{'}(1/2+it)/\zeta(1/2+it)$
 and $\zeta^{'}(1/2-it)/\zeta(1/2-it)$ are the function values which the function $\zeta^{'}(s)/\zeta(s)$ is on two conjugate points on the critical line. The function $\zeta^{'}(s)/\zeta(s)$ is a conjugate function. So, the function $\zeta^{'}(1/2+it)/\zeta(1/2+it)+\zeta^{'}(1/2-it)/\zeta(1/2-it)$ is a real value function.
And $Re{(\zeta^{'}(1/2+it)/\zeta(1/2+it))}=\frac{1}{2}(\zeta^{'}(1/2+it)/\zeta(1/2+it)+\zeta^{'}(1/2-it)/\zeta(1/2-it))=\frac{log \pi}{2}-\frac{1}{4}\psi(\frac{1/2+it}{2})-\frac{1}{4}\psi(\frac{1/2-it}{2})$.
\end{proof}

\begin{lemma}
$Re{(\zeta^{'}(1/2+it)/\zeta(1/2+it))} \approx \frac{log \pi}{2}-\frac{1}{4}\log(\frac{1/4+t^2}{4})-\frac{1}{1+4t^2}$
\end{lemma}
\begin{proof}
According to Eq(5.4), here, we only study top two term of psi-function.
 $Re{(\zeta^{'}(1/2+it)/\zeta(1/2+it))}\approx \frac{log \pi}{2}-\frac{1}{4}\log(\frac{1/2+it}{2})-\frac{1}{4}\log(\frac{1/2-it}{2})-\frac{1}{4}\frac{1}{1/2+it}-\frac{1}{4}\frac{1}{1/2-it}=\frac{log \pi}{2}-\frac{1}{4}\log(\frac{1/4+t^2}{4})-\frac{1}{1+4t^2}$.
\end{proof}
$Re\{\frac{\zeta^{'}(\frac{1}{2}+it)}{\zeta(\frac{1}{2}+it)}\}<0$. According to the definition of argument\cite{Stein}, the argument at zeros contains all the values. So, the argument at zeros can be an arbitrary values.
When we consider the monotonicity of the values of argument, we should not think that the argument at zeros is equal to $0$. By Theorem 2.5, we have Theorem 5.10.
\begin{theorem}
On the critical line  of the upper half $\mathbb{C}\cup\{\infty\}$, the values of argument of Riemann zeta function concerning the imaginary variable $t$ strictly and monotonically decrease.
\end{theorem}
$\frac{Q}{P}<0$. $Re\{\frac{\eta^{'}(\frac{1}{2}+it)}{\eta(\frac{1}{2}+it)}\}=Re\{\frac{\zeta^{'}(\frac{1}{2}+it)}{\zeta(\frac{1}{2}+it)}\}+\frac{Q}{P}<0$. Based on the same reason, we have Theorem 5.11.
\begin{theorem}
On the critical line  of the upper half $\mathbb{C}\cup\{\infty\}$, the values of argument of Dirichlet ate series concerning the imaginary variable $t$ strictly and monotonically decrease.
\end{theorem}
For (5.5), the equivalent result of Theorem 5.11 is Theorem 5.12.
\begin{theorem}
On the critical line  of the upper half $\mathbb{C}\cup\{\infty\}$, the degree numbers of the root locus of (5.3) concerning the imaginary variable $t$ strictly and monotonically decrease.
\end{theorem}
$\frac{\partial{\arg{\eta(s)}}}{\partial{t}}=Re\{\frac{\eta^{'}(\frac{1}{2}+it)}{\eta(\frac{1}{2}+it)}\}=Re\{\frac{\zeta^{'}(\frac{1}{2}+it)}{\zeta(\frac{1}{2}+it)}\}+\frac{Q}{P}<0$. So, we have Theorem 5.13.

\begin{theorem}
When $t>100$. On the critical line of the upper half $\mathbb{C}\cup\{\infty\}$, there don't exist the demarcation points in set $\Phi_{t}$.
\end{theorem}
\begin{remark}
when  $t>100$, there don't exist the zeros of partial derivative of the phase angle of Dirichlet ate series which the imaginary variable $t$ is as a variable on the critical line of the upper half $\mathbb{C}\cup\{\infty\}$.
\end{remark}

\section{The proof of Riemann Hypothesis}
\begin{lemma}
In the right half $\mathbb{C}\cup\{\infty\}\}$, Riemann zeta function and Dirichlet ate series don't have any poles. (5.5) has no zeros in the right half $\mathbb{C}\cup\{\infty\}\}$. In the left half $\mathbb{C}\cup\{\infty\}\}$, Dirichlet ate series  has poles.
\end{lemma}
\begin{proof}
$\zeta(s)=\frac{1}{1-2^{1-s}}\sum^{+\infty}_{n=1} \frac{(-1)^{n+1}}{n^s}$.
When $1>\sigma>0$ and $t=+\infty$, $\zeta(s)$ and $\eta(s)=\sum^{+\infty}_{n=1} \frac{(-1)^{n+1}}{n^s}$ both obtain the non-infinity finite values. Obviously, according to the expressions of two functions, on all of the infinite points inside the critical strip, two functions both obtain the finite values. Namely, on all of the infinite points inside the critical strip, two functions both don't have the poles.

On the finite position in the critical strip, the two function also obtain the finite values.

Let $s=-m$, $m$ is a positive integer. $\eta(-m)=\sum^{+\infty}_{n=1} (-1)^{n+1}n^m$. Let $n=2l-1$ and $n=2l$. $\eta(-m)=\sum^{+\infty}_{l=0} (2l-1)^m-(2l)^m$.

$(2l-1)^m-(2l)^m=(2l-1)^m(1-(1+\frac{1}{2l-1})^m)<0$. Therefore, $\eta(-m)$ is a  negative infinity. At $s=-m$ points, Dirichlet ate series gets the negative infinity.  In the left half $\mathbb{C}\cup\{\infty\}\}$, Dirichlet ate series  has poles.
\end{proof}

\begin{lemma}
Non-trivial zeros and the other zeros of the Dirichlet eta series produce "common root loci"; these "common root loci" cannot produce breakaway points in the left half complex plane. The root loci of the same degree emitted by non-trivial zeros and by the other zeros of the Dirichlet eta series cannot be adjacent in the left half complex plane. The root loci of adjacent degree emitted by non-trivial zeros and by the other zeros of the Dirichlet eta series cannot be adjacent in the left half complex plane.
\end{lemma}
\begin{proof}
If non-trivial zeros and the other zeros of the Dirichlet eta series produce "common root loci", and these "common root loci" produce a breakaway point in the left half complex plane, then at least two "common root loci" are emitted from the non-trivial zeros and the other zeros of the Dirichlet eta series, and the initial segments of at least two of these "common root loci" must cross the left boundary line of the critical strip; the initial segments of at least two "common root loci" come out of the critical strip and enter the left half complex plane. If the initial segments of at least two "common root loci" intersect at a finite position of the left half complex plane, then, by Theorem 3.15, the terminal segment of at least one "common root locus" must still leave the finite breakaway point; from the left half complex plane, the terminal segment of one "common root locus" crosses the left boundary line of the critical strip and enters the critical strip. If the initial segments of at least two "common root loci" intersect at a position at infinity, then, by Theorem 3.33, a root locus cannot intersect itself, so the terminal segment of at least one "common root locus" leaves the breakaway point at infinity; from the left half complex plane, the terminal segment of one "common root locus" crosses the left boundary line of the critical strip and enters the critical strip.

According to the results obtained in the preceding proofs, on the left boundary line of the critical strip there must exist two points of the same root locus. Two points of the same root locus have the same phase angle $\varphi(\sigma,t)$.

$\varphi(0,t)$ on the left boundary line of the critical strip is a real-valued continuous function of the single real variable $t$. Therefore, by Rolle's theorem, between the two intersection points of the same root locus with the left boundary line of the critical strip there must exist at least one demarcation point of the imaginary part variable of the Dirichlet eta series. However, we have proved in Lemma 5.6 that on the left boundary line of the critical strip there does not exist a demarcation point of the imaginary part variable of the Dirichlet eta series. Thus a contradiction arises. Therefore, the hypothesis that non-trivial zeros and the other zeros of the Dirichlet eta series produce "common root loci", and that these "common root loci" produce a breakaway point in the left half complex plane, cannot hold. This proves: the "common root loci" produced by non-trivial zeros and the other zeros of the Dirichlet eta series cannot produce breakaway points in the left half complex plane.

On the left boundary line of the critical strip there does not exist a demarcation point of the imaginary part variable of the Dirichlet eta series.

The root loci of the same degree emitted by non-trivial zeros and by the other zeros of the Dirichlet eta series are adjacent in the left half complex plane. This is one possible situation. The root loci of adjacent degree emitted by non-trivial zeros and by the other zeros of the Dirichlet eta series are adjacent in the left half complex plane. This is also one possible situation. Repeating the preceding proof, one can prove that neither of these two possible situations can exist.
\end{proof}
\begin{lemma}
$\frac{1}{2}< \sigma\le 1$,  $s_{\infty}=\sigma\pm i\infty$, $\frac{\xi^{'}(s_{\infty})}{\xi(s_{\infty})}=0$.
\end{lemma}
\begin{proof}
$\xi(s)=e^{A+Bs}\prod_{\rho} (1-\frac{s}{\rho}e^{\frac{s}{\rho}}$. $B=-1-\frac{\nu}{2}+\frac{1}{2}\log(4\pi)=-\sum_{\rho} \frac{1}{\rho}$. $\frac{\xi^{'}(s)}{\xi(s)}=B+\sum_{\rho} \frac{1}{s-\rho}+\frac{1}{\rho}$.

Let $s_{\infty}=\sigma\pm i\infty$, $\frac{\xi^{'}(s_{\infty})}{\xi(s_{\infty})}=B+\sum_{\rho} \frac{1}{\sigma\pm i\infty-\rho}+\frac{1}{\rho}=0$.
\end{proof}

\begin{lemma}
$\frac{1}{2}\le \sigma\le1$,  $\abs{\frac{2^{s}log2}{1-2^{s}}}<46.08\log2$.
\end{lemma}
\begin{proof}
Let $s=\sigma+it$, $\frac{2^{s}log2}{1-2^{s}}=\frac{2^{\sigma+it}log2}{1-2^{\sigma+it}}=\frac{(2^{\sigma}\cos(t)+i2^{\sigma}\sin(t))\log2}{1-2^{\sigma}\cos(t)-i2^{\sigma}\sin(t)}$.

$=\frac{(2^{\sigma}\cos(t)+i2^{\sigma}\sin(t))\log2(1-2^{\sigma}\cos(t)+i2^{\sigma}\sin(t))}{(1-2^{\sigma}\cos(t))^2+2^{2\sigma}\sin^2(t)}$

$=\frac{(2^{\sigma}\cos(t)+i2^{\sigma}\sin(t)-2^{2\sigma}\cos^2(t)-i2^{2\sigma}\sin(t)\cos(t)-2^{2\sigma}\sin^2(t)+i2^{2\sigma}\sin(t)\cos(t))\log2}{1-2^{\sigma+1}\cos(t)+2^{2\sigma}\cos^2(t)+2^{2\sigma}\sin^2(t)}$

$=\frac{(2^{\sigma}\cos(t)+i2^{\sigma}\sin(t)-2^{2\sigma})\log2}{1-2^{\sigma+1}\cos(t)+2^{2\sigma}}$.

$\frac{1}{2}\le \sigma\le1$,

$\frac{1}{1-2^{\sigma+1}\cos(t)+2^{2\sigma}}<\frac{1}{1-2^{\sigma+1}+2^{2\sigma}}=\frac{1}{(2^{\sigma}-1)^2}$. When $\sigma=1/2$, $\frac{1}{(2^{\sigma}-1)^2}$ gets the maximum value. Its maximum value is 5.76.

$\abs{(2^{\sigma}\cos(t)+i2^{\sigma}\sin(t)-2^{2\sigma})}\log2=2^{\sigma}\log2\abs{\cos(t)+i\sin(t)-2^{\sigma}}<2\log2(1+1+2)=8\log2$.

Therefore, $\abs{\frac{2^{s}log2}{1-2^{s}}}<46.08\log2$.
\end{proof}

\begin{lemma}
The "common root loci" produced by non-trivial zeros and the other zeros of the Dirichlet eta series cannot produce breakaway points at the points at infinity in the left half of the critical strip.
\end{lemma}
\begin{proof}
$\frac{\zeta^{'}(s)}{\zeta(s)}=\frac{\eta^{'}(s)}{\eta(s)}-\frac{2^{1-s}log2}{1-2^{1-s}}$. $\frac{\eta^{'}(s)}{\eta(s)}=-\frac{1}{s}-\frac{1}{s-1}-
\frac{1}{2}\psi(\frac{s}{2})+\frac{\log(\pi)}{2}
+\frac{\xi^{'}(s)}{\xi(s)}+\frac{2^{1-s}log2}{1-2^{1-s}}$.  $\frac{\zeta^{'}(s)}{\zeta(s)}+\frac{\zeta^{'}(1-s)}{\zeta(1-s)}=log \pi-\frac{1}{2}\psi(\frac{s}{2})-\frac{1}{2}\psi(\frac{1-s}{2})$.

$\zeta(1-s)=\frac{\eta(1-s)}{1-2^{s}}$, $-\zeta^{'}(1-s)=\frac{-\eta^{'}(1-s)(1-2^{s})+\eta(s)2^s\log2}{(1-2^{s})^2}=-\frac{\eta^{'}(1-s)}{1-2^{s}}+\frac{\eta(s)2^s\log2}{(1-2^{s})^2}$.

$\frac{\zeta^{'}(1-s)}{\zeta(1-s)}=\frac{\eta^{'}(1-s)}{\eta(1-s)}-\frac{2^s\log2}{1-2^{s}}$.

$\frac{\eta^{'}(s)}{\eta(s)}-\frac{2^{1-s}log2}{1-2^{1-s}}+\frac{\eta^{'}(1-s)}{\eta(1-s)}-\frac{2^s\log2}{1-2^{s}}=log \pi-\frac{1}{2}\psi(\frac{s}{2})-\frac{1}{2}\psi(\frac{1-s}{2})$.

$-\frac{1}{s}-\frac{1}{s-1}-
\frac{1}{2}\psi(\frac{s}{2})+\frac{\log(\pi)}{2}
+\frac{\xi^{'}(s)}{\xi(s)}+\frac{2^{1-s}log2}{1-2^{1-s}}-\frac{2^{1-s}log2}{1-2^{1-s}}+\frac{\eta^{'}(1-s)}{\eta(1-s)}-\frac{2^s\log2}{1-2^{s}}=log \pi-\frac{1}{2}\psi(\frac{s}{2})-\frac{1}{2}\psi(\frac{1-s}{2})$.

$-\frac{1}{s}-\frac{1}{s-1}
+\frac{\xi^{'}(s)}{\xi(s)}+\frac{\eta^{'}(1-s)}{\eta(1-s)}-\frac{2^s\log2}{1-2^{s}}=\frac{\log(\pi)}{2}-\frac{1}{2}\psi(\frac{1-s}{2})$.

If non-trivial zeros and the other zeros of the Dirichlet eta series produce "common root loci", and these "common root loci" produce a breakaway point at the points at infinity in the left half of the critical strip, then there exists a point at infinity $\sigma_0\pm\infty$ such that $\frac{\eta^{'}(1-\sigma_0\pm \infty)}{\eta(1-\sigma_0\pm \infty)}=0$. $\frac{1}{2}< \sigma_0<1$,

$\frac{\xi^{'}(\sigma_0\pm \infty)}{\xi(\sigma_0\pm \infty)}=\frac{1}{\sigma_0\pm \infty}+\frac{1}{\sigma_0\pm \infty-1}+\frac{log \pi}{2}-\frac{1}{2}\psi(\frac{1-\sigma_0\pm \infty}{2})+\frac{2^{\sigma_0\pm \infty}log2}{1-2^{\sigma_0\pm \infty}}$.

By Lemma 6.4, $\abs{\frac{2^{\sigma_0\pm \infty}}log2}{1-2^{\sigma_0\pm \infty}}<46.08\log2$.  $\frac{\xi^{'}(\sigma_0\pm \infty)}{\xi(\sigma_0\pm \infty)}=0$. $\psi(\frac{1-\sigma_0\pm \infty}{2})=\infty$. Therefore, the equation $\frac{\xi^{'}(\sigma_0\pm \infty)}{\xi(\sigma_0\pm \infty)}=\frac{1}{\sigma_0\pm \infty}+\frac{1}{\sigma_0\pm \infty-1}+\frac{log \pi}{2}-\frac{1}{2}\psi(\frac{1-\sigma_0\pm \infty}{2})+\frac{2^{\sigma_0\pm \infty}log2}{1-2^{\sigma_0\pm \infty}}$ doesn't hold. Hence, the assumption that non-trivial zeros and the other zeros of the Dirichlet eta series produce "common root loci", and that these "common root loci" produce a breakaway point at the points at infinity in the left half of the critical strip, cannot hold. The result that the "common root loci" produced by non-trivial zeros and the other zeros of the Dirichlet eta series cannot produce breakaway points at the points at infinity in the left half of the critical strip holds.
\end{proof}

The root loci of the same degree emitted by non-trivial zeros and by the other zeros of the Dirichlet eta series may be adjacent at positions at infinity in the left half of the critical strip. The root loci of adjacent degree emitted by non-trivial zeros and by the other zeros of the Dirichlet eta series may be adjacent at positions at infinity in the left half of the critical strip. We need to prove that these two possible situations certainly cannot exist.

When these root loci of the same degree emitted by non-trivial zeros and by the other zeros of the Dirichlet eta series are adjacent at positions at infinity in the left half of the critical strip, the root locus from the non-trivial zero or from the other zero of the Dirichlet eta series to the adjacent point at infinity is taken as the root locus of the initial segment. The root locus from the adjacent point at infinity to the terminal point of the root locus is taken as the root locus of the terminal segment. Since the non-trivial zeros and the other zeros of the Dirichlet eta series all lie in the critical strip, and their root loci are adjacent at positions at infinity in the left half of the critical strip. According to the positional relations of the root loci, and according to the condition for two root loci of the same degree to be adjacent, the root loci of the terminal segments must all lie between the two adjacent root loci of the initial segments; starting from the adjacent points, the root loci of the terminal segments continue to extend. They must go around one of the two non-trivial zeros or one of the other zeros of the Dirichlet eta series, and extend all the way to the terminal point in the left half complex plane.

Therefore, according to the positional relations of the root loci, near the points at infinity in the left half of the critical strip, if straight lines parallel to the real axis are drawn, then on these straight lines there must exist at least four intersection points of the root loci of the same degree. These four intersection points all have the same phase angle $\varphi(\sigma,t)$.

On these straight lines parallel to the real axis, $\varphi(\sigma,t_1)$ is a real-valued continuous function of the single real variable $\sigma$. Therefore, by Rolle's theorem, between these four intersection points there must exist at least 3 demarcation points of the real part variable of the Dirichlet eta series. Since these straight lines are arbitrarily close to the adjacent points at infinity in the left half of the critical strip, in the neighbourhood of the two adjacent points at infinity in the left half of the critical strip, inside any small adjacent domain of the two adjacent points, when the real part $\sigma$ acts as a variable, there surely exists at least one demarcation point of the real part variable of the root locus equation (5.5).

Near the two adjacent points at infinity in the left half of the critical strip, if straight lines parallel to the imaginary axis are drawn, then, since the root loci of the two terminal segments of the same degree must go around one of the two non-trivial zeros or one of the other zeros of the Dirichlet eta series, on the straight lines parallel to the imaginary axis there must exist at least two intersection points of the root loci of the same degree. These two intersection points have the same phase angle $\varphi(\sigma,t)$.

On these straight lines parallel to the imaginary axis, $\varphi(\sigma_1,t)$ is a real-valued continuous function of the single real variable $t$. Therefore, by Rolle's theorem, between these two intersection points there must exist at least 1 demarcation point of the imaginary part variable of the Dirichlet eta series. Since these straight lines are arbitrarily close to the adjacent points at infinity in the left half of the critical strip, in the neighbourhood of the two adjacent points at infinity in the left half of the critical strip, inside any small adjacent domain of the two adjacent points, when the imaginary part $t$ acts as a variable, there surely exists at least one demarcation point of the imaginary part variable of the root locus equation (5.5).

Thus it is proved: at positions at infinity in the left half of the critical strip, the root loci of the same degree emitted by non-trivial zeros and by the other zeros of the Dirichlet eta series are adjacent, and the two adjacent points become the breakaway point of the root loci. At positions at infinity in the left half of the critical strip, the situation in which root loci of the same degree are adjacent cannot exist; only the situation in which root loci of the same degree intersect at positions at infinity in the left half of the critical strip can exist.

Completely repeating the preceding proof, one can prove: at positions at infinity in the left half of the critical strip, the situation in which root loci of adjacent degree are adjacent cannot exist; only the situation in which root loci of the same degree intersect at positions at infinity in the left half of the critical strip can exist.

\begin{lemma}
If non-trivial zeros exist outside the critical line, then the "common root loci" produced by the non-trivial zeros in the left half of the critical strip and by the other zeros of the Dirichlet eta series cannot produce breakaway points at the points at infinity on the right side of the critical strip. The root loci of the same degree emitted by the non-trivial zeros in the left half of the critical strip and by the other zeros of the Dirichlet eta series cannot be adjacent at the points at infinity on the right side of the critical strip. The root loci of adjacent degree emitted by the non-trivial zeros in the left half of the critical strip and by the other zeros of the Dirichlet eta series cannot be adjacent at the points at infinity on the right side of the critical strip.
\end{lemma}
\begin{proof}
If non-trivial zeros exist outside the critical line, and the non-trivial zeros in the left half of the critical strip emit root loci towards the points at infinity of the right half complex plane, and the root loci intersect at the points at infinity of the right half complex plane. Let the non-trivial zero of the Dirichlet eta series on the left side of the critical line be $cp$. $cp$ emits a root locus $RL_{cp}$ towards the points at infinity of the right half complex plane. If the root locus $RL_{cp}$ reaches the points at infinity of the right half complex plane, $RL_{cp}$ must cross the critical line, from the left side of the critical line to the right side of the critical line. None of the points at infinity of the right half complex plane is a zero of (5.5). Moreover, in the whole right half complex plane, the Dirichlet eta series cannot have poles, and no zero of (5.5) can exist in the whole right half complex plane. After $RL_{cp}$ reaches the points at infinity of the right half complex plane, $RL_{cp}$ must return from the points at infinity of the right half complex plane to the left side of the critical line. $RL_{cp}$ must again cross the critical line, from the right side of the critical line to the left side of the critical line. Thus there arise two points of the same root locus on the critical line. Two points of the same root locus have the same phase angle $\varphi(\sigma,t)$.

$\varphi(\sigma,1/2)$ on the critical line is a real-valued continuous function of the single real variable $\sigma$. Therefore, by Rolle's theorem, between the two intersection points of $RL_{cp}$ with the critical line there must exist at least one demarcation point of the imaginary part variable of the Dirichlet eta series. However, we have proved in Theorem 5.13 that on the critical line there does not exist a demarcation point of the imaginary part variable of the Dirichlet eta series. Thus a contradiction arises. Therefore, the assumption that non-trivial zeros exist outside the critical line, and that the non-trivial zeros in the left half of the critical strip emit root loci towards the points at infinity of the right half complex plane, and that the root loci intersect at the points at infinity of the right half complex plane, cannot hold. This proves: if non-trivial zeros exist outside the critical line, then the "common root loci" produced by the non-trivial zeros in the left half of the critical strip and by the other zeros of the Dirichlet eta series cannot produce breakaway points at the points at infinity on the right side of the critical strip.

On the critical line there does not exist a demarcation point of the imaginary part variable of the Dirichlet eta series. The root loci of the same degree emitted by the non-trivial zeros in the left half of the critical strip and by the other zeros of the Dirichlet eta series are adjacent at the points at infinity of the right half complex plane. This is also one possible situation. The root loci of adjacent degree emitted by the non-trivial zeros in the left half of the critical strip and by the other zeros of the Dirichlet eta series are adjacent at the points at infinity of the right half complex plane. This is also one possible situation. Repeating the preceding proof, one can prove that neither of these two possible situations can exist.
\end{proof}

\begin{theorem}
All non-trivial zeros of the Riemann zeta function lie on the critical line.
\end{theorem}
\begin{proof}
At finite positions of the complex plane, if the root loci originating from several distinct poles of (5.5) come into contact, the root loci of the same degree starting from several poles of (5.5) can only intersect, and they produce the breakaway points of the intersecting root loci. This is the result already proved in Section 3.

Dirichlet ate series  $\eta(s)=\sum^{+\infty}_{n=1} \frac{(-1)^{n+1}}{n^s}$ is a special series of $FF(s)=\sum^{+\infty}_{n=1}\frac{a_{n}(n)}{(b_{n}(n))^{s}}$. We have proved Theorem 4.9. By Theorem 4.9, we have: in the finite position in the right half complex plane, there does not exist the critical point of Dirichlet ate series. By Theorem 3.9, in the finite position in the complex plane, the breakaway point of the root locus of (5.5) is the critical point of Dirichlet ate series. Therefore, in the finite position in the right half complex plane, there does not exist the breakaway point of the root locus of (5.5).

According to the two classes of results already proved, if non-trivial zeros exist outside the critical line, then the root loci of the non-trivial zeros outside the critical line and the root loci of the other zeros of the Dirichlet eta series cannot be adjacent at finite positions of the right half complex plane. These root loci also cannot produce finite breakaway points at finite positions of the right half complex plane.

Lemma 6.2 proves: the "common root loci" produced by non-trivial zeros and the other zeros of the Dirichlet eta series cannot produce breakaway points in the left half complex plane. The root loci of the same degree emitted by non-trivial zeros and by the other zeros of the Dirichlet eta series cannot be adjacent in the left half complex plane. The root loci of adjacent degree emitted by non-trivial zeros and by the other zeros of the Dirichlet eta series cannot be adjacent in the left half complex plane of the critical strip.

Lemma 6.5 proves: the "common root loci" produced by non-trivial zeros and the other zeros of the Dirichlet eta series cannot produce breakaway points at the points at infinity in the left half of the critical strip. We have also proved: at positions at infinity in the left half of the critical strip, the situation in which root loci of the same degree are adjacent cannot exist; only the situation in which root loci of the same degree intersect at positions at infinity in the left half of the critical strip can exist. At positions at infinity in the left half of the critical strip, the situation in which root loci of adjacent degree are adjacent cannot exist; only the situation in which root loci of the same degree intersect at positions at infinity in the left half of the critical strip can exist. Hence it is proved: the "common root loci" produced by non-trivial zeros and the other zeros of the Dirichlet eta series cannot produce breakaway points at the points at infinity in the left half of the critical strip. At positions at infinity in the left half of the critical strip, the situation in which root loci of the same degree are adjacent cannot exist. At positions at infinity in the left half of the critical strip, the situation in which root loci of adjacent degree are adjacent cannot exist.

Lemma 6.6 proves: if non-trivial zeros exist outside the critical line, then the "common root loci" produced by the non-trivial zeros in the left half of the critical strip and by the other zeros of the Dirichlet eta series cannot produce breakaway points at the points at infinity on the right side of the critical strip. The root loci of the same degree emitted by the non-trivial zeros in the left half of the critical strip and by the other zeros of the Dirichlet eta series cannot be adjacent at the points at infinity on the right side of the critical strip. The root loci of adjacent degree emitted by the non-trivial zeros in the left half of the critical strip and by the other zeros of the Dirichlet eta series cannot be adjacent at the points at infinity on the right side of the critical strip.

Therefore, these three lemmas prove: if non-trivial zeros exist outside the critical line, then the root loci of the non-trivial zeros in the left half of the critical strip cannot be adjacent to the root loci of the other zeros of the Dirichlet eta series at any position at infinity. These root loci also cannot produce breakaway points.

The preceding work proves: at finite positions, the root loci of the non-trivial zeros outside the critical line cannot be adjacent to the root loci of the other zeros of the Dirichlet eta series. These root loci also cannot produce finite breakaway points. At all positions at infinity, the root loci of the non-trivial zeros in the left half of the critical strip cannot be adjacent to the root loci of the other zeros of the Dirichlet eta series. These root loci also cannot produce breakaway points at infinity. But if non-trivial zeros exist outside the critical line, these non-trivial zeros in the left half of the critical strip must emit root loci. The root loci of these non-trivial zeros in the left half of the critical strip must come into contact with the root loci of the other zeros of the Dirichlet eta series. The result of the root loci coming into contact is that the root loci are adjacent or intersect, and intersecting root loci produce breakaway points. Either the root loci of these non-trivial zeros in the left half of the critical strip must be in the situation of being adjacent to the root loci of the other zeros of the Dirichlet eta series, or the root loci of these non-trivial zeros in the left half of the critical strip must intersect the root loci of the other zeros of the Dirichlet eta series and produce breakaway points. One of the two cases must exist. Thus we obtain two classes of mutually contradictory results.

Therefore, this proves that the hypothesis that non-trivial zeros exist outside the critical line cannot hold. Non-trivial zeros cannot exist outside the critical line.
\end{proof}

\bibliographystyle{unsrt}

\end{document}